\documentclass[12pt]{amsart}
\usepackage{amsmath}
\usepackage{amssymb}
\usepackage{color}
\usepackage{soul}
\usepackage{ifthen}

\newtheorem{propo}{Proposition}[section]
\newtheorem{corol}[propo]{Corollary}
\newtheorem{theor}[propo]{Theorem}
\newtheorem{lemma}[propo]{Lemma}
\theoremstyle{definition}
\newtheorem{defin}[propo]{Definition}
\newtheorem{examp}[propo]{Example}

\newtheorem{setti}[propo]{Setting}
\theoremstyle{remark}
\newtheorem{remar}[propo]{Remark}

\newtheorem{conve}[propo]{Convention}

\numberwithin{equation}{section}

\newcommand{\ad }{\mathrm{ad}\,}
\newcommand{\al }{\alpha }

\newcommand{\actl }{\boldsymbol{\cdot }}

\newcommand{\antip }{S}
\newcommand{\Aut }{\mathrm{Aut}}

\newcommand{\bq}{\bar{q}}
\newcommand{\brcopr }{\underline{\varDelta }}

\newcommand{\cC }{\mathcal{C}}

\newcommand{\cG }{\mathcal{G}}
\newcommand{\cI }{\mathcal{I}}
\newcommand{\cJ }{\mathcal{J}}

\newcommand{\cX }{\mathcal{X}}

\newcommand{\coal }{\delta }
\newcommand{\copr }{\varDelta }
\newcommand{\coun }{\varepsilon }
\newcommand{\cS }{\mathcal{S}}
\newcommand{\cU }{\mathcal{U}}
\newcommand{\cV }{\mathcal{V}}

\newcommand{\derK }{\partial ^K}
\newcommand{\derL }{\partial ^L}
\newcommand{\End }{\mathrm{End}}
\newcommand{\fie }{\Bbbk }
\newcommand{\fienz }{\fie ^\times }

\newcommand{\hght }[1]{h^{#1}}
\newcommand{\Hom }{\mathrm{Hom}}
\newcommand{\id}{\operatorname{id}}

\newcommand{\lag }{\mathfrak{g}}

\newcommand{\lcoaS }{\delta _{\mathrm{l}}}

\newcommand{\LT }{T}
\newcommand{\lYDcat}[1]{{}^{#1}_{#1}\mathcal{YD}}

\newcommand{\mul }{\mathrm{m}}
\newcommand{\NA }[1]{\mathfrak{B}(#1)}

\newcommand{\Ndb }{\alpha }
\newcommand{\ndN }{\mathbb{N}}
\newcommand{\ndQ }{\mathbb{Q}}

\newcommand{\ndZ }{\mathbb{Z}}
\newcommand{\Nich }{\mathfrak{B}}
\newcommand{\op }{^\mathrm{op}}
\newcommand{\cop }{^\mathrm{cop}}
\newcommand{\Ob }{\mathrm{Ob}}

\newcommand{\ot }{\otimes }

\newcommand{\pr }{\mathrm{pr}}

\newcommand{\qchoose }[3]{\binom{#1}{#2}_{\!#3}}
\newcommand{\qfact }[2]{(#1)^!_{#2}}
\newcommand{\qnum }[2]{(#1)_{#2}}

\newcommand{\rcoaS }{\delta _{\mathrm{r}}}

\newcommand{\rroots }[1]{(R^{#1})^\mathrm{re}}

\newcommand{\rrsC }{\mathcal{R}^{\mathrm{re}}}
\newcommand{\rsC }{\mathcal{R}}
\newcommand{\s }{\sigma}

\newcommand{\sHp }{\eta }

\newcommand{\tJ }{\tilde{\cJ }}

\newcommand{\ula }{\underline{a}}
\newcommand{\ulb }{\underline{b}}
\newcommand{\ulc }{\underline{c}}
\newcommand{\ullam }{\underline{\lambda }}
\newcommand{\ulE }{\underline{E}}
\newcommand{\ulF }{\underline{F}}
\newcommand{\ulK }{\underline{K}}
\newcommand{\ulL }{\underline{L}}

\newcommand{\Wg }{\mathcal{W}}
\newcommand{\YD }[1][ ]{Yetter-Drinfel'd#1}
\newcommand{\YDcat }{\mathcal{YD}}
\newcommand{\ydZI }{ {}_{\fie \ndZ ^I}^{\fie \ndZ ^I}\YDcat }

\title[
Lusztig isomorphisms for Drinfel'd doubles
]{Lusztig isomorphisms for Drinfel'd doubles of bosonizations
of Nichols algebras of diagonal type}

\author{I.~Heckenberger}
\thanks{Partially supported by the German Research Foundation (DFG) in
the framework of a Heisenberg fellowship}
\address{Istv\'an Heckenberger, Mathematisches Institut,
Ludwig-Maximili\-ans-Universit\"at M\"unchen,
Theresienstr. 39,
D-80333 M\"unchen, Germany}
\email{i.heckenberger@googlemail.com}

\begin{document}

\begin{abstract}
    In the structure theory of quantized enveloping algebras, the algebra
    isomorphisms determined by Lusztig led to the first general construction
    of PBW bases of these algebras. Also, they have important applications to
    the representation theory of these and related algebras.
    In the present paper the Drinfel'd double for a class of graded Hopf
    algebras is investigated. Various quantum algebras, including
    small quantum groups and
    multiparameter quantizations of semisimple Lie algebras and of Lie
    superalgebras, are covered by the given definition.
    For these Drinfel'd doubles Lusztig maps are defined.
    It is shown that these maps induce isomorphisms between doubles of
    bosonizations of Nichols
    algebras of diagonal type. Further, the obtained isomorphisms
    satisfy Coxeter type relations in a generalized
    sense. As an application, the Lusztig isomorphisms are used to
    give a characterization of Nichols algebras of diagonal type with
    finite root system.


\end{abstract}
\maketitle

\section{Historical remarks}
\label{sec:history}

The emergence of quantum groups following the work of
Dinfel'd \cite{inp-Drinfeld1} and Jimbo \cite{a-Jimbo1} was characterized by
the appearance of a huge amount of papers considering generalizations
of quantized enveloping algebras of semisimple Lie algebras,
their structure theory, and their applications in physics and mathematics. One
of the remarkable discoveries with far reaching consequences in the field was
Lusztig's construction of automorphisms of $U_q(\lag )$, see
\cite{b-Lusztig93}. It led to the construction of Poincar{\'e}--Birkhoff--Witt
(PBW) bases of $U_q(\lag )$ and to the study of crystal bases. Lusztig's
isomorphisms are also very important for the representation theory of quantized
enveloping algebras.

As a particular type of generalization of quantized enveloping algebras, in
the early 1990s quantized enveloping algebras of contragredient Lie
superalgebras have been intensively studied, see e.\,g.\
\cite{a-KhorTol91}, \cite{a-FlLeiVin91},
\cite{inp-KhorTol95},
\cite{a-BeKaMel98}, and
\cite{a-Yam99}. First, as noted in the introduction of
\cite{a-KhorTol91}, it was not clear whether there is an appropriate structure
which could play a similar role for quantized Lie superalgebras
as the Weyl group does for quantized semisimple Lie algebras. After the
appearance of Serganova's work \cite{a-Serg96} on generalized root systems the
idea of a Weyl groupoid and corresponding Lusztig isomorphisms were mentioned
by Khoroshkin and Tolstoy \cite[p.\,16]{inp-KhorTol95}
and used implicitly by Yamane
\cite[Sects.\,7.5,8]{a-Yam99}, \cite{a-Yam99e} in a topological setting. 
Presumably because of technical difficulties
the response on these papers was not very high, and a more detailed
elaboration of these structures is
still missing. As a result of the project aiming the classification of
finite-dimensional Nichols algebras of diagonal type,
the Weyl groupoid was
rediscovered in a more general context
\cite{a-Heck06a}, based on a natural definition of generalized root
systems (different from the one of Serganova).
In the meantime, a complete list of Nichols algebras
of diagonal type with finite root system \cite{p-Heck06b} was
determined
and a piece of an appealing structure theory
of generalized root systems and Weyl groupoids
\cite{a-HeckYam08}, \cite{p-CH08} is available.
A very interesting perspective for the future is the existence of root
systems and Weyl groupoids for a much larger class of Hopf algebras
\cite{p-HeckSchn08a}.

Recently, for two-parameter quantizations of finite-dimensional
simple Lie algebras,
Bergeron, Gao, and Hu \cite{a-BerGaoHu06} study generalizations of
Lusztig isomorphisms.
In the rank two case they are
able to define such isomorphisms in the full generality of their approach
\cite[Sect.\,3]{a-BerGaoHu06}.
In the present paper it is shown how to
use the Weyl groupoid for the definition of Lusztig isomorphisms for
a large class of quantum doubles, including (standard and)
multiparameter quantizations of
enveloping algebras of semisimple Lie algebras and Lie superalgebras
and their small quantum group analogs.
An important fact is that the use of the Weyl groupoid removes most of the
technical assumptions in the definition of the quantum doubles under
investigation.
 
In the case of Lie superalgebras and their quantized analogs a new phenomenon
compared to semisimple Lie algebras arises. Namely, (quantum) Serre
relations are not sufficient to define the Lie (or quantized enveloping)
superalgebra by generators and relations, see
\cite{a-FlLeiVin91} and
\cite{a-KhorTol91}.
The determination of a minimal set of defining relations
turned out to be solvable in principal by using the Weyl groupoid, ---
see \cite{a-Yam99} and \cite[Thm.\,1.6]{a-GrozLei01}. where the latter
has unfortunately neither a proof nor a reference, ---
but it involves technical difficulties. In the classical case computations
were done by Grozman, Leites \cite{a-GrozLei01}, and Yamane
\cite{a-Yam99}. The latter paper also treats the quantum case for its
topological version. The fact that the papers
\cite{a-KhorTol91}, \cite{a-GrozLei01}, and 
\cite{a-Yam99} give different sets of defining relations,
shows that a description
avoiding case by case considerations would be of advantage for further study
of the subject.

The Weyl groupoid turned out to be the key structure to answer
the first part of
\cite[Question\,5.9]{inp-Andr02}, namely,
to determine all finite-dimensional Nichols algebras of diagonal type.
In view of the results discussed above it seems that the second part of
\cite[Question\,5.9]{inp-Andr02},
which asks for the defining relations of these algebras,
can be answered in its naive sense --- by giving explicit lists
--- only in a very technical way.
A possible application of Lusztig isomorphisms and their properties
is to give an answer to
\cite[Question\,5.9]{inp-Andr02} in a conceptual way based on the idea
described in \cite[Thm.\,1.6]{a-GrozLei01} for contragredient
Lie superalgebras.

\section{On the structure of this paper}

The mathematical part of the paper starts in the next section with recalling
some combinatorial aspects of Nichols algebras of diagonal type. The Weyl
groupoid and the root system of a bicharacter are at the heart of
the structure theory of finite-dimensional (and also more general)
Nichols algebras of diagonal type, and they will appear on many places in the
paper. Then in Sect.~\ref{sec:Ddouble} the Drinfel'd double $\cU (\chi )$
of the tensor algebra $\cU ^+(\chi )$
of a braided vector space of diagonal type, see Def.~\ref{de:cU} and
Prop.~\ref{pr:cUgenrel}, is studied. For the convenience of the reader, many
facts known from the theory of quantized enveloping algebras and
superalgebras are worked out explicitly in the presented more general context.
The style of the presentation and the notation follow the conventions in
standard textbooks on quantum groups. In this section, more precisely in
Prop.~\ref{pr:goodideals}, a characterization of
ideals of $\cU (\chi )$ admitting a triangular decomposition of the
corresponding quotient algebra is proven, which seems to be new even for
multiparameter quantizations of Kac-Moody algebras, see
\cite[Prop.\,3.4]{p-KharSaga07}.

In Sect.~\ref{sec:Nichdiag} the definition and structure of Nichols
algebras is recalled. Most facts appear in some form in the
literature.

The main part of the paper starts in Sect.~\ref{sec:li}.
There are two important aims chased from now on.
First, for a class of Drinfel'd
doubles $U(\chi )$ of Nichols algebras of diagonal type
the definition of Lusztig
isomorphisms is given in Thm.~\ref{th:Liso}. For this definition a
combinatorial restriction on $\chi $ is indispensable, as explained at the
beginning of Subsect.~\ref{ssec:defli}. The idea behind this condition is
that the Lusztig isomorphisms are natural realizations of elements of the Weyl
groupoid $\Wg (\chi )$ attached to the bicharacter $\chi $, see
Def.~\ref{de:Weylg},
and the definition of the generating reflections
of the Weyl groupoid requires a finiteness condition on $\chi $.
The proof of the fact, that the Lusztig maps are indeed well-defined and
isomorphisms, requires several intermediate results. Therefore, and in order
to obtain statements in a more general setting, the Lusztig maps $\LT _p$
and $\LT _p^-$ are
introduced in the most universal setting in Lemma~\ref{le:Lusztig1}.
Besides the obvious analogy to Lusztig's definition, the main difference is
the missing of the constant factors in $\LT _p(E_i)$. The advantage of this
modification is that one can avoid case by case checkings in the proofs of
\textit{all} of the results concerning the maps $\LT _p$ in this paper.
This is not a negligible fact in view of \cite[Subsect.\,39.2]{b-Lusztig93}
and the classification result for Nichols algebras with finite root
system in \cite{p-Heck06b}, even if one restricts
himself to the rank 2 cases. However, the given definition has also its
disadvantage: In equations as for example Eq.~\eqref{eq:T(Eim)-4}
and Eq.~\eqref{eq:LTCox-1} one can not remove the field $\fie $. This implies
in particular that the Coxeter relations in Thm.~\ref{th:LTCox} hold
``only'' up to an automorphism $\varphi _{\ula }$ defined in
Prop.~\ref{pr:algiso}(1).

The main results concerning Lusztig isomorphisms are variants of the
corresponding statements for quantized enveloping algebras of
semisimple Lie algebras.
\begin{itemize}
  \item Prop.~\ref{pr:Lusztig1}: The Lusztig maps induce
    isomorphisms between the Drinfel'd doubles of the
    corresponding Nichols algebras of diagonal type.
  \item Thm.~\ref{th:LTCox}: Lusztig isomorphisms satisfy Coxeter
    type relations, up to a natural automorphism of $U(\chi )$.
  \item Thm.~\ref{th:wEpinU+}: The images of certain generators
    under a Lusztig isomorphism are in the upper triangular part of the
    Drinfel'd double.
  \item Cor.~\ref{co:Tw0}: Description of
    the Lusztig isomorphism corresponding to a
    longest element of the Weyl groupoid.
\end{itemize}

The other important aim of the main part of the paper
is to give a characterization of Nichols algebras
of diagonal type having a finite root system.
The corresponding result is
Thm.~\ref{th:Nichchar}. The theorem claims, roughly speaking, that a
``natural'' ideal $\cI ^+(\chi )$ of $\cU ^+(\chi )$ is the defining ideal of
the Nichols algebra $U^+(\chi )$ if and only if for all
$\chi '\in \Ob (\Wg (\chi ))$
there exist further ``natural'' ideals $\cI ^+(\chi ')$ of
$\cU ^+(\chi ')$, such that all Lusztig maps between the corresponding
quotient algebras are
well-defined. This theorem is descriptive, and admits to check whether a given
family of ideals defines the corresponding family of Nichols algebras.
However, it does not tell how to construct a minimal set of generators for the
defining ideal of the Nichols algebra. This problem remains open for further
research.

The paper ends with an application of Thm.~\ref{th:Nichchar}. More precisely,
in Ex.~\ref{ex:finCartan}
it is proven that the Nichols algebra $U^+(\chi )$ associated
to a bicharacter of finite Cartan type is, if the main parameter is not a root
of $1$, defined by Serre relations only. This result is standard in the case of
usual quantized enveloping algebras.

If not indicated otherwise, all algebras in the text will be defined over a
base field $\fie $ of arbitrary characteristic, and they are associative and
have a unit. The coproduct, counit, and antipode of a Hopf algebra will be
denoted by $\copr $, $\coun $, and $\antip $, respectively.
For the coproduct of a Hopf algebra $H$ the Sweedler
notation $\copr (h)=h_{(1)}\otimes h_{(2)}$ for all $h\in H$ will be used.
In contrast, for the coproduct $\brcopr $ of a braided Hopf algebra
$H'$ we
follow the modified Sweedler notation of Andruskiewisch and Schneider, see
the end of the introduction in \cite{inp-AndrSchn02}, in form
of $\brcopr (h)=h^{(1)}\ot h^{(2)}$ for all $h\in H'$. For an arbitrary
coalgebra $C$ let $C\cop $ denote the vector space $C$ together with the
coproduct opposite to the one of $C$. The antipode of Hopf
algebras and braided Hopf algebras is always meant to be bijective.
Let $\ndZ $ and $\ndN $ denote the set of integers and positive
integers, resepectively, and set $\ndN _0=\ndN \cup \{0\}$.

\section{Preliminaries}
\label{sec:prelims}

Let $\fie $ be a field and let $\fienz =\fie \setminus \{0\}$.
For any non-empty finite set $I$ let
$\{\Ndb _i\,|\,i\in I\}$ denote the
standard basis of $\ndZ ^I$.

\subsection{$q$-binomial coefficients}
\label{ssec:qbinoms}

The assertions and formulas
in this subsection are analogs of those in standard textbooks on
quantum groups, see for example \cite[Sects.\,1.3,\,34.1]{b-Lusztig93},
\cite[Sect.\,1.2.12-13]{b-Joseph}, and \cite[Sect.\,2.1]{b-KS}.

Let $q\in \fienz $. Set $\qnum{0}{q}=0$ and for all $m\in \ndN $ let
\begin{align}
	\label{eq:qnum}
	\qnum{m}{q}=&\,1+q+\cdots +q^{m-1},&
	\qnum{-m}{q}=&\,-\qnum{m}{q}.
\end{align}
Let $\qfact{0}{q}=1$ and for all $m\in \ndN $ let $\qfact{m}{q}=\prod
_{n=1}^m\qnum{n}{q}$.

The quantum plane is the unital associative $\fie $-algebra
$$\fie \langle u,v\rangle /(vu-quv).$$
The set $\{u^mv^n\,|\,m,n\in \ndN _0\}$ is a $\fie $-basis
of this algebra.
For all $m\in \ndN _0$ and $n\in \ndZ $ define $\qchoose{m}{n}{q}\in \fie $
by the equation
$$(u+v)^m=\sum _{n\in \ndZ }\qchoose{m}{n}{q}u^nv^{m-n}.$$
Since $(u+v)^{m+1}=(u+v)(u+v)^m=(u+v)^m(u+v)$, one obtains that
\begin{align}\label{eq:m+1choosen}
	\qchoose{m}{n-1}{q}+q^n\qchoose{m}{n}{q}
	=q^{m-n+1}\qchoose{m}{n-1}{q}+\qchoose{m}{n}{q}
	=\qchoose{m+1}{n}{q}
\end{align}
for all $m\in \ndN _0$, $n\in \ndZ $.
As a special case one gets
\begin{gather*}
	\qchoose{m}{n}{q}=0 \quad \text{for $n<0$ or $n>m$,}\\
	\qchoose{m}{0}{q}=\qchoose{m}{m}{q}=1,\quad
	\qchoose{m}{1}{q}=\qchoose{m}{m-1}{q}=\qnum{m}{q}.
\end{gather*}

\begin{lemma}
	\label{le:qchooserel}
	Let $q\in \fienz $, $m\in \ndN _0$, and $n\in \ndZ $. Then
	\begin{align*}
		\qnum{n+1}{q}\qchoose {m}{n+1}{q}=\qnum{m-n}{q}\qchoose{m}{n}{q}.
	\end{align*}
\end{lemma}
\begin{proof}
	Proceed by induction on $m$.
	If $m=0$, then both sides of the equation are zero for all $n\in
  \ndZ $. Suppose now that the claim holds for some $m\in \ndN _0$.
  Then Eq.~\eqref{eq:m+1choosen} and the induction hypothesis
  imply that
	\begin{align*}
		&\qnum{n+1}{q}\qchoose{m+1}{n+1}{q}=\qnum{n+1}{q}\left(
		\qchoose{m}{n}{q}+q^{n+1}\qchoose{m}{n+1}{q}\right)\\
		&\quad =q^{n}\qchoose{m}{n}{q}
		+\qnum{n}{q}\qchoose{m}{n}{q}
		+q^{n+1}\qnum{n+1}{q}\qchoose{m}{n+1}{q}\\
		&\quad =q^{n}\qchoose{m}{n}{q}
		+\qnum{m-n+1}{q}\qchoose{m}{n-1}{q}
		+q^{n+1}\qnum{m-n}{q}\qchoose{m}{n}{q}\\
		&\quad =q^{n}\qnum{m-n+1}{q}\qchoose{m}{n}{q}
		+\qnum{m-n+1}{q}\qchoose{m}{n-1}{q}\\
		&\quad =\qnum{m-n+1}{q}\qchoose{m+1}{n}{q}
	\end{align*}
  for all $n\in \ndZ $.
	This proves the lemma.
\end{proof}

\begin{lemma}
	\label{le:mq=0}
	Let $q\in \fienz $ and $m\in \ndN $. Assume that $\qnum{m}{q}=0$
	and $\qfact{m-1}{q}\not=0$.
  Then $\qchoose{m}{n}{q}=0$ for all $n\in \ndN $ with $n<m$.
\end{lemma}

\begin{proof}
	The assumption yields that $\qchoose{m}{1}{q}=\qnum{m}{q}=0$. Using
	Lemma~\ref{le:qchooserel} the claim follows easily by induction on $n$.
\end{proof}

\subsection{Cartan schemes, Weyl groupoids, and root systems}
\label{ssec:CS}

The combinatorics of the Drinfel'd double of the bosonization of a
Nichols algebra of diagonal type
is controlled to a large extent by its Weyl groupoid.
Here the language developed in
\cite{p-CH08} is used. Substantial part of the theory was obtained first in
\cite{a-HeckYam08}. In this subsection the most important definitions and
facts are recalled.

Let $I$ be a non-empty finite set.
By \cite[\S 1.1]{b-Kac90} a generalized Cartan matrix
$C=(c_{ij})_{i,j\in I}$
is a matrix in $\ndZ ^{I\times I}$ such that
\begin{enumerate}
  \item[(M1)] $c_{ii}=2$ and $c_{jk}\le 0$ for all $i,j,k\in I$ with
    $j\not=k$,
  \item[(M2)] if $i,j\in I$ and $c_{ij}=0$, then $c_{ji}=0$.
\end{enumerate}

\begin{defin} \label{de:CS}
  Let $I$ be a non-empty finite set,
  $A$ a non-empty set, $r_i : A \to A$ a map for all $i\in I$,
  and $C^a=(c^a_{jk})_{j,k \in I}$ a generalized Cartan matrix
  in $\ndZ ^{I \times I}$ for all $a\in A$. The quadruple
  \[ \cC = \cC (I,A,(r_i)_{i \in I}, (C^a)_{a \in A})\]
  is called a \textit{Cartan scheme} if
  \begin{enumerate}
  \item[(C1)] $r_i^2 = \id$ for all $i \in I$,
  \item[(C2)] $c^a_{ij} = c^{r_i(a)}_{ij}$ for all $a\in A$ and
    $i,j\in I$.
  \end{enumerate}
  One says that a Cartan scheme $\cC $ is \textit{connected}, if
  the group $\langle r_i\,|\,i\in I\rangle \subset \Aut (A)$ acts
  transitively on $A$, that is, if
  for all $a,b\in A$ with $a\not=b$ there exist $n\in \ndN _0$
  and $i_1,i_2,\ldots ,i_n\in I$ such that $b=r_{i_n}\cdots r_{i_2}
  r_{i_1}(a)$.
  Two Cartan schemes $\cC =\cC (I,A,(r_i)_{i\in I},(C^a)_{a\in A})$
  and $\cC '=\cC '(I',A',$
  $(r'_i)_{i\in I'},({C'}^a)_{a\in A'})$
  are called
  \textit{equivalent}, if there are bijections $\varphi _0:I\to I'$
  and $\varphi _1:A\to A'$ such that
  \begin{align}\label{eq:equivCS}
    \varphi _1(r_i(a))=r'_{\varphi _0(i)}(\varphi _1(a)),
    \qquad
    c^{\varphi _1(a)}_{\varphi _0(i) \varphi _0(j)}=c^a_{i j}
  \end{align}
  for all $i,j\in I$ and $a\in A$.

  Let $\cC = \cC (I,A,(r_i)_{i \in I}, (C^a)_{a \in A})$ be a
  Cartan scheme. For all $i \in I$ and $a \in A$ define $\s _i^a \in
  \Aut(\ndZ ^I)$ by
  \begin{align}
    \s _i^a (\Ndb _j) = \Ndb _j - c_{ij}^a\Ndb _i \qquad
    \text{for all $j \in I$.}
    \label{eq:sia}
  \end{align}
  This map is a reflection. 
  The \textit{Weyl groupoid of} $\cC $
  is the category $\Wg (\cC )$ such that $\Ob (\Wg (\cC ))=A$ and
  the morphisms are generated by the maps
  $\s _i^a\in \Hom (a,r_i(a))$ with $i\in I$, $a\in A$.
  Formally, for $a,b\in A$ the set $\Hom (a,b)$ consists of the triples
  $(b,f,a)$, where
  \[ f=\s _{i_n}^{r_{i_{n-1}}\cdots r_{i_1}(a)}\cdots
    \s _{i_2}^{r_{i_1}(a)}\s _{i_1}^a \]
  and $b=r_{i_n}\cdots r_{i_2}r_{i_1}(a)$ for some
  $n\in \ndN _0$ and $i_1,\ldots ,i_n\in I$.
  The composition is induced by the group structure of $\Aut (\ndZ ^I)$:
  \[ (a_3,f_2,a_2)\circ (a_2,f_1,a_1) = (a_3,f_2f_1, a_1)\]
  for all $(a_3,f_2,a_2),(a_2,f_1,a_1)\in \Hom (\Wg (\cC ))$.
  By abuse of notation one also writes
  $f\in \Hom (a,b)$ instead of $(b,f,a)\in \Hom (a,b)$.
 
  The cardinality of $I$ is termed the \textit{rank of} $\Wg (\cC )$.
  A Cartan scheme is called \textit{connected} if its Weyl groupoid
  is connected.
\end{defin}

Recall that a groupoid is a category such that all morphisms are
isomorphisms.
The Weyl groupoid $\Wg (\cC )$ of a Cartan scheme $\cC $ is a groupoid,
see \cite{p-CH08}. For all $i\in I$ and $a\in A$
the inverse of $\s _i^a$ is $\s _i^{r_i(a)}$.
If $\cC $ and $\cC '$ are equivalent Cartan schemes, then $\Wg (\cC )$ and
$\Wg (\cC ')$ are isomorphic groupoids.

A groupoid $G$ is called \textit{connected},
if for each $a,b\in \Ob (G)$ the class $\Hom (a,b)$ is non-empty.
Hence $\Wg (\cC )$ is a connected groupoid if and only if $\cC $ is a
connected Cartan scheme.

\begin{defin} \label{de:RSC}
  Let $\cC =\cC (I,A,(r_i)_{i\in I},(C^a)_{a\in A})$ be a Cartan
  scheme. For all $a\in A$ let $R^a\subset \ndZ ^I$, and define
  $m_{i,j}^a= |R^a \cap (\ndN_0\Ndb _i + \ndN_0\Ndb _j)|$ for all $i,j\in
  I$ and $a\in A$. One says that
  \[ \rsC = \rsC (\cC , (R^a)_{a\in A}) \]
  is a \textit{root system of type} $\cC $, if it satisfies the following
  axioms.
  \begin{enumerate}
    \item[(R1)]
      $R^a=R^a_+\cup - R^a_+$, where $R^a_+=R^a\cap \ndN_0^I$, for all
      $a\in A$.
    \item[(R2)]
      $R^a\cap \ndZ \Ndb _i=\{\Ndb _i,-\Ndb _i\}$ for all $i\in I$, $a\in A$.
    \item[(R3)]
      $\s _i^a(R^a) = R^{r_i(a)}$ for all $i\in I$, $a\in A$.
    \item[(R4)]
      If $i,j\in I$ and $a\in A$ such that $i\not=j$ and $m_{i,j}^a$ is
      finite, then
      $(r_ir_j)^{m_{i,j}^a}(a)=a$.
  \end{enumerate}
  If $\rsC $ is a root system of type $\cC $, then
  $\Wg (\rsC )=\Wg (\cC )$ is the \textit{Weyl groupoid of} $\rsC $.
  Further, $\rsC $ is called \textit{connected}, if $\cC $ is a connected
  Cartan scheme.
  If $\rsC =\rsC (\cC ,(R^a)_{a\in A})$ is a root system of type $\cC $
  and $\rsC '=\rsC '(\cC ',({R'}^a_{a\in A'}))$ is a root system of
  type $\cC '$, then we say that $\rsC $ and $\rsC '$ are \textit{equivalent},
  if $\cC $ and $\cC '$ are equivalent Cartan schemes given by maps $\varphi
  _0:I\to I'$, $\varphi _1:A\to A'$ as in Def.~\ref{de:CS}, and if
  the map $\varphi _0^*:\ndZ^I\to \ndZ^{I'}$ given by
  $\varphi _0^*(\Ndb _i)=\Ndb _{\varphi _0(i)}$ satisfies
  $\varphi _0^*(R^a)={R'}^{\varphi _1(a)}$ for all $a\in A$.
\end{defin}

There exist many interesting examples of root systems of type $\cC $ related
to semisimple Lie algebras, Lie superalgebras and Nichols algebras of diagonal
type, respectively. Further details and results can be found in
\cite{a-HeckYam08} and \cite{p-CH08}.

\begin{conve}\label{con:uind}
  In connection with Cartan schemes $\cC $,
  upper indices usually refer to elements
  of $A$. Often, these indices will be omitted if they are uniquely determined
  by the context. The notation $w1_a$ and $1_aw'$
  with $w,w'\in \Hom (\Wg (\cC ))$ and $a\in A$ means that $w\in \Hom
  (a,\underline{\,\,})$ and $w'\in \Hom (\underline{\,\,},a)$,
  respectively.
\end{conve}

A fundamental result about Weyl groupoids is the following theorem.

\begin{theor}\cite[Thm.\,1]{a-HeckYam08}\label{th:Coxgr}
  Let $\cC =\cC (I,A,(r_i)_{i\in I},(C^a)_{a\in A})$
  be a Cartan scheme and $\rsC =\rsC (\cC ,(R^a)_{a\in A})$ a root system
  of type $\cC $.
  Let $\Wg $ be the abstract
  groupoid with $\Ob (\Wg )=A$ such that $\Hom (\Wg )$ is
  generated by abstract morphisms $s_i^a\in \Hom (a,r_i(a))$,
  where $i\in I$ and $a\in A$, satisfying the relations
  \begin{align*}
    s_i s_i 1_a=1_a,\quad (s_j s_k)^{m_{j,k}^a}1_a=1_a,
    \qquad a\in A,\,i,j,k\in I,\, j\not=k,
  \end{align*}
  see Conv.~\ref{con:uind}.
  Here $1_a$ is the identity of the object $a$,
  and $(s_j s_k)^\infty 1_a$ is understood to be
  $1_a$. The functor $\Wg \to \Wg (\rsC )$, which is
  the identity on the objects, and on the set of
  morphisms is given by
  $s _i^a\mapsto \s_i^a$ for all $i\in I$, $a\in A$,
  is an isomorphism of groupoids.
\end{theor}

If $\cC $ is a Cartan scheme,
then the Weyl groupoid $\Wg (\cC )$ admits a length function $\ell :\Wg (\cC
)\to \ndN _0$ such that
\begin{align}
  \ell (w)=\min \{k\in \ndN _0\,|\,\exists i_1,\dots ,i_k\in I,a\in A:
  w=\s _{i_1}\cdots \s _{i_k}1_a\}
  \label{eq:ell}
\end{align}
for all $w\in \Wg (\cC )$.
If there exists a root system of type $\cC $,
then $\ell $ has very similar properties to the well-known length function
for Weyl groups, see \cite{a-HeckYam08}.

\begin{lemma}
  Let $\cC $ be a Cartan scheme and $\rsC $ a root system of type $\cC $.
  Let $a\in A$. Then $-c^a_{ij}=\max\{m\in \ndN _0\,|\,\Ndb _j+m\Ndb _i\in
  R^a_+\}$ for all $i,j\in I$ with $i\not=j$.
  \label{le:cm}
\end{lemma}

\begin{proof}
  By (C2) and (R3), $\s _i^{r_i(a)}(\Ndb _j)=\Ndb _j-c^a_{ij}\Ndb _i\in R^a_+$.
  Hence $-c^a_{ij}\le \max\{m\in \ndN _0\,|\,\Ndb _j+m\Ndb _i\in R^a_+\}$.
  On the other hand, if $\Ndb _j+m\Ndb _i\in R^a_+$, then $\s _i^a(\Ndb _j+m\al
  _i)=\Ndb _j+(-c^a_{ij}-m)\Ndb _i\in R^{r_i(a)}_+$ by (R3) and (R1), and hence
  $m\le -c^a_{ij}$. This proves the lemma.
\end{proof}

Let $\cC $ be a Cartan scheme and $\rsC $ a root system of type $\cC $.
We say that $\rsC $ is \textit{finite}, if $R^a$ is finite for all $a\in A$.
Following the terminology in \cite{b-Kac90}, for all $a\in A$ one
defines
\begin{align}
  \rroots a=\{w(\Ndb _i)\,|\,w\in \Hom (b,a),b\in A, i\in I\},
  \label{eq:rroots}
\end{align}
the set of \textit{real roots of} $a\in A$. Then
$\rrsC =\rrsC (\cC ,(\rroots a)_{a\in A})$ is a root system of type $\cC $.
The following lemmata are well-known for traditional root systems.

\begin{lemma} \cite[Lemma 2.11]{p-CH08}
  Let $\cC $ be a connected Cartan scheme and $\rsC $ a root system of type
  $\cC $. The following are equivalent.
  \begin{enumerate}
    \item $\rsC $ is finite.
    \item $R^a$ is finite for at least one $a\in A$.
    \item $\rrsC $ is finite.
    \item $\Wg (\rsC )$ is finite.
  \end{enumerate}
  \label{le:Rfincond}
\end{lemma}

\begin{lemma} \cite[Cor.\,5]{a-HeckYam08}
  Let $\cC $ be a connected Cartan scheme and $\rsC $ a finite root system
  of type $\cC $. Then for all $a\in A$ there exist unique elements $b\in A$
  and $w\in \Hom (b,a)$ such that $|R^a_+|=\ell (w)\ge \ell (w')$
  for all $w'\in \Hom (b',a')$, $a',b'\in A$.
  \label{le:longestw}
\end{lemma}

\subsection{The Weyl groupoid of a bicharacter}
\label{ssec:Weylgroupoid}

For an introduction to groupoids see \cite{a-Brown87}.
In this subsection the Weyl groupoid of a bicharacter
is introduced
following the general structure in the previous subsection.
This definition, which differs from the original one in
\cite[Sect.\,5]{a-Heck06a},
has many advantages.
One of them is that it fits better with the modern category theoretical
point of view. Another one is that for quantized enveloping algebras
of semisimple Lie algebras the Weyl groupoid becomes the Weyl group of the Lie
algebra, see Ex.~\ref{ex:cartantype}.

Let $I$ be a non-empty finite set.
Recall that a bicharacter on $\ndZ ^I$ with values in $\fienz $ is a map
$\chi :\ndZ ^I\times \ndZ ^I\to \fienz $ such that
\begin{align}
	\chi (a+b,c)=&\chi (a,c)\chi (b,c),&
	\chi (c,a+b)=&\chi (c,a)\chi (c,b)
	\label{eq:bichar}
\end{align}
for all $a,b,c\in \ndZ ^I$.
Then $\chi (0,a)=\chi (a,0)=1$ for all $a\in \ndZ ^I$.
Let $\cX $ denote the set of bicharacters on $\ndZ ^I$.  
For all $\chi \in \cX $ the maps
\begin{align}
  \label{eq:chiop}
  \chi \op : & \ndZ ^I\times \ndZ ^I\to \fienz ,&
  \chi \op (a,b)=&\,\chi (b,a),\\
  \label{eq:chiinv}
  \chi ^{-1} : & \ndZ ^I\times \ndZ ^I\to \fienz ,&
  \chi ^{-1}(a,b)=&\,\chi (a,b)^{-1},
  \intertext{and for all $\chi \in \cX $, $w\in \Aut _\ndZ (\ndZ ^I)$ the map}
  \label{eq:w*chi}
  w^*\chi : & \ndZ ^I\times \ndZ ^I\to \fienz ,&
  w^*\chi (a,b)=&\,\chi (w^{-1}(a),w^{-1}(b)),
\end{align}
are bicharacters on $\ndZ ^I$. The equation
\begin{align}
	\label{eq:w*functor}
	(ww')^*\chi =w^*(w'{}^*\chi )
\end{align}
holds for all $w,w'\in \Aut _{\ndZ }(\ndZ ^I)$ and all $\chi \in \cX $.

\begin{defin}\label{de:Cartan}
	Let $\chi \in \cX $, $p\in I$, and
	$q_{ij}=\chi (\Ndb _i,\Ndb _j)$ for all $i,j\in I$.
  Then $\chi $ is called $p$-\textit{finite},
  if for all $j\in I\setminus \{p\}$ there exists
  $m\in \ndN _0$ such that $\qnum{m+1}{q_{pp}}=0$ or
  $q_{pp}^m q_{pj}q_{jp}=1$.

  Assume that $\chi $ is $p$-finite.
  Let $c_{p p}^\chi =2$, and for all $j\in I\setminus \{p\}$ let
	$$ c_{pj}^\chi =-\min \{m\in \ndN _0 \,|\,
	(m+1)_{q_{pp}}(q_{pp}^m q_{pj} q_{jp}-1)=0\}. $$
  If $\chi $ is $i$-finite for all $i\in I$, then the matrix
  $C^\chi =(c_{ij}^\chi )_{i,j\in I}$ is called the
	\textit{Cartan matrix} associated to $\chi $.
  It is a generalized Cartan matrix, see Sect.~\ref{ssec:CS}.
\end{defin}

For all $p\in I$ and $\chi \in \cX $, where $\chi $ is $p$-finite, let
$\s _p^\chi \in \Aut _\ndZ (\ndZ ^I)$,
\begin{align}\label{eq:refl}
  \s _p^\chi (\Ndb _j)=\Ndb _j-c_{pj}^\chi \Ndb _p \quad \text{for all $j\in I$.}
\end{align}
Towards the definition of the Weyl groupoid of a
bicharacter, bijections $r_p:\cX \to \cX $
are defined for all $p\in I$. Namely, let
\begin{align}
  r_p: \cX \to \cX,\quad
  r_p(\chi )=
  \begin{cases}
    (\s _p^\chi )^*\chi & \text{if $\chi $ is $p$-finite,}\\
    \chi & \text{otherwise.}
  \end{cases}
  \label{eq:rp}
\end{align}
Let $p\in I$, $\chi \in \cX $,
$q_{ij}=\chi (\Ndb _i,\Ndb _j)$ for all $i,j\in I$.
If $\chi $ is $p$-finite, then
\begin{equation}
  \begin{aligned}
    r_p(\chi )(\Ndb _p,\Ndb _p)=&q_{p p}, &
    r_p(\chi )(\Ndb _p,\Ndb _j)=&q_{p j}^{-1}q_{p p}^{c_{pj}^\chi },\\
    r_p(\chi )(\Ndb _i,\Ndb _p)=&q_{i p}^{-1}q_{p p}^{c_{pi}^\chi },&
    r_p(\chi )(\Ndb _i,\Ndb _j)=&q_{i j} q_{i p}^{-c_{p j}^\chi }
    q_{p j}^{-c_{p i}^\chi } q_{p p}^{c_{pi}^\chi c_{p j}^\chi }
  \end{aligned}
  \label{eq:rpchi}
\end{equation}
for all $i,j\in I$. 
It is a small exercise to check that then
$(\s _p^\chi )^*\chi $ is $p$-finite, and
\begin{align}\label{eq:rp2}
  c_{pj}^{r_p(\chi )}=c_{pj}^\chi \quad \text{for all $j\in I$},
  \qquad r_p^2(\chi )=\chi .
\end{align}
The bijections $r_p$, $p\in I$, generate a subgroup
\begin{align*}
  \cG =\langle r_p\,|\, p\in I\rangle
\end{align*}
of the group of bijections of the set $\cX $. For all $\chi \in \cX $ let
$\cG (\chi )$ denote the $\cG $-orbit of $\chi $ under the action of $\cG $. 

Let $\chi \in \cX $ such that $\chi '$ is $p$-finite for all
$\chi '\in \cG (\chi )$ and $p\in I$.
Then
$$\cC (\chi )=
\cC (I,\cG (\chi ),(r_p)_{p\in I},(C^{g(\chi )})_{g\in \cG })$$
is a connected Cartan scheme
by Eq.~\eqref{eq:rp2}.

\begin{defin}\label{de:Weylg}
  Let $\chi \in \cX $ such that $\chi '$ is $p$-finite for all
  $\chi '\in \cG (\chi )$ and $p\in I$.
  Then the Weyl groupoid of $\chi $ is the Weyl groupoid
  of the Cartan scheme
  $\cC (\chi )$ and is denoted by $\Wg (\chi )$.
\end{defin}

Clearly, $\cC (\chi )=\cC (\chi ')$ and
$\Wg (\chi )=\Wg (\chi ')$ for all $\chi '\in \cG (\chi )$.

\begin{examp}\label{ex:cartantype}
  Let $C=(c_{ij})_{i,j\in I}$ be a generalized Cartan matrix. Let
  $\chi \in \cX $, $q_{ij}=\chi (\Ndb _i,\Ndb _j)$ for all $i,j\in I$,
  and assume that $q_{ii}^{c_{ij}}=q_{ij}q_{ji}$ for all $i,j\in I$,
  and that
  $\qnum{m+1}{q_{ii}}\not=0$ for all $i\in I$, $m\in \ndN _0$ with
  $m<\max \{-c_{ij}\,|\,j\in I\}$.
  (The latter is not an essential assumption, since if it fails, then
  one can replace $C$ by another generalized Cartan matrix
  $\tilde{C}$,
  such that $\chi $ has this property with respect to $\tilde{C}$.)
  One says that $\chi $ is of \textit{Cartan type}.
  Then $\chi $ is $p$-finite for all $p\in I$,
  and $c_{ij}^\chi =c_{ij}$ for all $i,j\in I$ by
  Def.~\ref{de:Cartan}. Eq.~\eqref{eq:rpchi} gives that
  \begin{align*}
    r_p(\chi )(\Ndb _i,\Ndb _i)=&q_{ii}=\chi (\Ndb _i,\Ndb _i),\\
    r_p(\chi )(\Ndb _i,\Ndb _j)\,r_p(\chi )(\Ndb _j,\Ndb _i)=&q_{ij}q_{ji}=
    r_p(\chi )(\Ndb _i,\Ndb _i)^{c_{ij}}
  \end{align*}
  for all $p,i,j\in I$.
  Hence $r_p(\chi )$ is again of Cartan type with the same Cartan matrix $C$.
  Thus $\chi '$ is $p$-finite for all $\chi '\in \cG (\chi )$ and $p\in I$.

  Assume now that
  $C$ is a symmetrizable generalized Cartan matrix.
  For all $i\in I$ let $d_i\in \ndN $ such that
  $d_ic_{ij}=d_jc_{ji}$ for all $i,j\in I$.
  Let $q\in \fienz $ such that $\qnum{m+1}{q^{2d_i}}\not=0$
  for all $i\in I$ and $m\in \ndN _0$
  with $m<\max \{-c_{ij}\,|\,j\in I\}$.
  Assume that
  $\chi (\Ndb _i,\Ndb _j)=q^{d_ic_{ij}}$ for all $i,j\in I$.
  Then $\chi $ is of Cartan type, hence $\chi $ is $p$-finite for all $p\in I$.
  Eq.~\eqref{eq:rpchi} implies that $r_p(\chi )=\chi $ for all $p\in I$,
  and hence $\cG (\chi )$ consists of precisely one element.
  In this case the Weyl groupoid $\Wg (\chi )$ is the group generated by
  the reflections $\s _p^\chi $ in Eq.~\eqref{eq:refl}, and hence
  $\Wg (\chi )$ is just
  the Weyl group associated to the generalized Cartan matrix $C$.
\end{examp}

\subsection{Roots and real roots}
\label{ssec:roots}

Let $I$ be a non-empty finite set and let $\chi \in \cX $, that is, a
bicharacter on $\ndZ ^I$ with values in $\fienz $.
Under suitable conditions there exists a canonical root system of type
$\cC (\chi)$ which is described in
this subsection. It is based on the construction of a restricted PBW basis
of Nichols algebras of diagonal type. More details can be found in
Sect.~\ref{sec:Nichdiag} and in \cite{inp-AndrSchn02} on Nichols algebras,
in \cite{a-Khar99} on the PBW
basis, and in \cite[Sect.\,3]{a-Heck06a} on the root system.

Let $\fie \ndZ ^I$ denote the group algebra of $\ndZ ^I$.
Let $V\in \ydZI $ be a $|I|$-dimensional \YD module
of diagonal type. Let $\coal :V\to \fie \ndZ ^I\ot V$ and $\actl :\fie \ndZ
^I\ot V\to V$ denote the left coaction and the left action of $\fie \ndZ ^I$
on $V$, respectively. Fix a basis $\{x_i\,|\,i\in I\}$ of $V$, elements
$g_i$, where $i\in I$, and a matrix $(q_{ij})_{i,j\in I}\in (\fienz
)^{I\times I}$, such that
\[ \coal (x_i)=g_i\ot x_i,\quad g_i\actl x_j=q_{ij}x_j \quad
\text{for all $i,j\in I$.} \]
Assume that $\chi (\Ndb _i,\Ndb _j)=q_{ij}$ for all $i,j\in I$.
For all $\Ndb \in \ndZ ^I$ define
\begin{align}
  \hght \chi (\Ndb )=
  \begin{cases}
    \min \{ m\in \ndN \,|\,
    \qnum{m}{\chi (\Ndb ,\Ndb )}=0\} & \text{if $\qnum{m}{\chi (\Ndb
    ,\Ndb )}=0$}\\
    & \text{for some $m\in \ndN $,}\\
    \infty & \text{otherwise.}
  \end{cases}
  \label{eq:height}
\end{align}
If $p\in I$ such that $\chi $ is $p$-finite, then
\begin{align}
  \hght{r_p(\chi )}(\s _p^\chi (\Ndb ))=\hght \chi (\Ndb )\quad
  \text{for all $\Ndb \in \ndZ ^I$}
  \label{eq:hghtrpchi}
\end{align}
by Eq.~\eqref{eq:w*chi}.

The tensor algebra $T(V)$ admits a universal braided Hopf algebra quotient
$\NA V$, called the \textit{Nichols algebra of} $V$.
As an algebra, $\NA V$ has a unique $\ndZ ^I$-grading
\begin{align}
  \NA V=\oplus _{\Ndb \in \ndZ ^I}\NA{V}_{\Ndb }
  \label{eq:NAVgrading}
\end{align}
such that $\deg x_i=\Ndb _i$ for all $i\in I$, see \cite[Rem.\,2.8]{p-AHS08}.
This is also a coalgebra grading.
There exists a totally ordered index set $(L,\le )$ and a family
$(y_l)_{l\in L}$ of $\ndZ ^I$-homogeneous elements $y_l\in \NA V$ such that
the set
\begin{equation}
\begin{aligned}
  \{ y_{l_1}^{m_1}y_{l_2}^{m_2}\cdots y_{l_k}^{m_k}\,|\,
  &k\ge 0,\,l_1,\dots ,l_k\in L,\,l_1>l_2>\cdots >l_k,\\
  &m_i\in \ndN _0,\,m_i<\hght \chi (\deg y_{l_i})
  \quad \text{for all $i\in I$}\}
\end{aligned}
  \label{eq:PBWbasis}
\end{equation}
forms a vector space basis of $\NA V$.
Comparing dimensions of homogeneous components gives that the set
\[ R^\chi _+=\{\deg y_l\,|\,l\in L\}\subset \ndN _0^I \]
depends on the matrix
$(q_{ij})_{i,j\in I}$, but not on the choice of
the basis $\{x_i\,|\,i\in I\}$, the set $L$,
and the elements $g_i$, $i\in I$, and $y_l$, $l\in L$. Let
\[ R^\chi =R^\chi _+\cup -R^\chi _+. \]

\begin{theor}\label{th:rsCchi}
  Let $\chi \in \cX $ such that $\chi '$ is $p$-finite for all $p\in I$,
  $\chi '\in \cG (\chi )$. Then
  $\rsC (\chi )=\rsC (\cC (\chi ), (R^{\chi '})_{\chi
  '\in \cG (\chi )})$ is a root system of type $\cC (\chi )$.
\end{theor}

\begin{proof}
  Axiom (R1) holds by definition. Next let us prove
  Axiom (R2). By definition of $R^a$, it suffices to consider the case
  $a=\chi $. Note that
  $\NA V_{m\Ndb _i}=\fie x_i^m$ for all $m\ge 0$, and $x_i^m$ is
  zero if and only if $\qfact{m}{q_{ii}}=0$.
  Therefore by Eq.~\eqref{eq:PBWbasis}
  and the definition of $\hght \chi (\Ndb _i)$
  there is precisely one element $y_l$, where $l\in L$,
  of degree
  $m\Ndb _i$, $m\ge 1$, and this is of degree $\Ndb _i$.
  This gives (R2).

  Axiom (R3) holds by \cite[Prop.\,1]{a-Heck06a}.
  Finally, let $i,j\in I$
  such that $i\not=j$ and $m^\chi _{ij}$ is finite. Since
  $\chi '$ is $p$-finite for all $p\in I$ and $\chi '\in \cG (\chi )$,
  the calculation in the proof of
  \cite[Lemma 5]{a-HeckYam08} --- which does not use Axiom (R4) ---
  implies that
  \[ (\s _i\s _j)^{m^\chi _{ij}}1_{\chi }=\id . \]
  Hence $(r_ir_j)^{m^\chi _{ij}}(\chi )=
  \big((\s _i\s _j)^{m^\chi _{ij}}1_{\chi }\big)^*\chi 
  =\id ^*\chi =\chi $.
  This proves the theorem.
\end{proof}

\begin{remar}
  Let $\chi \in \cX $. Assume that $\chi '$ is $p$-finite for all $\chi '\in
  \cG (\chi )$ and $p\in I$.
  In general, the matrices $C^{\chi '}$ and $C^{\chi }$, where $\chi '\in \cG
  (\chi )$, do not coincide. If $R^\chi $ is finite, then $C^{\chi '}$
  with $\chi '\in \cG (\chi )$ does
  not need to be a Cartan matrix of finite type,
  see e.\,g.~\cite[Table\,1, row 17]{a-Heck07a}.
\end{remar}

\begin{lemma} \label{le:equalrs}
  Let $\chi ,\chi '\in \cX $.
  Assume that $\chi ''$ is $p$-finite for all $p\in I$,
  $\chi ''\in \cG (\chi )\cup \cG (\chi ')$.

  (i) If $R^\chi _+=R^{\chi '}_+$, then
  $C^{w^*\chi }=C^{w^*\chi '}$ for all $w\in \Hom (\chi ,\underline{\,\,})
  \subset \Hom (\Wg (\chi ))$.

  (ii) Assume that $R^\chi _+$ and $R^{\chi '}_+$ are finite sets.
  If
  $C^{w^*\chi }=C^{w^*\chi '}$ for all $w\in \Hom (\chi ,\underline{\,\,})
  \subset \Hom (\Wg (\chi ))$, then
  $R^\chi _+=R^{\chi '}_+$.
\end{lemma}

\begin{proof}
  (i) Assume that $R^\chi _+=R^{\chi '}_+$.
  Then $C^{\chi }=C^{\chi '}$ by Lemma~\ref{le:cm}.
  Therefore $\s _p^\chi =\s _p^{\chi '}$ in $\Aut (\ndZ ^I)$
  for all $p\in I$.
  Using the finiteness assumption on $\chi $ and $\chi '$
  and Axiom (R3), by induction it follows that
  \begin{align}
    \s_{i_1}\cdots \s _{i_k}^\chi =\s _{i_1}\cdots
    \s _{i_k}^{\chi '}\text{ in $\Aut (\ndZ ^I)$},\,\,
  R^{(\s_{i_1}\cdots \s _{i_k}^\chi )^*\chi }=
  R^{(\s_{i_1}\cdots \s _{i_k}^{\chi '})^*\chi '},
    \label{eq:sss=sss}
  \end{align}
  and
  $C^{(\s_{i_1}\cdots \s _{i_k}^\chi )^*\chi }=
  C^{(\s_{i_1}\cdots \s _{i_k}^{\chi '})^*\chi '}$
  for all $k\in \ndN _0$ and $i_1,\dots ,i_k\in I$.
  Hence
  $C^{w^*\chi }=C^{w^*\chi '}$ for all
  $w\in \Hom (\chi ,\underline{\,\,})
  \subset \Hom (\Wg (\chi ))$.

  (ii) Since $R^\chi _+$ is finite,
  $R^\chi =\{w^{-1}(\Ndb _i)\,|\,w\in \Hom
  (\chi ,\underline{\,\,})\subset \Hom (\Wg (\chi ))\}$
  by \cite[Prop.\,2.12]{p-CH08}. By assumption on the Cartan matrices,
  the first formula in Eq.~\eqref{eq:sss=sss} holds
  for all $k\in \ndN _0$ and $i_1,\dots,i_k\in I$.
  Hence $R^\chi =R^{\chi '}$, and the
  lemma holds by (R1).
\end{proof}

Eqs.~\eqref{eq:chiop}--\eqref{eq:w*chi} describe natural relations between
various bicharacters on $\ndZ ^I$. These relations give rise to relations
between different Weyl groupoids and root systems, respectively.

\begin{propo}\label{pr:w*func}
	Let $\chi \in \cX $.

  (a)
  If $\chi $ is $p$-finite for all $p\in I$, then
  $C^{\chi \op }=C^{\chi ^{-1}}=C^{\chi }$.

  (b) If $\chi '$ is $p$-finite for all $\chi '\in \cG (\chi )$ and $p\in I$,
  then the Cartan schemes $\cC (\chi )$ and $\cC (\chi \op )$ are equivalent
  via $\varphi _0=\id $ and $\varphi _1(\chi ')=\chi '{}\op $
  for all $\chi '\in \cG (\chi )$.

  (c) If $\chi '$ is $p$-finite for all $\chi '\in \cG (\chi )$ and $p\in I$,
  then the Cartan schemes $\cC (\chi )$ and $\cC (\chi ^{-1})$ are equivalent
  via $\varphi _0=\id $ and $\varphi _1(\chi ')=\chi '{}^{-1}$
  for all $\chi '\in \cG (\chi )$.

  (d) One has $R^{\chi \op }=R^{\chi ^{-1}}=R^\chi $.
\end{propo}

\begin{proof}
	Part (a) follows immediately from Def.~\ref{de:Cartan}. Then
  the definition of $r_p(\chi )$ gives that the relations
  $\chi '\in \cG (\chi )$, $\chi '{}\op \in \cG (\chi \op )$, and
  $\chi '{}^{-1}\in \cG (\chi ^{-1})$ are mutually equivalent, and
  hence (b) and (c) follow from (a).
  The relation $R^{\chi \op }_+=R^\chi _+$ can be obtained by using twisting,
  see \cite[Prop.\,3.9]{inp-AndrSchn02}.
  The equation $R^{\chi ^{-1}}_+=R^\chi _+$ holds since
  for any finite-dimensional \YD module $V$ of diagonal type,
  the bicharacter corresponding to $V^*$ is the inverse of the bicharacter
  corresponding to $V$, and
  $\NA{V^*}$ and $\NA V$ are isomorphic as $\ndZ ^I$-graded algebras.
\end{proof}

Later on we will need functions $\lambda _i$ defined on the group of all
bicharacters. In the next lemma these functions are defined and some of their
properties are determined.

\begin{lemma}
  \label{le:lambda}
  Let $\chi \in \cX $, $p\in I$, and
  $q_{i j}=\chi (\Ndb _i,\Ndb _j)$ for all $i,j\in I$.
  Assume that $\chi $ is $p$-finite.
  Let $c_{pi}=c_{pi}^\chi $ for all $i\in I$.
  For all $i\in I\setminus \{p\}$ define
  $$\lambda _i(\chi )=\qfact{-c_{p i}}{q_{p p}}
  \prod _{s=0}^{-c_{p i}-1}(q_{p p}^s q_{p i}q_{i p}-1). $$
  Then for all $i\in I$ the following equations hold.
  \begin{align}
    \label{eq:lambdasym1}
	\lambda _i(r_p(\chi ))=&\,
    (q_{p p}^{-c_{p i}}q_{p i}q_{i p})^{c_{p i}}\lambda _i(\chi ),\\
    \label{eq:lambdasym2}
    \lambda _i(\chi ^{-1})=&\,
    (-q_{p p}^{-c_{p i}-1}q_{p i}q_{i p})^{c_{p i}}\lambda _i(\chi ).
  \end{align}
\end{lemma}

\begin{proof}
  Let $\bq _{i j}=(r_p(\chi ))(\Ndb _i,\Ndb _j)$ for all $i,j\in I$. Then
  by Eqs.~\eqref{eq:w*chi} and \eqref{eq:bichar} one gets
  \begin{align}
    \label{eq:rmatp}
    \bq _{p p}=q_{p p},\qquad
    \bq _{i p}\bq _{p i}=q_{i p}^{-1}q_{p i}^{-1}q_{p p}^{2c_{p i}} \quad
    \text{for all $i\in I\setminus \{p\}$.}
  \end{align}
%
%
%
%

  Eq.~\eqref{eq:lambdasym2} follows easily from the definition of
  $\lambda _i(\chi )$ using the formulas
  $\chi ^{-1}(\Ndb _i,\Ndb _j)=q_{i j}^{-1}$, where $i,j\in I$.

  By Eq.~\eqref{eq:rmatp} and part (a) of the lemma
  one obtains that
  \begin{align*}
	  \lambda _i(r_p(\chi ))=&\qfact{-c_{p i}}{\bq _{p p}}
    \prod _{s=0}^{-c_{p i}-1}(\bq _{p p}^s \bq _{p i}\bq _{i p}-1)\\
    =&\qfact{-c_{p i}}{q_{p p}}
    \prod _{s=0}^{-c_{p i}-1}(q_{p p}^{2c_{p i}+s} q_{p i}^{-1}q_{i p}^{-1}-1).
  \end{align*}
  If $q_{pp}^{-c_{p i}}q_{i p}q_{p i}=1$,
  then the latter formula is equal to $\lambda _i(\chi )$
  and hence part (b) of the lemma holds.
  Otherwise, since $\chi $ is $p$-finite, one gets
  $q_{p p}^{1-c_{p i}}=1$ and
  \begin{align*}
	  \lambda _i(r_p(\chi ))
	  =&\qfact{-c_{p i}}{q_{p p}}
    \prod _{s=0}^{-c_{p i}-1}q_{p p}^{c_{p i}+s+1}q_{p i}^{-1}q_{i p}^{-1}(1-q_{p p}^{-c_{p i}-1-s} q_{p i}q_{i p})\\
    =&(q_{p p}^{-c_{p i}}q_{p i}q_{i p})^{c_{p i}}q_{p p}^{(-c_{p i})(1-c_{p i})/2}\qfact{-c_{p i}}{q_{p p}}
    \prod _{s=0}^{-c_{p i}-1}(1-q_{p p}^s q_{p i}q_{i p}).
  \end{align*}
  By considering separately the cases where $-c_{p i}$ is even and odd,
  respectively, one can easily check that
  \begin{align}
    \label{eq:qpot}
    \qnum{1-c_{p i}}{q_{p p}}=0, \qfact{-c_{p i}}{q_{p p}}\not=0 \quad \Longrightarrow \quad 
    q_{p p}^{(-c_{p i})(1-c_{p i})/2}=(-1)^{c_{p i}}.
  \end{align}
  Hence part (b) of the lemma follows in the
  case $\qnum{1-c_{p i}}{q_{p p}}=0$, too.
\end{proof}

\section{A not so special Drinfel'd double}
\label{sec:Ddouble}

In this section the Drinfel'd double for a class of graded Hopf
algebras is constructed and some properties are proven.
In the literature,
various definitions of (multiparameter) quantizations of universal
enveloping algebras of semisimple Lie algebras and Lie superalgebras
appear as quotients of a special case of the presented Drinfel'd double.
Maybe the definitions most closest to those in the present paper are
those in \cite[Sec.\,3]{p-KharSaga07},
\cite{p-Pei07}, and
\cite[Def.\,1.5]{p-RadSchnRep06},
which are more special,
and the one in \cite[Sects.\,1.1, 8.1]{p-RadSchn06/67}, which is more general.
Our treatment, similarly to \cite{p-RadSchn06/67},
has the advantage that many combinatorial settings,
mainly on the structure constants attached to some root systems, are removed,
or they are shifted to assumptions on the Weyl groupoid.

\subsection{Construction of the Drinfel'd double}
\label{ssec:constDd}

The construction of a Drinfel'd double \cite[Sect.\,3.2]{b-Joseph}, also
called quantum double \cite[Sect.~8.2]{b-KS},
is based on a skew-Hopf pairing of two
Hopf algebras. We will follow this construction.
Further, we will often work with the category $\ydZI $
of \YD modules over the group algebra $\fie \ndZ ^I$ of
$\ndZ ^I$.
Roughly speaking, the objects of this category are vector spaces equipped with
a left action and left coaction of $\fie \ndZ ^I$ satisfying a compatibility
condition, and morphisms are preserving both the left action and the left
coaction. Precise definitions can be found e.\,g.\ in
\cite[Sect.\,10.6]{b-Montg93} and \cite[Sect.\,1.2]{inp-AndrSchn02}.

We keep the settings from the beginning of Sect.~\ref{sec:prelims}.
Let $I$ be a non-empty finite set and let $\chi \in \cX $.
Let $q_{i j}=\chi (\Ndb _i,\Ndb _j)$ for all $i,j\in I$.
Let $\cU ^{+0}=\fie [K_i,K_i^{-1}\,|\,i\in I]$ and
$\cU ^{-0}=\fie [L_i,L_i^{-1}\,|\,i\in I]$ be two copies of the group algebra
of $\ndZ ^I$.
Let
\begin{align}\label{eq:defV+V-}
  V^+(\chi )\in \lYDcat{\cU ^{+0}},\quad V^-(\chi )\in \lYDcat{\cU ^{-0}}
\end{align}
be $|I|$-dimensional vector spaces over $\fie $
with basis $\{E_i\,|\,i\in I\}$ and
$\{F_i\,|\,i\in I\}$, respectively, such that
the left action $\actl $ and the left coaction $\coal $
of $\cU ^{+0}$ on $V^+(\chi )$ and of $\cU ^{-0}$ on $V^-(\chi )$,
respectively, are determined by the formulas
\begin{align}\label{eq:U+0YD}
	K_i\actl E_j=&\,q_{i j}E_j,& K_i^{-1}\actl E_j=&\,q_{i j}^{-1}E_j,&
  \coal (E_i)=&\,K_i\otimes E_i,\\
  \label{eq:U-0YD}
	L_i\actl F_j=&\,q_{j i}F_j,& L_i^{-1}\actl F_j=&\,q_{j i}^{-1}F_j,&
  \coal (F_i)=&\,L_i\otimes F_i
\end{align}
for all $i,j\in I$.
Let
\begin{align}
  \label{eq:cU+-}
  \cU ^+(\chi )=TV^+(\chi ), \qquad \cU ^-(\chi )=TV^-(\chi )
\end{align}
denote the tensor algebra of $V^+(\chi )$ and $V^-(\chi )$, respectively.
Since $\ydZI $
is a tensor category, the algebras $\cU ^+(\chi )$
and $\cU ^-(\chi )$ are \YD modules over $\cU ^{+0}$ and $\cU ^{-0}$,
respectively.

The main objects of study in this paper are the Drinfel'd double $D(\cV
^+(\chi ),\cV ^-(\chi ))$ of the Hopf algebras
\begin{align}
  \label{eq:cV+-}
  \cV ^+(\chi )=&\,\cU ^+(\chi )\# \cU ^{+0},&
  \cV ^-(\chi )=&\,(\cU ^-(\chi )\# \cU ^{-0})\cop
\end{align}
and quotients of it.
Here $\#$ denotes Radford's biproduct \cite{a-Radford85}, also called
bosonization following Majid's terminology. As an algebra, it
is a smash product, see \cite[Def.\,4.1.3]{b-Montg93}.
In particular,
\begin{gather}
  \label{eq:cV+-rel}
  K_iE_j=q_{i j}E_jK_i,\qquad
  L_iF_j=q_{j i}F_jL_i
\end{gather}
for all $i,j\in I$, and the counits and coproducts of $\cV ^+(\chi )$ and $\cV
^-(\chi )$ are determined by the equations
\begin{gather}
  \left\{
  \begin{aligned}
    \coun (K_i)=&\,1,\quad \coun (E_i)=0, &
    \coun (L_i)=&\,1,\quad \coun (F_i)=0, \\
    \copr (K_i)=&\,K_i\otimes K_i,&
    \copr (L_i)=&\,L_i\otimes L_i,\\
    \copr (K_i^{-1})=&\,K_i^{-1}\otimes K_i^{-1},&
    \copr (L_i^{-1})=&\,L_i^{-1}\otimes L_i^{-1},\\
    \copr (E_i)=&\,E_i\otimes 1+K_i\otimes E_i,&
    \copr (F_i)=&\,1\otimes F_i+F_i\otimes L_i
  \end{aligned}
  \right.
  \label{eq:coprV+-}
\end{gather}
for all $i\in I$. The existence of the antipode follows from
\cite{a-Takeuchi71}.

The algebra $\cU ^+(\chi )$ itself is a braided Hopf algebra, see
Prop.~\ref{pr:cU+} below. A braided Hopf algebra is a Hopf algebra in a
braided (for example \YD[]) category. For further details we refer
to \cite{inp-Takeuchi00}. Moreover, under a connected Hopf algebra we mean a
connected coalgebra in the sense of \cite[Def.\,5.1.5]{b-Montg93}.

\begin{propo}\label{pr:cU+} \cite[Sect.\,2.1]{inp-AndrSchn02}
  The algebra $\cU ^+(\chi )$ is a connected braided Hopf algebra in the \YD
  category $\lYDcat{\cU ^{+0}}$, where the
  left action and the left coaction of $\cU ^{+0}$ on $\cU ^+(\chi )$ are
  determined by the formulas
\begin{gather}\label{eq:U0act}
	K_i\actl E_j=q_{ij}E_j,\quad 
  \coal (E_i)=K_i\otimes E_i
\end{gather}
for $i,j\in I$.
Further, the braiding $c\in \Aut _\fie (\cU ^+(\chi )\otimes \cU ^+(\chi ))$
is the canonical braiding of the category, that is
\begin{align}
  c(E\otimes E')=&E_{(-1)}\actl E'\otimes E_{(0)},&
  c(E_i\otimes E_j)=q_{ij} E_j\otimes E_i
  \label{eq:infbraiding}
\end{align}
for all $i,j\in I$ and $E,E'\in \cU ^+(\chi )$.
The braided coproduct
$\brcopr :\cU ^+(\chi )\to \cU ^+(\chi )\otimes \cU ^+(\chi )$ is defined by
\begin{align}
  \brcopr (E_i)=E_i\otimes 1 +1\otimes E_i \quad \text{for all $i\in I$.}
  \label{eq:brcoprE}
\end{align}
\end{propo}

\begin{remar}\label{re:brcopr}
  The coproduct of $\cV ^+(\chi )$ and the braided coproduct of $\cU ^+(\chi
  )$ are related by the formula
  \begin{align*}
    \copr (E)=E_{(1)}\otimes E_{(2)}=E^{(1)}(E^{(2)})_{(-1)}\ot
    (E^{(2)})_{(0)}\quad  \text{for all $E\in \cU ^+(\chi )$,}
  \end{align*}
  where $\brcopr (E)=E^{(1)}\ot E^{(2)}$.
\end{remar}

In order to form the Drinfel'd double $D(\cV ^+(\chi ),\cV ^-(\chi ))$, one
needs a skew-Hopf pairing
\begin{align*}
  \sHp : \cV ^+(\chi )\times \cV ^-(\chi )\to \fie ,\quad
  (x,y)\mapsto \sHp (x,y)
\end{align*}
of $\cV ^+(\chi )$ and $\cV ^-(\chi )$.
This means, see \cite[Sect.\,3.2.1]{b-Joseph},
that $\sHp $ is a bilinear map satisfying the equations
\begin{align}
  \label{eq:sHp1}
  \sHp (1,y)=&\,\coun (y),& \sHp (x,1)=&\,\coun (x),\\
  \label{eq:sHp2}
  \sHp (xx',y)=&\,\sHp (x',y_{(1)})\sHp (x,y_{(2)}),&
  \sHp (x,yy')=&\,\sHp (x_{(1)},y)\sHp (x_{(2)},y'),\\
  \label{eq:sHp3}
  &\qquad 
  \makebox[0pt][l]{$\sHp (\antip (x),y)=\sHp (x,\antip ^{-1}(y))$}
\end{align}
for all $x,x'\in \cV ^+(\chi )$ and $y,y'\in \cV ^-(\chi )$.
Equivalently, $\sHp $ is a Hopf pairing of $\cV ^+(\chi )$ and
$\cV ^-(\chi )\cop =\cU ^-(\chi )\#\cU ^{-0}$.
For all $\chi \in \cX $
let us fix the skew-Hopf pairing given by the following proposition.

\begin{propo}
  \label{pr:sHpdef}
  (i) There exists a unique skew-Hopf pairing $\sHp $ of $\cV ^+(\chi )$ and
  $\cV ^-(\chi )$ such that for all $i,j\in I$
  \begin{align*}
    \sHp (E_i,F_j)=-\delta _{i,j},\quad
    \sHp (E_i,L_j)=0,\quad
    \sHp (K_i,F_j)=0,\quad
    \sHp (K_i,L_j)=q_{ij}.
  \end{align*}

  (ii) The skew-Hopf pairing $\sHp $ satisfies the equations
  \begin{align*}
    \sHp (EK,FL)=\sHp (E,F)\sHp (K,L)
  \end{align*}
  for all $E\in \cU ^+(\chi )$,
  $F\in \cU ^-(\chi )$, $K\in \cU ^{+0}$, and $L\in \cU ^{-0}$.
\end{propo}

\begin{proof}
  (i)
  First we prove the uniqueness of the pairing. 
  Since $\cV ^+(\chi )$ is generated by the set
  $\{E_i,K_i,K_i^{-1}\,|\,i\in I\}$, the linearity of $\eta $ in the first
  argument and the
  first formula in Eq.~\eqref{eq:sHp2} tell that $\sHp $ is determined by the
  values $\sHp (x,y)$, where
  \begin{align}
    \label{eq:xisgener}
    x\in \{1\}\cup \{K_i,K_i^{-1},E_i\,|\,i\in I\}
  \end{align}
  and $y\in \cV ^-(\chi )$. Since $\copr $ maps the elements of the latter
  set to linear combinations of tensor products of the same elements, see
  Eq.~\eqref{eq:coprV+-}, the linearity of $\sHp $ in the second argument and
  the second formula in Eq.~\eqref{eq:sHp2} yield that $\sHp $ is determined
  by the values $\sHp (x,y)$, where $x$ is as in Rel.~\eqref{eq:xisgener}
  and
  \begin{align}
    \label{eq:yisgener}
    y\in \{1\}\cup \{L_i,L_i^{-1},F_i\,|\,i\in I\}.
  \end{align}
  Further, by Eq.~\eqref{eq:sHp1} and relations
  $K_iK_i^{-1}=1$ and $L_iL_i^{-1}=1$ for all
  $i\in I$ it suffices to consider the case
  \begin{align*}
    x\in \{K_i,E_i\,|\,i\in I\},\quad
    y\in \{L_i,F_i\,|\,i\in I\}.
  \end{align*}
  The numbers $\sHp (x,y)$ for such $x,y$ are
  given in the proposition.

  Now we turn to the proof of the existence. Notice that both $\cV ^+(\chi )$
  and $\cV ^-(\chi )$ are generated by finitely many elements and defined by
  finitely many relations. Using arguments analogous to those in the first
  part of the proof, one obtains that a pairing $\sHp $ satisfying
  Eqs.~\eqref{eq:sHp1} and \eqref{eq:sHp2} exists if the equations
  \begin{align*}
    \sHp (K_iE_j-q_{i j}E_jK_i,y)=0 \text{ for all $y\in \{L_k,F_k\,|\,k\in
    I\}$}
  \end{align*}
  are compatible with the first formula in Eq.~\eqref{eq:sHp2} and equations
  \begin{align*}
    \sHp (x,L_iF_j-q_{j i}F_jL_i)=0 \text{ for all $x\in \{K_k,E_k\,|\,k\in
    I\}$}
  \end{align*}
  are compatible with the second formula in Eq.~\eqref{eq:sHp2}. These are
  easy calculations. Finally, one has to check that Eq.~\eqref{eq:sHp3} holds
  for all $x,y$. Using the fact that $\antip $ and $\antip ^{-1}$ are algebra
  and coalgebra antihomomorphisms, one can reduce the problem to the case when
  $x$ and $y$ are generators. Again in this case the equation can be easily
  shown.

  (ii) Let $E\in \cU ^+(\chi )$, $F\in \cU ^-(\chi )$, $K\in \cU ^{+0}$, and
  $L\in \cU ^{-0}$. By the definition of the coproduct of $\cV ^+(\chi )$ and
  $\cV ^-(\chi )$ one obtains the following equations.
  \begin{align*}
    E_{(1)}K_{(1)}\sHp (E_{(2)}K_{(2)},L)=&\,EK_{(1)}\sHp (K_{(2)},L),\\
    \sHp (EK,F)=&\,\sHp (K,F_{(1)})\sHp (E,F_{(2)})=\coun (K)\sHp (E,F),\\
    \sHp (EK,FL)=&\,\sHp (E_{(1)}K_{(1)},F)\sHp (E_{(2)}K_{(2)},L)\\
    =&\,\sHp (EK_{(1)},F)\sHp (K_{(2)},L)\\
    =&\,\coun (K_{(1)})\sHp (E,F)\sHp (K_{(2)},L)
    =\sHp (E,F)\sHp (K,L).
  \end{align*}
  This proves the proposition.
\end{proof}

\begin{remar}
  One can slightly generalize Prop.~\ref{pr:sHpdef}.
  Let $(a_i)_{i\in I}\in \fie ^I$.
  The proof of the proposition shows that if one
  replaces equation $\sHp (E_i,F_j)=-\delta _{i,j}$ by
  $\sHp (E_i,F_j)=a_i\delta _{i,j}$, then the pairing $\sHp $ still exists
  and is unique. In what follows, we will stick to the setting 
  in Prop.~\ref{pr:sHpdef}.
\end{remar}

The following definition is a combination of Prop.~\ref{pr:sHpdef} and the
definition in \cite[Sect.\,3.2.4]{b-Joseph}.

\begin{defin}\label{de:cU}
  Let $\chi \in \cX $.
  For all $i,j\in I$ let $q_{ij}=\chi (\Ndb _i,\Ndb _j)$.
  Let $\cU (\chi )$ be the Drinfel'd double of $\cV ^+(\chi )$ and $\cV
  ^-(\chi )$ with respect to the skew-Hopf pairing in Prop.~\ref{pr:sHpdef},
  that is $\cU (\chi )$ is the unique Hopf algebra such that
  \begin{enumerate}
    \item $\cU (\chi )=\cV ^+(\chi )\otimes \cV ^-(\chi )$ as a vector space,
    \item the maps $\cV ^+(\chi )\to \cU (\chi )$, $x\mapsto x\otimes 1$
      and $\cV ^-(\chi )\to \cU (\chi )$, $y\mapsto 1\otimes y$ are Hopf
      algebra maps,
    \item the product of $\cU (\chi )$ is given by
      \begin{align}\label{eq:Ddproduct}
        (x\otimes y)(x'\otimes y')=x\sHp (x'_{(1)},\antip (y_{(1)}))x'_{(2)}
        \otimes y_{(2)}\sHp (x'_{(3)},y_{(3)})y'
      \end{align}
      for all $x,x'\in \cV ^+(\chi )$ and $y,y'\in \cV ^-(\chi )$.
  \end{enumerate}
  In what follows, the tensor product sign in elements of
  $\cU (\chi )$ will be omitted.
  Let $\cU ^0(\chi )$ denote the commutative cocommutative
  Hopf subalgebra
  \begin{align}
    \cU ^0(\chi )=\fie [K_i,K_i^{-1},L_i,L_i^{-1}\,|\,i\in I]
    \label{eq:cU0}
\end{align}
of $\cU (\chi )$.
For all $\mu =\sum _{i\in I} m_i\Ndb _i\in \ndZ ^I$ let
\[ K_\mu =\prod _{i\in I} K_i^{m_i}, \quad L_\mu =\prod _{i\in I} L_i^{m_i}.
\]
\end{defin}

Alternatively, one can define the algebra $\cU (\chi )$ in terms of generators
and relations. The equivalence of these definitions is an easy standard
calculation, see for example \cite[Lemma\,3.2.5]{b-Joseph}.

\begin{propo}\label{pr:cUgenrel}
  The algebra $\cU (\chi )$ is generated by
  the elements $K_i$, $K_i^{-1}$, $L_i$, $L_i^{-1}$, $E_i$, and $F_i$,
  where $i\in I$,
  and defined by the relations
  \begin{align}
    XY= YX \quad & \makebox[0pt][l]{for all $X,Y\in \{K_i,K_i^{-1},
    L_i,L_i^{-1}\,|\,i\in I\}$,}
    \label{eq:KLrel}\\
    K_iK_i^{-1}=&\,1, & L_iL_i^{-1}=&\,1,
    \label{eq:KKrel}
  \\
    K_iE_jK_i^{-1}=&\,q_{ij}E_j, & L_iE_jL_i^{-1}=&\,q_{ji}^{-1}E_j,
    \label{eq:KErel}
  \\
    K_iF_jK_i^{-1}=&\,q_{ij}^{-1}F_j, & L_iF_jL_i^{-1}=&\,q_{ji}F_j,
    \label{eq:KFrel}\\
    E_iF_j&\makebox[0pt][l]{$-F_jE_i=\delta _{i,j}(K_i-L_i)$.}
    \label{eq:EFrel}
  \end{align}
\end{propo}

Note that by definition the coalgebra structure
of\/ $\cU (\chi )$ is determined by Eqs.~\eqref{eq:coprV+-}.

\begin{remar}\label{re:cU}
  1. Assume that there exist a symmetrizable generalized
  Cartan matrix $C=(c_{i j})_{i.j\in I}$ with integer entries,
  positive integers $d_i$, $i\in I$, and a number $q\in \fienz $
  which is not a root of $1$, such that
  $$q_{i j}=q^{d_ic_{i j}} \quad \text{for all $i,j\in I$}. $$
  Then the quantized symmetrizable Kac-Moody algebra associated to the matrix
  $C$, see \cite[Def.\,3.2.9]{b-Joseph},\cite[3.1.1]{b-Lusztig93}
  and Rem.~\ref{re:Uqg},
  is a quotient of the algebra
  $\cU (\chi )$ by a Hopf ideal. In the special case when $C$ is of finite
  type, the quantized Kac-Moody algebra is the
  Drinfel'd-Jimbo algebra or quantized enveloping algebra of the semisimple
  Lie algebra corresponding to $C$. See also Rem.~\ref{re:Uqg} and
  Thm.~\ref{th:nondegpair}.

  2. Usually, on the right hand side of Eq.~\eqref{eq:EFrel} a denominator
  appears. This allows an easier consideration of classical limits and
  specialization arguments. In our paper we will neither consider classical
  limits, nor will use specialization. Omitting the denominator we even
  achieve a slight generalization of the traditional setting by admitting the
  cases when $q_{i i}=\pm 1$ for some $i\in I$.

  3. Quantized Lie superalgebras, see \cite[Def.\,2.1]{a-KhorTol91},
  and quantized enveloping algebras for Borcherds superalgebras,
  see \cite{a-BeKaMel98},
  are quotients of algebras of the form $\cU (\chi )$ or $\cU (\chi )\#\fie
  \Gamma $, too, where $\Gamma $ is a finite group and $\#$ denotes
  Radford's biproduct, and
  $\chi =\chi \op $ again has to satisfy some additional conditions depending
  on the underlying Lie superalgebra.

  4. Two-parameter quantum groups, see e.\,g.~\cite{a-BenWith04} and
  \cite{a-BerGaoHu06}, are quotients of algebras of the form $\cU (\chi )$,
  where the definition of $\chi $ needs two parameters. In these examples
  $\chi \not=\chi \op $.
\end{remar}

\begin{remar}\label{re:U0YD}
  By Eqs.~\eqref{eq:cU+-} and \eqref{eq:KErel} the vector space $V^+(\chi )$
  and the algebra $\cU ^+(\chi )$ are \YD modules over $\cU ^0(\chi )$.
\end{remar}

The algebra $\cU (\chi )$ admits a unique $\ndZ ^I$-grading
\begin{equation}
\begin{gathered}
  \cU (\chi )=\bigoplus _{\mu \in \ndZ ^I}\cU (\chi )_\mu ,\\
  1\in \cU (\chi )_0,\quad
  \cU (\chi )_\mu \cU (\chi )_\nu \subset
  \cU (\chi )_{\mu +\nu } \quad \text{for all $\mu ,\nu \in \ndZ ^I$,}
\end{gathered}
  \label{eq:Zngrading}
\end{equation}
such that $K_i,K_i^{-1},L_i,L_i^{-1}\in \cU (\chi )_0$,
$E_i\in \cU (\chi )_{\Ndb _i}$, and
$F_i\in \cU (\chi )_{-\Ndb _i}$ for all $i\in I$.
For all $\mu =\sum _{i\in I}m_i\Ndb _i\in \ndZ ^I$ let
$|\mu |=\sum _{i\in I}m_i\in \ndZ $.
The decomposition
\begin{equation}
	\cU (\chi )=\bigoplus _{m\in \ndZ }\cU (\chi )_m,\quad
	\text{where}\quad
	\cU (\chi )_m=\bigoplus _{\mu \in \ndZ ^I:|\mu |=m}\cU (\chi )_\mu ,
  \label{eq:Zgrading}
\end{equation}
gives a $\ndZ $-grading of $\cU (\chi )$ called the \textit{standard
grading}.
For any $\ndZ ^I$-homogeneous subquotient $\cU '$ of $\cU (\chi )$ the
notation $\cU '_\Ndb $ and $\cU '_m$ for the homogeneous components of
degree $\Ndb \in \ndZ ^I$ and $m\in \ndZ $, respectively, will be used.
Note that in general the subspaces $\cU '_0$ for $0\in \ndZ ^I$
and $\cU '_0$ for $0\in \ndZ $ are different.

\begin{propo}\label{pr:algiso}
  Let $\chi \in \cX $.
  
  (1) Let $\ula =(a_i)_{i\in I}\in (\fienz )^I$.
	There exists a unique algebra automorphism
	$\varphi _{\ula }$ of $\cU (\chi )$ such that
\begin{equation}
	\varphi _{\ula }(K_i)=K_i,\,\,
	\varphi _{\ula }(L_i)=L_i,\,\,
	\varphi _{\ula }(E_i)=a_iE_i,\,\,
	\varphi _{\ula }(F_i)=a_i^{-1}F_i.
  \label{eq:cUauto1}
\end{equation}

(2) Let $\tau $ be a permutation of $I$ and let $\hat{\tau }$ be the
automorphism of $\ndZ ^I$ given by $\hat{\tau }(\Ndb _i)=\Ndb _{\tau (i)}$
for all $i\in I$. Then there exists a unique algebra isomorphism $\varphi
_\tau :\cU (\chi )\to \cU (\hat{\tau }^*\chi )$ such that
\begin{equation}
\begin{aligned}
    \varphi _\tau (K_i)=&K_{\tau (i)},&
    \varphi _\tau (L_i)=&L_{\tau (i)},\\
    \varphi _\tau (E_i)=&E_{\tau (i)},&
    \varphi _\tau (F_i)=&F_{\tau (i)}.
\end{aligned}
	\label{eq:cUtauiso}
\end{equation}

(3) For all $m\in \ndZ $ there exists a unique algebra automorphism $\varphi _m$
of $\cU (\chi )$ such that
\begin{equation}
\begin{aligned}
\varphi _m(K_i)=&K_i,& \varphi _m(L_i)=&L_i,\\
\varphi _m(E_i)=&K_i^mL_i^{-m}E_i,&
\varphi _m(F_i)=&F_iK_i^{-m}L_i^m.
\end{aligned}
	\label{eq:cUauto2}
\end{equation}

(4) There exists a unique algebra automorphism
$\phi _1$ of $\cU (\chi )$
such that
\begin{equation}
\begin{aligned}
	\phi _1(K_i)=&K_i^{-1},& \phi _1(L_i)=&L_i^{-1},\\
	\phi _1(E_i)=&F_iL_i^{-1},&
	\phi _1(F_i)=&K_i^{-1}E_i.
\end{aligned}
	\label{eq:cUauto3}
\end{equation}

(5) There is a unique algebra isomorphism
$\phi _2:\cU (\chi )\to \cU (\chi ^{-1})$
such that
\begin{equation}
\begin{aligned}
	\phi _2(K_i)=&K_i,& \phi _2(L_i)=&L_i,&
	\phi _2(E_i)=&F_i,& \phi _2(F_i)=&-E_i.
\end{aligned}
	\label{eq:cUiso1}
\end{equation}

(6) The algebra map
$\phi _3:\cU (\chi )\to \cU (\chi \op )\cop $
defined by the formulas
\begin{equation}
\begin{aligned}
	\phi _3(K_i)=&L_i,& \phi _3(L_i)=&K_i,&
	\phi _3(E_i)=&F_i,& \phi _3(F_i)=&E_i.
\end{aligned}
	\label{eq:cUiso2}
\end{equation}
is an isomorphism of Hopf algebras.

(7) There is a unique algebra antiautomorphism
$\phi _4$ of $\cU (\chi )$
such that
\begin{align}
	\phi _4(K_i)=&K_i,& \phi _4(L_i)=&L_i,&
	\phi _4(E_i)=&F_i,& \phi _4(F_i)=&E_i.
	\label{eq:cUantiauto}
\end{align}
\end{propo}

\begin{proof}
  One has to check the compatibility of the definitions with the defining
  relations of $\cU (\chi )$, which is easy. The bijectivity can be proven by
  writing down the inverse map explicitly, see also Prop.~\ref{pr:commiso}
  below.
  In case of the map $\phi _3$ note that one has $\copr (\phi _3(X))=\phi
  _3(X_{(2)})\ot \phi _3(X_{(1)})$ for all generators $X$ of $\cU (\chi )$
  which implies that $\phi _3$ is a coalgebra antihomomorphism.
\end{proof}

\begin{corol} 
  The antipode of $\cU (\chi )$ can be obtained as $\antip =\phi _1\phi
  _4\varphi _{\ula }$, where $a _i=-1$ for all $i\in I$.
  \label{co:antipU}
\end{corol}

\begin{proof}
  Eqs.~\eqref{eq:coprV+-} imply that
  \begin{equation}
    \begin{aligned}
      \antip (E_i)=&\,-K_i^{-1}E_i,&
      \antip (F_i)=&\,-F_iL_i^{-1},\\
      \antip (K_i)=&\,K_i^{-1},&
      \antip (L_i)=&\,L_i^{-1}
      \end{aligned}
    \label{eq:antipU}
  \end{equation}
  for all $i\in I$.
  Then it is easy to check that the equation
  $\antip =\phi _1\phi _4\varphi _{\ula }$
  holds on the generators of $\cU (\chi )$.
  Thus the corollary follows since
  both sides of the equation are algebra antihomomorphisms.
\end{proof}

The description of $\varphi _1$ below will be used in the proof of
Lemma~\ref{le:phiS+}.

\begin{lemma}
  Let $\ula \in (\fienz )^I$ with $a_i=q_{i i}^{-1}$ for all $i\in I$.
  Then
  \begin{align*}
    \varphi _1\varphi _{\ula }(E)=\chi (\mu ,\mu )EK_\mu L_\mu ^{-1}
  \end{align*}
  for all $\mu \in \ndZ ^I$ and all $E\in \cU (\chi )_\mu $.
  \label{le:varphi1}
\end{lemma}

\begin{proof}
  Check the formula for the generators of $\cU (\chi )$, and that it is
  compatible with the product of $\ndZ ^I$-homogeneous elements.
\end{proof}

\begin{propo}\label{pr:commiso}
	The isomorphisms in Prop.~\ref{pr:algiso} satisfy the following
	relations.

  (i) Let $\ula ,\ulb \in (\fienz )^I$ and $m,n\in \ndZ $. Then
  $\varphi _{\ula }\varphi _{\ulb }=\varphi _{\ulc }$,
  $\varphi _m\varphi _{\ula }=\varphi _{\ula }\varphi _m$, and
  $\varphi _m\varphi _n=\varphi _{m+n}$, where $c_i=a_i b_i$ for
  all $i\in I$.

  (ii) Let $\ula \in (\fienz )^I$ and $i\in \{1,2,3,4\}$. Then
  $\varphi _{\ula }\phi _i=\phi _i\varphi _{\ulb }$, where
  $b_i=a_i^{-1}$ for all $i\in I$.
  
  (iii) Let $m\in \ndZ $ and $\ula \in (\fienz )^I$ with
  $a_i=q_{i i}^{-2m}$ for all $i\in I$. Then
  $\varphi _m\phi _1=\phi _1\varphi _m\varphi _{\ula }$,
  $\varphi _m\phi _2=\phi _2\varphi _{-m}\varphi _{\ula }^{-1}$,
  $\varphi _m\phi _3=\phi _3\varphi _m\varphi _{\ula }$, and
  $\varphi _m\phi _4=\phi _4\varphi _{-m}$.

  (iv) Let $\ula ,\ulb \in (\fienz )^I$ with
  $a_i=q_{i i}$ and $b_i=-1$ for all $i\in I$.
  Then $\phi _1^2=\varphi _{-1}\varphi _{\ula }$,
  $\phi _2^2=\varphi _{\ulb }$, and $\phi _3^2=\phi _4^2=\id $,

  (v) Let $\ula ,\ulb \in (\fienz )^I$ with
  $a_i=q_{i i}^{-1}$ and $b_i=-1$ for all $i\in I$.
  Then $\phi _1\phi _2=
  \phi _2\phi _1\varphi _1\varphi _{\ula }\varphi _{\ulb }$,
  $\phi _1\phi _3=\phi _3\phi _1\varphi _{\ula }$, and
  $\phi _1\phi _4=\phi _4\phi _1\varphi _1\varphi _{\ula }^2$.

  (vi) Let $\ulb \in (\fienz )^I$ with
  $b_i=-1$ for all $i\in I$. Then
  $\phi _2\phi _3=\phi _3\phi _2\varphi _{\ulb }$,
  $\phi _2\phi _4=\phi _4\phi _2\varphi _{\ulb }$, and
  $\phi _3\phi _4=\phi _4\phi _3$.
\end{propo}

\begin{proof}
  Evaluate both sides of the equations on the generators of $\cU (\chi )$ and
  compare the results.
\end{proof}

For arbitrary $X,Y\in \cU (\chi )$ and $K\in \cU ^0(\chi )$ let
\begin{align*}
[X,Y]=XY-YX,\quad 
K\actl X:=(\ad K)X=K_{(1)}X\antip (K_{(2)}),
\end{align*}
where $\ad $ denotes left adjoint action.
This interpretation of the operation $\actl $ is consistent with
Rels.~\eqref{eq:U+0YD}, \eqref{eq:U-0YD}, \eqref{eq:KErel} and
\eqref{eq:KFrel}.
For the computation of commutation relations in $\cU (\chi )$ later on the
following lemma will be useful. The proof is a direct consequence of
Prop.~\ref{pr:cUgenrel}.

\begin{lemma}
	\label{le:Urel}
	Let $p\in I$ and $X\in \cU (\chi )$.
	Then
	\begin{align}
    [K_p^{-1}E_p,X]=&\,K_p^{-1}(E_pX -(K_p\actl X)E_p)=K_p^{-1}\,(\ad E)X,
		\label{eq:Emukomm}\\
		[X ,F_pL_p^{-1}]=&\,(X F_p-F_p(L_p^{-1}\actl X ))L_p^{-1}.
		\label{eq:Fmukomm}
	\end{align}
\end{lemma}

The next claim is known as the \textit{triangular decomposition of}
$\cU (\chi )$.

\begin{propo} \label{pr:trdeccU}
  The multiplication maps
  \begin{align*}
    \mul :\cU ^+(\chi )\otimes \cU ^0(\chi )\otimes \cU ^-(\chi )\to &\,
    \cU (\chi ),\\
    \mul :\cU ^-(\chi )\otimes \cU ^0(\chi )\otimes \cU ^+(\chi )\to &\,
    \cU (\chi )
  \end{align*}
  are isomorphisms of $\ndZ ^I$-graded vector spaces.
\end{propo}

\begin{proof}
  The first map is an isomorphism by construction of $\cU (\chi )$.
  The proof for the second one is also standard. It relies mainly on the fact
  that Eq.~\eqref{eq:Ddproduct} has an ``inverse'' which tells that
  \begin{align*}
    xy=\sHp (x_{(1)},y_{(1)})y_{(2)}x_{(2)}\sHp (x_{(3)},\antip (y_{(3)}))
  \end{align*}
  for all $x\in \cV ^+(\chi )$ and $y\in \cV ^-(\chi )$.
\end{proof}

\subsection{Kashiwara maps}
\label{ssec:Kashiwara}

For quantized enveloping algebras $U_q(\lag )$ of semisimple Lie algebras
$\lag $ Kashiwara \cite{a-Kashi91} constructed certain skew-derivations of the
upper triangular part $U_q^+(\lag )$ by considering commutators in $U_q(\lag
)$. This construction can be generalized to our setting.

\begin{lemma}\label{le:commEFi}
	For all $i\in I$ there exist unique linear maps
	$\derK _i,\derL _i\in \End _\fie (\cU ^+(\chi ))$ such that
	\begin{align*}
		[E,F_i]=\derK _i(E)K_i-L_i\derL _i(E)\quad
		\text{for all $E\in \cU ^+(\chi )$.}
	\end{align*}
	The maps $\derK _i,\derL _i\in \End _\fie (\cU ^+(\chi ))$ are
	skew-derivations.
	More precisely,
	\begin{gather}
	  \derK _i(1)=\derL _i(1)=0,\quad 
	  \derK _i(E_j)=\derL _i(E_j)=\delta _{i,j},\label{eq:derKL1}\\
	\begin{aligned}
		\derK _i(EE')=&\derK _i(E)(K_i\actl E')+E\derK _i(E'),\\
		\derL _i(EE')=&\derL _i(E)E'+(L_i^{-1}\actl E)\derL _i(E')
	\end{aligned}
		\label{eq:derKL2}
	\end{gather}
	for all $i,j\in I$ and $E,E'\in \cU ^+(\chi )$.
\end{lemma}

\begin{proof}
  The triangular decomposition of $\cU (\chi )$ and Rels.~\eqref{eq:KErel}
  imply uniqueness of
	the maps $\derK _i$ and $\derL _i$.
	Since $\cU ^+(\chi )$ is the free algebra generated by 
	$V^+(\chi )$, the existence of the maps $\derK _i$ and $\derL _i$
	follows from Rels.~\eqref{eq:EFrel} and the formula
	\begin{align*}
		[EE',F_i]=&[E,F_i]E'+E[E',F_i]\\
		=&\big(\derK _i(E)K_i-L_i\derL _i(E)\big)E'
		+E\big(\derK _i(E')K_i-L_i\derL _i(E')\big)\\
		=&\big(\derK _i(E)(K_i\actl E')+E\derK _i(E')\big)K_i\\
		&-L_i\big(\derL _i(E)E'+(L_i^{-1}\actl E)\derL _i(E')\big),
	\end{align*}
	where $E,E'\in \cU ^+(\chi )$. This also proves the last part of the
	lemma.
\end{proof}

The maps $\derK _i$, $\derL _i$ with $i\in I$ are variations of
Lusztig's maps $r_i$ and ${}_i r$ in \cite{b-Lusztig93},
as the second part of Lemma~\ref{le:commEFi} shows.

\begin{lemma} \label{le:commder}
  Let $i,j\in I$ and $E\in \cU ^+(\chi )$. Then
  \begin{gather}
    \derK _i(K_j\actl E)=q_{ji}K_j\actl (\derK _i(E)),\,\,
    \derK _i(L_j\actl E)=q_{ij}^{-1}L_j\actl (\derK _i(E)),
    \label{eq:derK_KL}\\
    \derL _i(K_j\actl E)=q_{ji}K_j\actl (\derL _i(E)),\quad
    \derL _i(L_j\actl E)=q_{ij}^{-1}L_j\actl (\derL _i(E)),
    \label{eq:derL_KL}\\
    \derK _i\derL _j=\derL _j\derK _i.
    \label{eq:commder}
  \end{gather}
\end{lemma}

\begin{proof}
  The first equation in \eqref{eq:derK_KL} holds for $E=1$ and
  $E=E_m$, where $m\in I$, by Eqs.~\eqref{eq:derKL1} and
  \eqref{eq:KErel}.
  Further, for $E,E'\in \cU ^+(\chi )$ one gets
  \begin{align*}
    \derK _i(K_j\actl (EE'))=&\derK _i((K_j\actl E)(K_j\actl E'))\\
    =&\derK _i(K_j\actl E)(K_iK_j\actl E')+(K_j\actl E)\derK _i(K_j\actl E').
  \end{align*}
  Thus the first equation in~\eqref{eq:derK_KL} follows by
  induction on the $\ndZ $-degree of $E$
  using Eq.~\eqref{eq:KLrel}. The second equation
  in~\eqref{eq:derK_KL} and the equations in~\eqref{eq:derL_KL}
  can be obtained similarly.

  Now we prove Eq.~\eqref{eq:commder}.
  Using Eq.~\eqref{eq:derKL1} one obtains that
  \begin{align*}
    \derK _i\derL _j(E)=\derL _j\derK _i(E)=0
  \end{align*}
  for all $i,j\in I$ and $E\in \{1,E_m\,|\,m\in I\}$.
  Further,
  \begin{align*}
    (\derK _i&\derL _j-\derL _j\derK _i)(EE')\\
    =&\derK _i(\derL _j(E)E'+(L_j^{-1}\actl E)\derL _j(E'))
    -\derL _j(\derK _i(E)(K_i\actl E')+E\derK _i(E'))\\
    =&(\derK _i\derL _j-\derL _j\derK _i)(E)(K_i\actl E')
    +(L_j^{-1}\actl E)(\derK _i\derL _j-\derL _j\derK _i)(E')
  \end{align*}
  for all $E,E'\in \cU ^+(\chi )$
  because of the first part of the lemma. Thus the claim follows by induction.
\end{proof}

The following statement gives a
characterization of a class of ideals of $\cU (\chi )$ compatible with the
triangular decomposition of $\cU (\chi )$. This proposition seems to be new
even for multiparameter quantizations of Kac-Moody algebras, see
\cite[Prop.\,3.4]{p-KharSaga07}.

\begin{propo}\label{pr:goodideals}
  Let $\cI ^+\subset \cU ^+(\chi )\cap \ker \coun $ and
  $\cI ^-\subset \cU ^-(\chi )\cap \ker \coun $ be
  a (not necessarily $\ndZ $-graded) ideal of $\cU ^+(\chi )$ and 
  $\cU ^-(\chi )$, respectively.
  Then the following statements are equivalent.

  \begin{enumerate}
    \item (Triangular decomposition of $\cU (\chi )/(\cI ^++\cI ^-)$)
      The multiplication map
      $\mul :\cU ^+(\chi )\otimes \cU ^0(\chi )\otimes \cU ^-(\chi )\to
      \cU (\chi )$ induces an isomorphism
      $$\cU ^+(\chi )/\cI ^+\otimes \cU ^0(\chi )\otimes
      \cU ^-(\chi )/\cI ^-\to \cU (\chi )/(\cI ^++\cI ^-)$$
      of vector spaces.
    \item \label{en:gideal2}
      The following equation holds.
      \begin{align*}
        \cU (\chi )\cI ^+\cU (\chi )+\cU (\chi )\cI ^-\cU (\chi )=
        \cI ^+\cU ^0(\chi )\cU ^-(\chi )+
        \cU ^+(\chi )\cU ^0(\chi )\cI ^-.
      \end{align*}
    \item The vector spaces $\cI ^+\cU ^0(\chi )\cU ^-(\chi )$ and
    $\cU ^+(\chi )\cU ^0(\chi )\cI ^-$ are ideals of $\cU (\chi )$.
    \item
      For all $X\in \cU ^0(\chi )$ and $i\in I$ one has
      \begin{align*}
        X\actl \cI ^+\subset &\,\cI ^+,&
        X\actl \cI ^-\subset &\,\cI ^-,\\
        \derK _i(\cI ^+)\subset &\cI ^+,& \derL _i(\cI ^+)\subset &\cI ^+,\\
        \derK _i(\phi _4(\cI ^-))\subset &\phi _4(\cI ^-),&
        \derL _i(\phi _4(\cI ^-))\subset &\phi _4(\cI ^-).
      \end{align*}
  \end{enumerate}
\end{propo}

\begin{proof}
  (1)$\Leftrightarrow $(2). The map in part~(1) is surjective by the triangular
  decomposition of $\cU (\chi )$. The injectivity of the map in part~(1)
  means precisely that part~(2) is true.

  (3)$\Rightarrow $(2). This follows from the triangular decomposition of $\cU
  (\chi )$.

  (2)$\Rightarrow $(4). By the triangular decomposition of $\cU (\chi )$, the
  linear map
  \begin{align*}
    \zeta ^+:\cU ^+(\chi )\cU ^0(\chi )\cU ^-(\chi )\to \cU ^+(\chi )\cU
    ^0(\chi ),\quad
    abc\mapsto ab\varepsilon (c),
  \end{align*}
  where $a\in \cU ^+(\chi )$, $b\in \cU ^0(\chi )$, and $c\in \cU ^-(\chi )$,
  is a well-defined surjective linear map from $\cU (\chi )$ to $\cU ^+(\chi
  )\cU ^0(\chi )$.
  The equation in part~(2) and the standing assumption $\cI ^-\subset \ker
  \coun $ imply that
  \begin{align*}
    \zeta ^+\big(
    \cU (\chi )\cI ^+\cU (\chi )+\cU (\chi )\cI ^-\cU (\chi )\big)=
    \cI ^+\cU ^0(\chi ).
  \end{align*}
  Since $\zeta ^+|_{\cU ^+(\chi )\cU ^0(\chi )}$ is injective and
  $\cI ^+\subset \cU ^+(\chi )$, the above equation implies that
  \begin{align}
    \label{eq:cI+U0}
    \cU ^+(\chi )\cU ^0(\chi )\cap \big(
    \cU (\chi )\cI ^+\cU (\chi )+\cU (\chi )\cI ^-\cU (\chi )\big)= &\,
    \cI ^+\cU ^0(\chi ),\\
    \label{eq:cI+}
    \cU ^+(\chi )\cap \big(
    \cU (\chi )\cI ^+\cU (\chi )+\cU (\chi )\cI ^-\cU (\chi )\big)= &\,
    \cI ^+.
  \end{align}
  Now let $X\in \cU ^0(\chi )$ and $E\in \cI ^+$. Since $\cU ^0(\chi )$ is a
  group algebra, for the proof of the first two relations in part~(4)
  one can assume that $X$ is a group-like element.
  Then $XEX^{-1}\in \cU ^+(\chi )$ by Eqs.~\eqref{eq:KErel}, and hence
  $XEX^{-1}\in \cI ^+$ by Eq.~\eqref{eq:cI+}.
  Similarly one gets $X\actl \cI ^-\subset \cI ^-$ for all
  $X\in \cU ^0(\chi )$.

  Let again $E\in \cI ^+$. By Lemma~\ref{le:commEFi} and Eq.~\eqref{eq:cI+U0}
  one has
  \begin{align*}
    \derK _i(E)K_i-L_i\derL _i(E)\in \cI ^+\cU ^0(\chi ).
  \end{align*}
  By triangular decomposition of $\cU (\chi )$ and Eqs.~\eqref{eq:KErel}
  one obtains that
  $\derK _i(\cI ^+)\subset \cI ^+$ and
  $\derL _i(\cI ^+)\subset \cI ^+$.
  Finally, notice that the pair $(\cI ^+,\cI ^-)$ can be replaced by the pair
  $(\phi _4(\cI ^-),\phi _4(\cI ^+))$, and by definition of $\phi _4$ the
  equation in part~(2) holds for $(\cI ^+,\cI ^-)$ if and only if it holds for
  $(\phi _4(\cI ^-),\phi _4(\cI ^+))$. This symmetry yields immediately the
  remaining relations in part~(4).

  (4)$\Rightarrow $(3).
  We prove first that $\cI ^+\cU ^0(\chi )\cU ^-(\chi )$
  is an ideal of $\cU (\chi )$.
  Since $\cI ^+$ is a right ideal of $\cU ^+(\chi )$,
  triangular decomposition of $\cU (\chi )$ implies that
  $\cI ^+\cU ^0(\chi )\cU ^-(\chi )=\cI ^+\cU (\chi )$
  is a right ideal of $\cU (\chi )$.
  Since $\cI ^+$ is a left ideal of $\cU ^+(\chi )$, one obtains that
  \begin{align*}
    \cU ^+(\chi )\cI ^+\cU ^0(\chi )\cU ^-(\chi )\subset
    \cI ^+\cU ^0(\chi )\cU ^-(\chi ).
  \end{align*}
  Let $X\in \{K_i,L_i,F_i\,|\,i\in I\}$. The relation
  \begin{align*}
    X\cI ^+\in \cI ^+\cU ^0(\chi )\cU ^-(\chi )
  \end{align*}
  follows immediately from Lemma~\ref{le:commEFi} and the relations in
  part~(4). Thus $\cI ^+\cU ^0(\chi )\cU ^-(\chi )$ is also a left ideal of
  $\cU (\chi )$. By the same arguments one gets that $\phi _4(\cI ^-)\cU
  ^0(\chi )\cU ^-(\chi )$ is an ideal of $\cU (\chi )$. Apply the algebra
  antiautomorphism $\phi _4$ to this fact to obtain that 
  $\cU ^+(\chi )\cU ^0(\chi )\cI ^-$ is an ideal of $\cU (\chi )$.
\end{proof}

\begin{remar}\label{re:cI+}
Assume that $\cI =(\cI ^+,\cI ^-)$ is an ideal of $\cU (\chi )$ as in
Prop.~\ref{pr:goodideals}.
Because of Prop.~\ref{pr:goodideals}\eqref{en:gideal2}, see
Eq.~\eqref{eq:cI+}, the ideals $\cI ^+\subset \cU ^+(\chi )$ and
$\cI ^-\subset \cU ^-(\chi )$
are uniquely determined by $\cI $.
Explicitly, $\cI ^+=\cU ^+(\chi )\cap \cI $ and
$\cI ^-=\cU ^-(\chi )\cap \cI $.
\end{remar}

\begin{remar}
  Similarly to the proof of Prop.~\ref{pr:goodideals}
  one can show that the claim of Prop.~\ref{pr:goodideals}(4)
  holds if and only if
  the multiplication map
  $\mul :\cU ^-(\chi )\otimes \cU ^0(\chi )\otimes \cU ^+(\chi )
  \to \cU (\chi )$ induces an isomorphism
  $$\cU ^-(\chi )/\cI ^-\otimes \cU ^0(\chi )\otimes
  \cU ^+(\chi )/\cI ^+\to \cU (\chi )/(\cI ^++\cI ^-)$$
  of vector spaces.
\end{remar}

Let $\pi _1:\cU ^+(\chi )\to V^+(\chi )=\cU ^+(\chi )_1$ denote the
surjective $\ndZ $-graded map, see Eqs.~\eqref{eq:Zgrading},
with $\pi _1(E_i)=E_i$ for all $i\in I$.
For all $j\in I$ let $E_j^*\in  V^+(\chi )^*$ be the linear functional
with $E_j^*(E_i)=\delta _{i,j}$ for all $i\in I$. Recall the braided
Hopf algebra structure of $\cU ^+(\chi )$ given in Prop.~\ref{pr:cU+}.

\begin{lemma}
  \label{le:derK=copr}
  For all $i\in I$ and $E\in \cU ^+(\chi )$
  \begin{align*}
    \derK _i(E)=(\id \otimes E_i^*\circ \pi _1)\brcopr (E),\quad
    \derL _i(E)=(E_i^*\circ \pi _1\otimes \id )\brcopr (E),
  \end{align*}
  where $\cU ^+(\chi )\otimes \fie $ and $\fie \otimes \cU ^+(\chi )$
  are identified with $\cU ^+(\chi )$.
\end{lemma}

\begin{proof}
  Both equations hold for $E\in \fie \oplus V^+(\chi )$ by
  Eqs.~\eqref{eq:derKL1}. One checks easily that for the right hand sides of
  the equations analogous formulas as Eqs.~\eqref{eq:derKL2} hold.
\end{proof}

\begin{corol}\label{co:Hopfideal}
  Let $\cI ^+\subset \bigoplus _{m=2}^\infty \cU ^+(\chi )_m$ be a
  \YD submodule (with respect to $\cU ^0(\chi )$, see Rem.~\ref{re:U0YD},)
  and a biideal of $\cU ^+(\chi )$, \emph{i.\,e.} $\cI $ is an ideal
  and a braided coideal of $\cU ^+(\chi )$. Then $\cI ^+
  \cU ^0(\chi )\cU ^-(\chi )$ is a Hopf ideal of $\cU (\chi )$.
\end{corol}

\begin{proof}
  The assumptions yield that
  $\cI ^+\cU ^0(\chi )\cU ^-(\chi )$ is a coideal of $\cU (\chi )$.
  Lemma~\ref{le:derK=copr} gives that $\derK _i(\cI ^+)\subset \cI ^+$ and
  $\derL _i(\cI ^+)=\cI ^+$ for all $i\in I$. Further,
  $X\actl \cI ^+\subset \cI ^+$ by assumption, and hence
  Prop.~\ref{pr:goodideals} (4)$\Rightarrow $(3) implies that
  $\cI ^+\cU ^0(\chi )\cU^-(\chi )$ is an ideal of $\cU (\chi )$.
  Finally, $\cI ^+\cU ^0(\chi )$ is a Hopf ideal of
  $\cU ^+(\chi )\cU ^0(\chi )$
  by a result of Takeuchi, see \cite[Lemma\,5.2.10]{b-Montg93} and the
  corresponding remark in \cite[Sect.\,2.1]{inp-AndrSchn02}. Thus $\cI ^+\cU
  ^0(\chi )\cU ^-(\chi )$ is a Hopf ideal of $\cU (\chi )$.
\end{proof}

\subsection{Some relations of $\cU (\chi )$}
\label{ssec:cUrels}

Let $\chi \in \cX $ and let $p\in I$.
For any $i\in I\setminus \{p\}$ let $E^+_{i,0(p)}=E^-_{i,0(p)}=E_i$,
and for all $m\in \ndN $ define recursively
\begin{align}
  \label{eq:defEpim}
  E^+_{i,m+1(p)}=&\,E_pE^+_{i,m(p)}-(K_p\actl E^+_{i,m(p)})E_p,\\
  \label{eq:defEmim}
  E^-_{i,m+1(p)}=&\,E_pE^-_{i,m(p)}-(L_p\actl E^-_{i,m(p)})E_p.
\end{align}
In connection with the variable $p$ we will also write
$E^+_{i,m}$ for $E^+_{i,m(p)}$ and $E^-_{i,m}$ for $E^-_{i,m(p)}$,
where $m\in \ndN _0$. If somewhere $p$ has
to be replaced by another variable, then we will not use this abbreviation.
Observe that $E^-_{i,m}=\phi _3\phi _2(E^+_{i,m})$, where
$E^+_{i,m}\in \cU ( (\chi ^{-1})\op )$ and
$E^-_{i,m}\in \cU (\chi )$.

Using Eq.~\eqref{eq:m+1choosen} and induction on $m$ one can show that
the explicit form of the elements $E^\pm _{i,m}$ is as follows.
\begin{align}
	E^+_{i,m}=&\sum _{s=0}^m(-1)^sq_{pi}^sq_{pp}^{s(s-1)/2}
	\qchoose{m}{s}{q_{pp}}E_p^{m-s}E_iE_p^s, \label{eq:Eim+}\\
	E^-_{i,m}=&\sum _{s=0}^m(-1)^sq_{ip}^{-s}q_{pp}^{-s(s-1)/2}
	\qchoose{m}{s}{q_{pp}^{-1}}E_p^{m-s}E_iE_p^s. \label{eq:Eim-}
\end{align}

\begin{lemma}
	\label{le:E+E-}
	For all $i\in I\setminus \{p\}$ and all $m\in \ndN _0$
	\begin{align*}
		\fie E^+_{i,m+1}=&\,\fie (E^+_{i,m}E_p-(L_iL_p^m\actl E_p)E^+_{i,m}),\\
		\fie E^-_{i,m+1}=&\,\fie (E^-_{i,m}E_p-(K_iK_p^m\actl E_p)E^-_{i,m}).
	\end{align*}
\end{lemma}

\begin{proof}
	The first equation of the lemma follows immediately from
	$$L_iL_p^m\actl E_p=q_{pi}^{-1}q_{pp}^{-m}E_p,\quad
  K_p\actl E^+_{i,m}=q_{pi}q_{pp}^mE^+_{i,m}.$$
	To get the second equation, apply $\phi _3\phi _2$ to the first one.
\end{proof}

\begin{lemma}
  \label{le:brcoprE}
  (i) For all $m\in \ndN _0$
  \begin{align*}
	  \brcopr (E_p^m)=\sum _{r=0}^m\qchoose{m}{r}{q_{pp}}E_p^r\otimes E_p^{m-r}.
  \end{align*}

  (ii) For all $i\in I\setminus \{p\}$ and all $m\in \ndN _0$
  \begin{align*}
	  \brcopr (E^+_{i,m})=&E^+_{i,m}\otimes 1
	  +\sum _{r=0}^m\qchoose{m}{r}{q_{pp}}
	  \prod _{s=1}^r(1-q_{pp}^{m-s}q_{pi}q_{ip})E_p^r\otimes E^+_{i,m-r},
    \displaybreak[2]\\
	  \brcopr (E^-_{i,m})=&1\otimes E^-_{i,m}
    +\sum _{r=0}^mq_{pi}^r\qchoose{m}{r}{q_{pp}}
	  \prod _{s=1}^r(1-q_{pp}^{s-m}q_{pi}^{-1}q_{ip}^{-1})E^-_{i,m-r}\otimes
    E_p^r.
  \end{align*}
\end{lemma}

\begin{proof}
  Use Prop.~\ref{pr:cU+}, Eq.~\eqref{eq:m+1choosen}, and induction on $m$.
\end{proof}

Lemmata~\ref{le:brcoprE}, \ref{le:commEFi} and \ref{le:derK=copr}
can be used to obtain commutation relations which will be essential
to determine Lusztig isomorphisms between Drinfel'd doubles.

\begin{corol}\label{co:derspecel}
	For all $m\in \ndN _0$ and $i,j\in I\setminus \{p\}$
  {\allowdisplaybreaks
	\begin{align*}
		\derK _p(E_p^m)=&\,\qnum{m}{q_{pp}}E_p^{m-1},&
		\derK _i(E_p^m)=&\,0,\\
		\derL _p(E_p^m)=&\,\qnum{m}{q_{pp}}E_p^{m-1},&
		\derL _i(E_p^m)=&\,0,\\
		\derK _j(E^+_{i,m})=&\,\delta _{i,j}\prod
		_{s=0}^{m-1}(1-q_{pp}^sq_{pi}q_{ip}) E_p^m,&
		\derK _p(E^+_{i,m})=&\,0,\\
		\derK _p(E^-_{i,m})=&\,\qnum{m}{q_{pp}}
		(1-q_{pp}^{1-m}q_{pi}^{-1}q_{ip}^{-1}) E^-_{i,m-1},&
		\derK _j(E^-_{i,m})=&\delta _{i,j}\delta _{m,0}1,\\
		\derL _p(E^+_{i,m})=&\,\qnum{m}{q_{pp}}
		(1-q_{pp}^{m-1}q_{pi}q_{ip}) E^+_{i,m-1},&
		\derL _j(E^+_{i,m})=&\delta _{i,j}\delta _{m,0}1,\\
		\derL _j(E^-_{i,m})=&\,\delta _{i,j}q_{pi}^m\prod
		_{s=0}^{m-1}(1-q_{pp}^{-s}q_{pi}^{-1}q_{ip}^{-1}) E_p^m,&
		\derL _p(E^-_{i,m})=&\,0.
	\end{align*}
  }
\end{corol}

\begin{corol}
  \label{co:rootvrel}
  For all $m\in \ndN _0$ and all $i\in I\setminus \{p\}$
  \begin{align*}
    [E_p^m,F_p]=&\qnum{m}{q_{pp}}(q_{pp}^{1-m}K_p-L_p)E_p^{m-1},\\
    [E^+_{i,m},F_p]=&\qnum{m}{q_{pp}}(q_{pp}^{m-1}q_{pi}q_{ip}-1)L_pE^+_{i,m-1},\\
    [E^+_{i,m},F_i]=&q_{ip}^{-m}\prod _{s=0}^{m-1}(1-q_{pp}^sq_{pi}q_{ip})K_iE_p^m
    -\delta _{m,0}L_i,\\
    [E^-_{i,m},F_p]=&
    q_{pp}^{1-m}\qnum{m}{q_{pp}}(1-q_{pp}^{1-m}q_{pi}^{-1}q_{ip}^{-1})K_pE^-_{i,m-1},\\
    [E^-_{i,m},F_i]=&\delta _{m,0}K_i
    -q_{pi}^m\prod _{s=0}^{m-1}(1-q_{pp}^{-s}q_{pi}^{-1}q_{ip}^{-1})L_iE_p^m.
  \end{align*}
  Moreover, $[E^+_{i,m},F_j]=[E^-_{i,m},F_j]=0$
  for all $m\in \ndN _0$ and $i,j\in I\setminus \{p\}$ with $i\not=j$.
\end{corol}

For $i\in I\setminus \{p\}$ and $m\in \ndN _0$ let
\begin{align}
  \label{eq:F+def}
	F^+_{i,m}=&\phi _3(E^+_{i,m}),&
	F^-_{i,m}=&\phi _3(E^-_{i,m}),
\end{align}
where $E^+_{i,m}$, $E^-_{i,m}$ are elements of $\cU ^+(\chi \op )$.
In particular,
\begin{equation}
\begin{aligned}
  F^+_{i,0}=&F_i,& F^+_{i,m+1}=&F_pF^+_{i,m}-(L_p\actl F^+_{i,m})F_p,\\
  F^-_{i,0}=&F_i,& F^-_{i,m+1}=&F_pF^-_{i,m}-(K_p\actl F^-_{i,m})F_p
\end{aligned}
  \label{eq:F+}
\end{equation}
for all $i\in I$ and $m\in \ndN _0$.

By induction on $m$ one can show easily the following.

\begin{lemma}
  \label{le:phiE+}
  Let $i\in I\setminus \{p\}$.
  For all $\ula \in (\fienz )^I$, $n\in \ndZ $ and $m\in \ndN _0$
  \begin{align*}
    \varphi _{\ula }(E^\pm _{i,m})\in &\,\fienz E^\pm _{i,m},&
    \varphi _n(E^\pm _{i,m})\in &\, \fienz K_p ^{m n} L_p^{-m n} K_i^n L_i^{-n}
    E^\pm _{i,m},\\
    \varphi _{\ula }(F^\pm _{i,m})\in &\,\fienz F^\pm _{i,m},&
    \varphi _n(F^\pm _{i,m})\in &\,\fienz K_p ^{-m n} L_p^{m n} K_i^{-n} L_i^n
    F^\pm _{i,m}.
  \end{align*}
  Further, for all $m\in \ndN _0$
  \begin{align*}
    \phi _1(E^\pm _{i,m})\in &\,\fienz F^\pm _{i,m}L_i^{-1}L_p^{-m},&
    \phi _1(F^\pm _{i,m})\in &\,\fienz K_i^{-1}K_p^{-m}E^\pm _{i,m},\\
    \phi _2(E^\pm _{i,m})=&\,F^\mp _{i,m},&
    \phi _2(F^\pm _{i,m})=&\,(-1)^{m+1}E^\mp _{i,m},\\
    \phi _3(E^\pm _{i,m})=&\,F^\pm _{i,m},&
    \phi _3(F^\pm _{i,m})=&\,E^\pm _{i,m},\\
    \phi _4(E^\pm _{i,m})\in &\,\fienz F^\mp _{i,m},&
    \phi _4(F^\pm _{i,m})\in &\,\fienz E^\mp _{i,m}.
  \end{align*}
\end{lemma}

\begin{lemma}
  \label{le:E+F+1}
  For all $i\in I\setminus \{p\}$ and all $m,n\in \ndN _0$  with $m\ge n$
  \begin{align*}
    [E^+_{i,m},F^+_{i,n}]=
    (-1)^n q_{i p}^{n-m}q_{p p}^{n(n-m)}&\prod _{s=0}^{n-1}\qnum{m-s}{q_{p p}}
    \prod _{s=0}^{m-1}(1-q_{p p}^s q_{p i}q_{i p})\times \\
    &\qquad (K_p^n K_i-\delta _{m,n}L_p^n L_i)E_p^{m-n}.
  \end{align*}
\end{lemma}

\begin{proof}
  Proceed by induction on $n$. The formula for $n=0$ was proven in
  Cor.~\ref{co:rootvrel}.
  Assume now that $m,n\in \ndN _0$ and $n<m$. Then
  \begin{align}
    [E^+_{i,m},F^+_{i,n+1}]=&
    [E^+_{i,m},F_p F^+_{i,n}-q_{p p}^n q_{i p}F^+_{i,n}F_p] \notag \\
    =&[E^+_{i,m},F_p]F^+_{i,n}+F_p[E^+_{i,m},F^+_{i,n}] \label{eq:E+F+1}\\
    &-q_{p p}^n q_{i p}[E^+_{i,m},F^+_{i,n}]F_p
    -q_{p p}^n q_{i p}F^+_{i,n}[E^+_{i,m},F_p]. \notag
  \end{align}
  Let
  $$\al _{m,n}=(-1)^n q_{i p}^{n-m}q_{p p}^{n(n-m)}
  \prod _{s=0}^{n-1}\qnum{m-s}{q_{p p}}
  \prod _{s=0}^{m-1}(1-q_{p p}^s q_{p i}q_{i p}).$$
    By induction hypothesis and Cor.~\ref{co:rootvrel} the sum of the
    second and third summands of the expression~\eqref{eq:E+F+1} is
  \begin{align*}
    &\al _{m,n}(F_p K_p^n K_i E_p^{m-n}
    -q_{p p}^n q_{i p}K_p^n K_i E_p^{m-n} F_p)\\
    &\qquad =-q_{p p}^n q_{i p}\al _{m,n}K_p^n K_i(E_p^{m-n} F_p-F_p E_p^{m-n})\\
    &\qquad =-q_{p p}^n q_{i p}\al _{m,n}\qnum{m-n}{q_{pp}}K_p^n K_i
    (q_{pp}^{1-m+n}K_p-L_p)E_p^{m-n-1}\\
    &\qquad =\al _{m,n+1}K_p^n K_i (K_p-q_{pp}^{m-n-1}L_p)E_p^{m-n-1}.
  \end{align*}
  Similarly, the sum of the first and fourth summands is equal to
  \begin{align*}
    &\qnum{m}{q_{p p}}(q_{p p}^{m-1}q_{p i}q_{i p}-1)
    (L_p E^+_{i,m-1}F^+_{i,n}-q_{p p}^n q_{i p} F^+_{i,n} L_p E^+_{i,m-1})\\
    &\quad =\qnum{m}{q_{p p}}(q_{p p}^{m-1}q_{p i}q_{i p}-1)
    L_p[E^+_{i,m-1},F^+_{i,n}]\\
    &\quad =\qnum{m}{q_{p p}}(q_{p p}^{m-1}q_{p i}q_{i p}-1)
    \al _{m-1,n}L_p (K_p^n K_i-\delta _{m-1,n} L_p^n L_i )E_p^{m-1-n}\\
    &\quad =q_{pp}^{m-n-1}\al _{m,n+1}L_p
    (K_p^n K_i-\delta _{m,n+1}L_p^n L_i)E_p^{m-1-n}.
  \end{align*}
  The latter two formulas imply the statement of the lemma for the expression
  $[E^+_{i,m},F^+_{i,n+1}]$.
\end{proof}

\begin{lemma}
  \label{le:E+F+2}
  Let $m,n\in \ndN _0$ and $i,j\in I\setminus \{p\}$ such that $i\not=j$. Then
  \begin{align*}
    [E^+_{i,m},F^+_{j,n}]=0.
  \end{align*}
\end{lemma}

\begin{proof}
  Proceed by induction on $n$.
  For $n=0$ the lemma follows from Cor.~\ref{co:rootvrel}.
  Assume that $n\in \ndN _0$ with $[E^+_{i,m},F^+_{j,n}]=0$
  for all $m\in \ndN _0$. Then
  \begin{align*}
    &[E^+_{i,m},F^+_{j,n+1}]=[E^+_{i,m},F_p F^+_{j,n}
    -q_{p p}^n q_{j p}F^+_{j,n}F_p] \\
    &\quad =[E^+_{i,m},F_p]F^+_{j,n}+F_p[E^+_{i,m},F^+_{j,n}] \\
    &\qquad \quad -q_{p p}^n q_{j p}[E^+_{i,m},F^+_{j,n}]F_p
    -q_{p p}^n q_{j p}F^+_{j,n}[E^+_{i,m},F_p]\\
    &\quad =[E^+_{i,m},F_p]F^+_{j,n}-q_{p p}^n q_{j p}F^+_{j,n}[E^+_{i,m},F_p]\\
    &\quad =\qnum{m}{q_{p p}}(q_{p p}^{m-1}q_{p i}q_{i p}-1)
    (L_pE^+_{i,m-1}F^+_{j,n}-q_{p p}^n q_{j p}F^+_{j,n}L_pE^+_{i,m-1})\\
    &\quad =\qnum{m}{q_{p p}}(q_{p p}^{m-1}q_{p i}q_{i p}-1)
    L_p[E^+_{i,m-1},F^+_{j,n}]=0
  \end{align*}
  by induction hypothesis.
\end{proof}

\begin{defin}
  Let $p\in I$. Let $\cU ^+_{+p}(\chi )$ and $\cU ^+_{-p}(\chi )$
  denote the subalgebra (with unit) of $\cU ^+(\chi )$ generated by
  $\{E^+_{j,m}\,|\,j\in I\setminus \{p\},m\in \ndN _0\}$ and
  $\{E^-_{j,m}\,|\,j\in I\setminus \{p\},m\in \ndN _0\}$, respectively.
\end{defin}

\begin{lemma}
  \label{le:U+pcoid}
  Let $p\in I$.
  
  (i) The algebras $\cU ^+_{+p}(\chi )$, $\cU ^+_{-p}(\chi )$
  are \YD submodules of $\cU ^+(\chi )$ in $\lYDcat{\cU ^0(\chi )}$.
  
  (ii) One has $\cU ^+_{+p}(\chi )\subset \ker \derK _p$ and
  $\cU ^+_{-p}(\chi )\subset \ker \derL _p$.

  (iii) The algebra $\cU ^+_{+p}(\chi )$ is a left
  coideal of $\cU ^+(\chi )$ and the algebra $\cU ^+_{-p}(\chi )$
  is a right coideal of $\cU ^+(\chi )$, that is
  \begin{align*}
    \brcopr (\cU ^+_{+p}(\chi ))\subset &\,\cU ^+(\chi )\otimes
    \cU ^+_{+p}(\chi ),\\
    \brcopr (\cU ^+_{-p}(\chi ))\subset &\,\cU ^+_{-p}(\chi 
	)\otimes \cU ^+(\chi ).
  \end{align*}
  
  (iv) Let $X\in \cU ^+_{+p}(\chi )$ and $Y\in \cU ^+_{-p}(\chi )$.
  Then
  \begin{align*}
    E_pX-(K_p\actl X)E_p\in \cU ^+_{+p}(\chi ),\quad
    E_pY-(L_p\actl Y)E_p\in \cU ^+_{-p}(\chi ).
  \end{align*}
\end{lemma}

\begin{proof}
  Part (i) follows from the definition of $E^{\pm }_{i,m}$ and
  Eqs.~\eqref{eq:U0act} and \eqref{eq:KErel}. Part~(ii) can be
  obtained from Eqs.~\eqref{eq:derKL1}, \eqref{eq:derKL2}, part~(i), and
  Cor.~\ref{co:derspecel}. Part~(iii) follows immediately from
  Lemma~\ref{le:brcoprE}(ii). Finally, consider the first relation of
  part~(iv). First of all, this relation holds for all generators $X$ of
  $\cU ^+_{+p}(\chi )$ by definition of $E^+_{i,m}$. It is easy to
  see that if it holds for $X=X_1$ and $X=X_2$, then it also holds for
  $X=X_1X_2$. Thus the equation holds for all $X\in \cU ^+_{+p}(\chi )$.
  The second equation in part~(iv) can be proven similarly.
\end{proof}

\begin{lemma}
  \label{le:U+dec}
  For all $p\in I$ the multiplication maps
  \begin{align*}
    \mul :\cU ^+_{+p}(\chi )\otimes \fie [E_p]\to \cU ^+(\chi ),\quad
    \mul :\cU ^+_{-p}(\chi )\otimes \fie [E_p]\to \cU ^+(\chi )
  \end{align*}
  are isomorphisms of \YD modules, where $\fie [E_p]$ denotes the polynomial ring
  in one variable $E_p$.
\end{lemma}

\begin{proof}
  We will prove surjectivity and injectivity of the first
  multiplication map. The proof for the second goes analogously.

  The surjectivity of the first map follows from the facts that
  \begin{itemize}
    \item $E_i\in \cU ^+_{+p}(\chi )\fie [E_p]$ for all $i\in I$,
    \item $\cU ^+_{+p}(\chi )\fie [E_p]$ is a subalgebra of
      $\cU (\chi )$ by Lemma~\ref{le:U+pcoid}(i),(iv).
  \end{itemize}
  Now we prove injectivity. Since $\cV ^+(\chi )$ is a $\ndZ
  ^I$-graded Hopf algebra with $\cV ^+(\chi )_{m\Ndb _p}=E_p^m\cU
  ^{+0}$ for all $m\in \ndN _0$, there is a unique $\ndZ ^I$-graded
  retraction $\pi _{(p)}$ of the Hopf algebra embedding $\iota
  _{(p)}:\fie [E_p]\#\cU ^{+0}\to \cV ^+(\chi )$. Thus $\cV ^+(\chi )$
  is a right $\fie [E_p]\#\cU ^{+0}$-Hopf module, see
  \cite[Def.\,1.9.1]{b-Montg93}, where the right module structure
  comes from multiplication and the right coaction is $(\id \ot \pi
  _{(p)})\copr $. Further, the elements of $\cU ^+_{+p}(\chi )$ are
  right coinvariant by Lemmata~\ref{le:brcoprE}(ii),
  \ref{le:U+pcoid}(i) and Rem.~\ref{re:brcopr}. Thus $\mul $ is
  injective by the fundamental theorem of Hopf modules
  \cite[1.9.4]{b-Montg93}.
\end{proof}

\begin{lemma}
  \label{le:E+=E-}
  Let $p\in I$ such that $\chi $ is $p$-finite.
  Let $i\in I\setminus \{p\}$ and $c_{p i}=c_{p i}^\chi $. Then
  \begin{align*}
    E^+_{i,1-c_{pi}}-E^-_{i,1-c_{pi}}\in \fie E_iE_p^{1-c_{p i}}, \quad
    F^+_{i,1-c_{pi}}-F^-_{i,1-c_{pi}}\in \fie F_iF_p^{1-c_{p i}}.
  \end{align*}
  If $\qfact{1-c_{p i}}{q_{pp}}\not=0$, then both expressions are zero.
\end{lemma}

\begin{proof}
  Lemma~\ref{le:U+dec}
  and Eqs.~\eqref{eq:Eim+}, \eqref{eq:Eim-} imply
  that there exist $a_s\in \fie $, where $1\le s\le 1-c_{pi}$, such that
  \begin{align*}
    E^-_{i,1-c_{pi}}=&E^+_{i,1-c_{pi}}
    +\sum _{s=1}^{1-c_{pi}}a_sE^+_{i,1-c_{pi}-s}E_p^s.
  \end{align*}
  Apply $\derK _p$ to this expression. By Cor.~\ref{co:derspecel}
  one gets $\derK _p(E^-_{i,1-c_{pi}})=0$ because of the definition
  of $c_{pi}$. By Lemmata~\ref{le:commEFi} and \ref{le:U+pcoid}(ii),
  \begin{align*}
    \derK _p(E^-_{i,1-c_{pi}})
    =&\sum _{s=1}^{1-c_{pi}}a_sE^+_{i,1-c_{pi}-s}\derK _p(E_p^s)
    =\sum _{s=1}^{1-c_{pi}}a_s\qnum{s}{q_{pp}}E^+_{i,1-c_{pi}-s}E_p^{s-1}.
  \end{align*}
  If $1\le s\le -c_{p i}$,
  then $\qnum{s}{q_{p p}}\not=0$ by definition of
  $c_{p i}$. Therefore Lemma~\ref{le:U+dec}
  implies that $a_s=0$ whenever $1\le s\le -c_{pi}$. Further, if
  $\qfact{1-c_{p i}}{q_{pp}}\not=0$,
  then also $a_{1-c_{p i}}=0$ by the same reason.
  This gives the statement of the lemma for
  $E^+_{i,1-c_{pi}}-E^-_{i,1-c_{pi}}$. The statement for
  $F^+_{i,1-c_{pi}}-F^-_{i,1-c_{pi}}$ follows from this by applying the
  isomorphism $\phi _3$ and using Lemma~\ref{le:phiE+}.
\end{proof}

\section{Nichols algebras of diagonal type}
\label{sec:Nichdiag}

In this section some facts about Nichols algebras $\Nich (V)$
of \YD modules $V\in \lYDcat{H}$ are recalled, where $H$ is a Hopf algebra.
These (braided Hopf) algebras are named by
W.~Nichols who initiated the study of them \cite{a-Nichols78}.
More details can be found \emph{e.\,g.} in
\cite[Sect.\,2.1]{inp-AndrSchn02} and
\cite{inp-Takeuchi05}.
Here it will be shown that
the Drinfel'd double $\cU (\chi )$
admits a natural quotient which is the Drinfel'd double
of the Hopf algebras $\Nich (V^+(\chi ))\#\cU ^{+0}$ and
$\Nich (V^-(\chi ))\#\cU ^{-0}$. These results generalize the corresponding
statements in \cite[Sect.~3.1]{b-Joseph}.

\begin{defin}\label{de:Nichols}
  Let $H$ be a Hopf algebra and $V\in \lYDcat{H}$ a finite-dimensional
  vector space over $\fie $. The tensor algebra
  $TV$ is a braided Hopf algebra in the \YD category $\lYDcat{H}$,
	where the coproduct is defined by
	$$ \brcopr (v)=v\otimes 1+1\otimes v\quad \text{for all $v\in V$}.$$
	Let $\cS $ be a maximal one among all braided coideals of $TV$
  contained in $\bigoplus _{n\ge 2}T^nV$, that is,
  \begin{align*}
	  \brcopr (\cS )\subset \cS \otimes TV+TV\otimes \cS .
  \end{align*}
	Then $\cS $ is uniquely determined and it is a braided Hopf ideal of $TV$
  in the category $\lYDcat{H}$ (see also the arguments in the proof of
  Lemma~\ref{le:S+homog}).
	The quotient braided Hopf algebra $\Nich (V)=TV/\cS $ is termed the
	\textit{Nichols algebra of} $V$. If $H$ is the group algebra of an abelian
	group and $V$ is semisimple, then $\Nich (V)$ is called a
	\textit{Nichols algebra of diagonal type}.
\end{defin}

The following two statements have analogs for arbitrary Hopf algebras $H$ and
(finite-dimensional) \YD modules $V\in \lYDcat{H}$. For convenience we will
state the versions needed in this paper and also give short proofs.

\begin{lemma}
  \label{le:S+homog}
  Let $\chi \in \cX $.
	The maximal coideal $\cS ^+(\chi )$ of $\cU ^+(\chi )\in \lYDcat{\cU ^0(\chi
  )}$ from Def.~\ref{de:Nichols} is a \YD submodule of $\cU ^+(\chi )$ and is
  a homogeneous ideal of $\cU ^+(\chi )$
  with respect to the $\ndZ ^I$-grading.
\end{lemma}

\begin{proof}
  Since the action and coaction of $\cU ^0(\chi )$ on $\cU ^+(\chi )$ are
  homogeneous with respect to the standard grading,
  the smallest \YD submodule of $\cU ^+(\chi )$ containing $\cS ^+(\chi )$
  is a coideal of
  $\cU ^+(\chi )$ consisting of elements of degree at least $2$.
  By maximality of $\cS ^+(\chi )$ the coideal $\cS ^+(\chi )$ is a \YD
  submodule of $\cU ^+(\chi )$.

	The coproduct $\brcopr $ is a homogeneous map of degree $0$. It is easy to
	see that for any coideal $\cI \subset \bigoplus _{n=2}^\infty \cU ^+(\chi
  )_n$ the vector space $\bigoplus _{\mu \in \ndZ ^I}\pr _\mu (\cI )\supset \cI
  $ is a coideal of $\cU ^+(\chi )$, where $\pr _\mu $ is the homogeneous
  projection onto the homogeneous subspace of $\cU ^+(\chi )$ of degree $\mu
  \in \ndZ ^I$. By the maximality assumption one obtains that $\cS ^+(\chi
  )=\bigoplus _{\mu \in \ndZ ^I} \pr _\mu (\cS ^+(\chi ))$.
\end{proof}

  The Nichols algebra
  $\cU ^+(\chi )/\cS ^+(\chi )$ is denoted usually by $\Nich (V^+(\chi ))$.
  Later on, following the standard notation for quantized enveloping algebras,
  it will be more convenient to write $U^+(\chi )$ instead of
  $\Nich (V^+(\chi ))$.
  The coideal structure of $\cS ^+(\chi )$ induces a
  $U^+(\chi )$-bicomodule structure on $\cU ^+(\chi )$.
  The left and right coactions can be defined by
  \begin{align}
    \label{eq:Nichbicom}
    \lcoaS (X)=(\Pi \ot \id )\brcopr ,\qquad
    \rcoaS (X)=(\id \ot \Pi )\brcopr ,
  \end{align}
  where $\Pi :\cU ^+(\chi )\to U^+(\chi )$ is the canonical
  surjection of braided Hopf algebras.

\begin{propo}
	\label{pr:Nicholschar}
	Let $X\in \cU ^+(\chi )$. The following are equivalent.
	\begin{enumerate}
		\item $X\in \cS ^+(\chi )$.
		\item $\coun (X)=0$ and $\derK _p(X)\in \cS ^+(\chi )$ for all $p\in
			I$.
		\item $\coun (X)=0$ and $(\Pi \ot \pi _1)\brcopr (X)=0$.
		\item $\coun (X)=0$ and $\derL _p(X)\in \cS ^+(\chi )$ for all $p\in
			I$.
		\item $\coun (X)=0$ and $(\pi _1\ot \Pi )\brcopr (X)=0$.
	\end{enumerate}
\end{propo}

\begin{proof}
  Implications (1)$\Rightarrow $(2) and (1)$\Rightarrow $(4) follow
  from Lemma~\ref{le:derK=copr} and the assumption $\cS ^+(\chi )\subset
  \bigoplus _{n\ge 2}T^n V^+(\chi )$.
  Lemma~\ref{le:derK=copr} also yields the implications
  (2)$\Rightarrow $(3) and (4)$\Rightarrow $(5).
  We are content with giving a proof for the
  implication (3)$\Rightarrow $(1), the one for (5)$\Rightarrow $(1) being
  similar.

  We finish the proof with showing (3)$\Rightarrow $(1). The proof of
  (5)$\Rightarrow $(1) is similar. Suppose that (3) holds.  Since $\cS
  ^+(\chi )$ is $\ndZ $-homogeneous with respect to the standard
  grading of $\cU ^+(\chi )$ by Lemma~\ref{le:S+homog} and
  $\brcopr $, $\Pi $ and $\pi _1$ are $\ndZ $-homogeneous maps,
  one can assume that $X$ is $\ndZ $-homogeneous.
  Since $\Pi (1)=1$, (3) implies that the
  $\ndZ$-degree of $X$ is at least $2$. Let $C$ be the left
  $U^+(\chi )$-subcomodule of $\cU ^+(\chi )$
  generated by $X$.
  Then $C$ is $\ndZ $-graded, since
  $\brcopr $ is $\ndZ $-homogeneous.
  One gets
  \begin{align*}
    \Pi (X^{(1)})\ot (\Pi \ot \pi _1)\brcopr (X^{(2)})=&\,
    \Pi (X^{(1)})\ot \Pi (X^{(2)})\ot \pi _1(X^{(3)})\\
    =&\, \brcopr (\Pi (X^{(1)}))\ot \pi _1(X^{(2)})=0.
  \end{align*}
  Let $C^+=\{Y-\coun (Y)1\,|\,Y\in C\}$. Then $X\in C^+$, and by the above
  equation
  all elements of $C^+$ satisfy (3). Hence
  $C^+\subset \bigoplus _{n\ge 2}\cU ^+(\chi )_n$. Further,
  \begin{align*}
    \brcopr (C^+)\subset (\cS ^+(\chi )+C^+)\ot
    \cU ^+(\chi )+\cU ^+(\chi )\ot C^+
  \end{align*}
  and hence $\cS ^+(\chi )+C^+$ is a coideal of $\cU ^+(\chi )$. By maximality
  of $\cS ^+(\chi )$ one obtains that $X\in C^+\subset \cS ^+(\chi )$ which
  proves statement~(1).
\end{proof}

Prop.~\ref{pr:Nicholschar} yields a convenient characterization of the ideal
$\cS ^+(\chi )$.

\begin{propo}
	\label{pr:Nicholschar2}
	The following ideals of $\cU ^+(\chi )$ coincide.
  \begin{enumerate}
    \item The ideal $\cS ^+(\chi )$.
    \item Any maximal element in the set of all ideals $\cI ^+$ of $\cU
      ^+(\chi )$ with
      \begin{gather*}
        \coun (\cI ^+)=\{0\},\qquad
        \derK _p(\cI ^+)\subset \cI ^+\quad
        \text{for all $p\in I$.}
      \end{gather*}
    \item Any maximal element in the set of all ideals $\cI ^+$ of $\cU
      ^+(\chi )$ with
      \begin{gather*}
        \coun (\cI ^+)=\{0\},\qquad
        \derL _p(\cI ^+)\subset \cI ^+\quad
        \text{for all $p\in I$.}
      \end{gather*}
  \end{enumerate}
\end{propo}

\begin{proof}
  Prop.~\ref{pr:Nicholschar} implies that $\cS ^+(\chi )$ satisfies the
  properties of (2) and (3). It remains to show that any ideal in (2)
  respectively (3) coincides with $\cS ^+(\chi )$.
  We give an indirect proof for the ideals in (2). The ideals in (3) can be
  treaten in a similar way.

  Let $\cI ^+$ be maximal as in (2).
  Since $\derK _p$ is homogeneous of degree $-1$ with respect to the standard
  grading of $\cU ^+(\chi )$, the vector space
  $\bigoplus _{n=0}^\infty \pi _n(\cI ^+)$ becomes an ideal of $\cU ^+(\chi )$
  containing $\cI ^+$ and satisfying the conditions in (2).
  Thus the maximality of $\cI ^+$ implies that $\cI ^+$ is homogeneous with
  respect to the standard grading. Further, the assumptions in (2)
  imply that
  $\cI ^+\subset \bigoplus _{n=2}^\infty \cU ^+(\chi )_n$.
  By a similar argument, using also Prop.~\ref{pr:Nicholschar}(1)$\Rightarrow
  $(2), one obtains that $\cI ^+$ contains $\cS ^+(\chi )$.
  Assume now that $\cI ^+\not=\cS ^+(\chi )$. Let $E\in \cI ^+$ be a
  homogeneous element of minimal degree, say $n$, with $E\notin \cS ^+(\chi )$.
  Then $n\ge 2$, and
  $\derK _p(E)\in \cI ^+\cap \cU ^+(\chi )_{n-1}=\cS ^+(\chi )\cap \cU
  ^+(\chi )_{n-1}$ for all $p\in I$. Hence $E\in \cS ^+(\chi )$ by
  Prop.~\ref{pr:Nicholschar}(2)$\Rightarrow $(1). This is a contradiction.
\end{proof}

Besides the properties in Lemma~\ref{le:S+homog}, the braided Hopf
ideal $\cS ^+(\chi )$ has the following additional symmetries.

\begin{lemma}\label{le:phiS+}
	Let $\chi \in \cX $. For all $m\in \ndZ $
  \begin{gather}
    \varphi _m(\cS ^+(\chi ))\cU ^0(\chi )=
    \cS ^+(\chi )\cU ^0(\chi ), \label{eq:varphimS+}\\
	  \phi _2(\cS ^+(\chi ^{-1}))=\phi _3(\cS ^+(\chi \op ))
    =\phi _4(\cS ^+(\chi )),\label{eq:phi2S+}\\
    \phi _1(\cS ^+(\chi ))\cU ^0(\chi )=
    \phi _4(\cS ^+(\chi ))\cU ^0(\chi ).\label{eq:phi1S+}
  \end{gather}
\end{lemma}

\begin{proof}
  Lemmata~\ref{le:varphi1}, \ref{le:S+homog} and Prop.~\ref{pr:commiso}(i)
  imply Eq.~\eqref{eq:varphimS+}.
  Since $\cS ^+(\chi )$ is a braided Hopf ideal of $\cU ^+(\chi )$,
  $\cS ^+(\chi )\cU ^0(\chi )$
  is a Hopf ideal of $\cU ^+(\chi )\cU ^0(\chi )$.
  Thus Eq.~\eqref{eq:phi1S+} follows from Prop.~\ref{pr:commiso}
  and Cor.~\ref{co:antipU}.
  
	By Prop.~\ref{pr:commiso} and Lemma~\ref{le:S+homog} it remains to
	prove that
  \begin{gather}
		\phi _3\phi _4(\cS ^+(\chi))\subset \cS ^+(\chi \op ),\quad
		\phi _2\phi _4(\cS ^+(\chi))\subset \cS ^+(\chi ^{-1}).
    \label{eq:phiS+}
  \end{gather}
	We show the
	first formula in \eqref{eq:phiS+}. The proof of the other one is
  similar.

	The proof is based on Prop.~\ref{pr:Nicholschar}. For brevity write
	$\phi = \phi _3\phi _4$.
	Consider the maps $\derK _p\circ \phi $ and
	$\phi \circ \derK _p$ as linear maps from $\cU ^+(\chi )$ to
	$\cU ^+(\chi \op )$. Lemma~\ref{le:commEFi} and
	Prop.~\ref{pr:algiso} imply that for all $p,i\in I$ and $X,Y\in \cU
	^+(\chi )$
	\begin{align*}
		\phi (K_p\actl X)=&\,
		\phi (K_pXK_p^{-1})=
		L_p^{-1}\phi (X)L_p=L_p^{-1}\actl \phi (X),\\
		\phi (\derK _p(E_i))=&\,
		\derL _p(\phi (E_i))=\delta _{p,i},\\
		\phi (\derK _p(XY))=&\, (L_p^{-1}\actl \phi (Y))\phi (\derK _p(X))
		+\phi (\derK _p(Y))\phi (X),\\
		\derL _p(\phi (XY))=&\, \derL _p(\phi (Y))\phi (X)
		+(L_p^{-1}\actl \phi (Y))\derL _p(\phi (X)).
	\end{align*}
	Hence for all $p\in I$
	\begin{align*}
		\phi _3\phi _4\circ \derK _p=\derL _p\circ \phi _3\phi _4.
	\end{align*}
	Thus the first formula in Eq.~\eqref{eq:phiS+} holds by
	Prop.~\ref{pr:Nicholschar}.
\end{proof}

Now the algebra $U(\chi )$ can be defined.

\begin{propo}\label{pr:Uchi}
	Let $\cS ^-(\chi )=\phi _4(\cS ^+(\chi ))$.
	The vector space
  $$ \cS (\chi )= \cS ^+(\chi )\cU ^0(\chi )\cU ^-(\chi )
  +\cU ^+(\chi )\cU ^0(\chi )\cS ^-(\chi )$$
	is a Hopf ideal of $\cU (\chi )$. The quotient Hopf algebra
	$\cU (\chi )/\cS (\chi )$ will be denoted by $U(\chi )$.
\end{propo}

\begin{proof}
  By Prop.~\ref{pr:goodideals}(4)$\Rightarrow $(2),
  Lemma~\ref{le:S+homog}, and Prop.~\ref{pr:Nicholschar},
  $\cS (\chi )$ is an ideal of $\cU (\chi )$.
  Using additionally Def.~\ref{de:Nichols},
  $\cS ^+(\chi )\cU ^0(\chi )\cU ^-(\chi )$ is a Hopf ideal of $\cU (\chi )$.
  Similarly,
  $\cU ^0(\chi )\cS ^+(\chi \op )$ is a Hopf ideal of
  $\cU ^0(\chi )\cU ^+(\chi \op )$, and
  hence
  $\cU ^0(\chi )\cS ^-(\chi )=\phi _3(\cU ^0(\chi )\cS ^+(\chi \op ))$, see
  Lemma~\ref{le:phiS+}, is a Hopf ideal of $\cU ^0(\chi )\cU ^-(\chi )$ by
  Prop.~\ref{pr:algiso}(6). Therefore
  $\cU ^+(\chi )\cU ^0(\chi )\cS ^-(\chi )$ is a Hopf ideal of $\cU (\chi )$.
\end{proof}

\begin{remar}\label{re:Uqg}
  Suppose that $\chi \in \cX $ is symmetric, \emph{i.\,e.}
  $\chi =\chi \op $.
  Then $K_pL_p$ is for all $p\in I$ a central group-like element of the Hopf
  algebras $\cU (\chi )$ and $U(\chi )$. In the example in Rem.~\ref{re:cU}.1
  the quantized symmetrizable Kac-Moody algebra is precisely $U(\chi
  )/(K_pL_p-1\,|\,p\in I)$, see also Thm.~\ref{th:nondegpair} below.
\end{remar}

By Rem.~\ref{re:cI+} one has $\cS (\chi )\cap \cU ^+(\chi )=\cS ^+(\chi )$.
Thus let
\begin{align*}
U^+(\chi )= &\,\cU ^+(\chi )+\cS (\chi )/\cS (\chi )\cong
\cU ^+(\chi )/\cS ^+(\chi ),\\
U^-(\chi )= &\,\cU ^-(\chi )+\cS (\chi )/\cS (\chi )\cong
\cU ^-(\chi )/\cS ^-(\chi ),\\
U^+_{+p}(\chi )= &\,\cU ^+_{+p}(\chi )+\cS (\chi )/\cS (\chi )\cong
\cU ^+_{+p}(\chi )/(\cS ^+(\chi )\cap \cU ^+_{+p}(\chi )),\\
U^+_{-p}(\chi )= &\,\cU ^+_{-p}(\chi )+\cS (\chi )/\cS (\chi )\cong
\cU ^+_{-p}(\chi )/(\cS ^+(\chi )\cap \cU ^+_{-p}(\chi )).
\end{align*}

\begin{theor}\label{th:nondegpair}
  The skew-Hopf pairing
  $\sHp $ in Prop.~\ref{pr:sHpdef} induces a skew-Hopf pairing
  of the Hopf algebras $U^+(\chi )\#\cU ^{+0}$ and
  $(U^-(\chi )\#\cU ^{-0})\cop $. The restriction of this pairing to
  $U^+(\chi )\times U^-(\chi )$ is non-degenerate.
\end{theor}

\begin{proof}
  Recall that $\cS ^+(\chi )\cU ^{+0}\subset
  \oplus _{m=2}^\infty \cU ^+(\chi )_m \cU ^{+0}$
  and
  $\cU ^{-0}\cS ^-(\chi )\subset
  \oplus _{m\le -2} \cU ^{-0}\cU ^-(\chi )_m$, and that
  $\sHp :\cU ^+(\chi )\cU ^{+0}\times \cU ^{-0}\cU ^-(\chi ) \to \fie $
  is $\ndZ $-homogeneous.
  Hence
  \begin{align*}
    \sHp (\cS ^+(\chi )\cU ^{+0},
    \cU ^{-0}(\fie \oplus \cU ^-(\chi )_1))=&0,\\
    \sHp ( (\fie \oplus \cU ^+(\chi )_1)\cU ^{+0},
    \cU ^{-0}\cS ^-(\chi ))=&0.
  \end{align*}
  By the arguments in the proof of Prop.~\ref{pr:Uchi},
  $\cS ^+(\chi )\cU ^{+0}\subset \cU ^+(\chi )\cU ^{+0}$ and
  $\cU ^{-0}\cS ^-(\chi )\subset \cU ^{-0}\cU ^-(\chi )$
  are Hopf ideals. Then Eq.~\eqref{eq:sHp2} gives that
  $\cS ^+(\chi )\cU ^{+0}$ is contained in the left radical
  and $\cU ^{-0}\cS ^-(\chi )$ is contained in the right radical of
  $\sHp $. Thus $\sHp $ induces a skew-Hopf pairing
  of the Hopf algebras $U^+(\chi )\#\cU ^{+0}$ and
  $(U^-(\chi )\#\cU ^{-0})\cop $. The left radical of the restriction
  of $\sHp $ to $\cU ^+(\chi )\times \cU ^-(\chi )$ is a braided coideal of
  the braided Hopf algebra $\cU ^+(\chi )$ by Eq.~\eqref{eq:sHp2},
  Prop.~\ref{pr:sHpdef}(ii), and Rem.~\ref{re:brcopr}.
  Moreover, it is $\ndZ $-homogeneous and contained in $\oplus
  _{m=2}^\infty \cU ^+(\chi )_m$ by Prop.~\ref{pr:sHpdef}(i). By
  definition,
  $\cS ^+(\chi )$ is the maximal such braided coideal of
  $\cU ^+(\chi )$, and hence $\cS ^+(\chi )$ is the left radical
  of the restriction of $\sHp $ to
  $\cU ^+(\chi )\times \cU ^-(\chi )$. Since
  $\dim \cU ^+(\chi )_m=\dim \cU ^-(\chi )_{-m}<\infty $ and
  $\dim \cS ^+(\chi )_m=\dim \phi _4(\cS ^+(\chi )_m)=
  \dim \cS ^-(\chi )_{-m}$ for all $m\in \ndN _0$, it follows that
  $\cS ^-(\chi )$ is the right radical
  of the restriction of $\sHp $ to
  $\cU ^+(\chi )\times \cU ^-(\chi )$. This proves the claim.
\end{proof}

\begin{corol}
  The Hopf algebra $U(\chi )$ is naturally isomorphic to the Drinfel'd double
  of the Hopf algebras $U^+(\chi )\#\cU ^{+0}$ and
  $(U^-(\chi )\#\cU ^{-0})\cop $.
\end{corol}

By Prop.~\ref{pr:Nicholschar} the maps $\derK _p,\derL _p\in
\End _\fie (\cU ^+(\chi ))$ induce $\fie $-endomor\-phisms of $U^+(\chi )$
which will again be denoted by $\derK _p$ and $\derL _p$, respectively.
The following application of Lemma~\ref{le:U+pcoid}(ii) will be important
in the next section.

\begin{propo}\label{pr:U+pchar}
	For all $p\in I$ the following equations hold.
	\begin{align*}
		\ker (\derK _p:U^+(\chi )\to U^+(\chi ))=&\,U^+_{+p}(\chi ),\\
		\ker (\derL _p:U^+(\chi )\to U^+(\chi ))=&\,U^+_{-p}(\chi ).
	\end{align*}
\end{propo}

\begin{proof}
  The inclusions ``$\supset $'' follow from Lemma~\ref{le:U+pcoid}(ii).
  By Lemma~\ref{le:U+dec} and Eqs.~\eqref{eq:derKL2} it suffices to show that
  $\derK _p(E_p^m)=0$ respectively $\derL _p(E_p^m)=0$
  for some $m\in \ndN $ implies that $E_p^m=0$ in $U^+(\chi )$.
  By Cor.~\ref{co:derspecel} one has $\derK _i(E_p^m)=\derL _i(E_p^m)=0$
  for all $i\in I\setminus \{p\}$. Therefore
  Prop.~\ref{pr:Nicholschar} implies that
  for all $m\in \ndN $ the relations $E_p^m=0$, $\derK _p(E_p^m)=0$, and
  $\derL _p(E_p^m)=0$ are equivalent.
\end{proof}

\section{Lusztig isomorphisms}
\label{sec:li}

One of our main goals in this paper is the construction of Lusztig
isomorphisms between Drinfel'd doubles of bosonizations of Nichols
algebras of diagonal type, see Thm.~\ref{th:Liso}. This is not
possible for all $\chi \in \cX $. Analogously to the quantized
enveloping algebra setting, one has to assume that $\chi $ is
$p$-finite for some $p\in I$, see Def.~\ref{de:Cartan}. Further, the
proof of the existence of the Lusztig maps and their bijectivity is
somewhat complex. Therefore first we introduce small ideals, with help
of which Lusztig maps can be defined, see Lemma~\ref{le:Lusztig1}.
This definition will then be used to induce isomorphisms between
Drinfel'd doubles. In
Subsect.~\ref{ssec:Coxli} many known relations for compositions of
Lusztig automorphisms are generalized to our setting.

In the whole section let $\chi \in \cX $,
$q_{i j}=\chi (\Ndb _i,\Ndb _j)$, and $c_{i j}=c^\chi _{i j}$
for all $i,j\in I$.

\subsection{Definition of Lusztig isomorphisms}
\label{ssec:defli}

Recall Eqs.~\eqref{eq:defEpim}-\eqref{eq:defEmim} and the definition
of $\hght \chi $ from Eq.~\eqref{eq:height}.

\begin{defin}\label{de:Ip+}
  Let $p\in I$ and $h=\hght \chi (\Ndb _p)$. Assume that $\chi $ is $p$-finite.
  Let
	$\cI _p^+(\chi )\subset \cU ^+(\chi )$ and
	$\cI _p^-(\chi )\subset \cU ^-(\chi )$ be the
	following ideals. If $\hght \chi (\Ndb _p)<\infty $, then let
	\begin{align*}
    \cI _p^+(\chi )=&(E_p^h,E^+_{i,1-c_{pi}}\,|\,
    i\in I\setminus \{p\}\text{ such that }1-c_{pi}<h),\\
		\cI _p^-(\chi )=&(F_p^h,F^+_{i,1-c_{pi}}\,|\,
    i\in I\setminus \{p\}\text{ such that }1-c_{pi}<h).
	\end{align*}
	Otherwise define
	\begin{align*}
    \cI _p^+(\chi )=(E^+_{i,1-c_{pi}}\,|\,i\in I\setminus \{p\}),\qquad
    \cI _p^-(\chi )=(F^+_{i,1-c_{pi}}\,|\,i\in I\setminus \{p\}).
	\end{align*}
\end{defin}

\begin{propo}\label{pr:Ip+}
	Let $p\in I$. Assume that $\chi $ is $p$-finite. Let $h=\hght \chi
  (\Ndb _p)$.
	
	(i) If $h<\infty $,
  then the following ideals of $\cU ^+(\chi )$ coincide.
	\begin{itemize}
		\item $\cI _p^+(\chi )$,
		\item $(E_p^h,E^-_{i,1-c_{pi}}\,|\,
      i\in I\setminus \{p\}\text{ such that }1-c_{pi}<h)$,
    \item $(E_p^h,E^+_{i,1-c_{pi}}\,|\,i\in I\setminus \{p\})$,
    \item $(E_p^h,E^-_{i,1-c_{pi}}\,|\,i\in I\setminus \{p\})$.
	\end{itemize}
	
	(ii) If $h=\infty $, then the following
	ideals of $\cU ^+(\chi )$ coincide.
	\begin{itemize}
		\item $\cI _p^+(\chi )$,
    \item $(E^-_{i,1-c_{pi}}\,|\,i\in I\setminus \{p\})$.
	\end{itemize}
\end{propo}

\begin{proof}
	For both statements the equality of the first two ideals follows from
	Lemma~\ref{le:E+=E-}. For the remaining assertions of part (i) of the lemma
	it suffices to show that if $1-c_{p i}\ge h$, (that is, $1-c_{p i}=h$ by
  definition of $c_{p i}$,) then
	$E^+_{i,h}$ and $E^-_{i,h}$ are elements of $\cI _p^+(\chi )$. The
	latter follows from the assumption $\qnum{h}{q_{p p}}=0$,
	Lemma~\ref{le:mq=0} and Eqs.~\eqref{eq:Eim+},\eqref{eq:Eim-}.
\end{proof}

The following lemma is a direct consequence of Lemma~\ref{le:phiE+} and 
Prop.~\ref{pr:Ip+}.

\begin{lemma}
  \label{le:Ipiso}
	Let $p\in I$. Assume that $\chi $ is $p$-finite.
  Then the ideals $\cI ^\pm _p(\chi )$ are compatible with the automorphisms
  and antiautomorphism in Prop.~\ref{pr:algiso} in the sense that
  \begin{align*}
    \varphi _{\ula }(\cI ^\pm _p(\chi ))=&\,
    \cI ^\pm _p(\chi ),&
    \phi _4(\cI ^\pm _p(\chi ))=&\, \cI ^\mp _p(\chi ),\\
    \phi _2(\cI ^\pm _p(\chi ))=&\, \cI ^\mp _p(\chi ^{-1}),&
    \phi _3(\cI ^\pm _p(\chi ))=&\, \cI ^\mp _p(\chi \op ),\\
    \cU ^0(\chi )\varphi _m(\cI ^\pm _p(\chi ))=&\,
    \cU ^0(\chi )\cI ^\pm _p(\chi ),&
    \cU ^0(\chi )\phi _1(\cI ^\pm _p(\chi ))=&\,
    \cU ^0(\chi )\cI ^\mp _p(\chi )
  \end{align*}
  for all  $\ula \in (\fienz )^I$ and $m\in \ndZ $.
\end{lemma}

Further, Cor.~\ref{co:derspecel} gives the following.

\begin{lemma}
  \label{le:Ipgood}
	Let $p\in I$. Assume that $\chi $ is $p$-finite.
  Then for the ideals $\cI ^\pm _p(\chi )$ the equivalent statements in
  Prop.~\ref{pr:goodideals} hold.
\end{lemma}

\begin{lemma}
	Let $p\in I$. Assume that $\chi $ is $p$-finite.
	
	(i) The $\fie $-endomorphism of
	$(\cU ^+_{+p}(\chi )+\cI ^+_p(\chi ))/\cI ^+_p(\chi )$ given
	by $X\mapsto (\ad E_p)X$ is locally nilpotent.

	(ii) The $\fie $-endomorphism of
	$(\cU ^+_{-p}(\chi )+\cI ^+_p(\chi ))/\cI ^+_p(\chi )$ given
	by $Y\mapsto E_pY-(L_p\actl Y)E_p$ is locally nilpotent.
	\label{le:adlocnil}
\end{lemma}

\begin{proof}
	The given maps are endomorphisms by the definition of $\cU ^+_{\pm p}(\chi
	)$.
	The statements of the lemma follow immediately from the following two facts.
	First, both $\fie $-endomorphisms are in fact skew-derivations of the
	corresponding algebra
	$(\cU ^+_{\pm p}(\chi )+\cI ^+_p(\chi ))/\cI ^+_p(\chi )$.
	Second, by the definitions of $E^\pm _{i,m}$ and $\cI ^+_p(\chi )$
	and by Prop.~\ref{pr:Ip+} these skew-derivations are nilpotent on the
	corresponding algebra generators $E^\pm _{i,m}$.
\end{proof}

Next we perform the first step towards the definition of Lusztig
isomorphisms.
Recall the definition of $\lambda _i(\chi )$ from Lemma~\ref{le:lambda}.

\begin{lemma}\label{le:Lusztig1}
  Let $p\in I$. Assume that $\chi $ is $p$-finite.
  There are unique algebra maps
  $$\LT _p,\LT _p^-:
  \cU (\chi )\to \cU (r_p(\chi ))/\big(\cI _p^+(r_p(\chi )),
  \cI _p^-(r_p(\chi ))\big)$$
  such that\footnote{To avoid confusion in the proof of the lemma,
  the generators of $\cU (r_p(\chi ))$ are underlined. This
  convention will be used only in this subsection.}
      \begin{align*}
	\LT _p(K_p)=&\LT _p^-(K_p)=\ulK _p^{-1},&
	\LT _p(K_i)=&\LT _p^-(K_i)=\ulK _i\ulK _p^{-c_{p i}},\\
	\LT _p(L_p)=&\LT _p^-(L_p)=\ulL _p^{-1},&
	\LT _p(L_i)=&\LT _p^-(L_i)=\ulL _i\ulL _p^{-c_{p i}},\\
	\LT _p(E_p)=&\ulF _p\ulL _p^{-1},&
	\LT _p(E_i)=&\ulE ^+_{i,-c_{p i}},\\
	\LT _p(F_p)=&\ulK _p^{-1}\ulE _p,&
	\LT _p(F_i)=&\lambda _i(r_p(\chi ))^{-1}\ulF ^+_{i,-c_{p i}},\\
	\LT _p^-(E_p)=&\ulK _p^{-1}\ulF _p,&
	\LT _p^-(E_i)=&\lambda _i(r_p(\chi ^{-1}))^{-1}
	\ulE ^-_{i,-c_{pi}},\\
	\LT _p^-(F_p)=&\ulE _p\ulL _p^{-1},&
	\LT _p^-(F_i)=&(-1)^{c_{p i}}\ulF ^-_{i,-c_{p i}}.
      \end{align*}
\end{lemma}

\begin{proof}
  One has to show the compatibility of the definitions of $\LT _p$, $\LT _p^-$
  with the defining relations of $\cU (\chi )$.

  Let $\bq _{i j}=r_p(\chi )(\Ndb _i,\Ndb _j)$ for all $i,j\in I$.
	The compatibility of $\LT _p$ with the relations
	\eqref{eq:KLrel}--\eqref{eq:KFrel} is ensured (and enforced) by the
  choice of $r_p(\chi )\in \cX $, see Eq.~\eqref{eq:rp}.
	The relation
	$$ [\LT _p(E_p),\LT _p(F_p)]=\LT _p(K_p-L_p) $$
	is part of the proof of Prop.~\ref{pr:algiso}(3). Further, for all
	$i\in I\setminus \{p\}$ one gets
	\begin{align*}
		[\LT _p(E_i),\LT _p(F_p)]
		=[\ulE ^+_{i,-c_{p i}},\ulK _p^{-1}\ulE _p]
		=-\ulK _p^{-1}\ulE ^+_{i,1-c_{p i}}\in \cI _p^+(r_p(\chi ))
	\end{align*}
	because of Eqs.~\eqref{eq:Emukomm} and \eqref{eq:defEpim} and
	Prop.~\ref{pr:Ip+}. Similarly,
	\begin{align*}
		[\LT _p(E_p),\LT _p(F_i)]=&[\ulF _p\ulL _p^{-1},
		\lambda _i(r_p(\chi ))^{-1}\ulF ^+_{i,-c_{p i}}] \\
		=&\lambda _i(r_p(\chi ))^{-1}\bq _{i p}^{-1}\bq _{p p}^{c_{p i}}
		\ulF ^+_{i,1-c_{p i}}\in \cI _p^-(r_p(\chi ))
	\end{align*}
	by Eqs.~\eqref{eq:Fmukomm} and \eqref{eq:F+}.

	Assume now that $i,j\in I\setminus \{p\}$ such that $i\not=j$.
	Then
	\begin{align}
	 [\LT _p(E_i),\LT _p(F_j)]
	 =[\ulE^+_{i,-c_{p i}},\lambda _j(r_p(\chi ))^{-1}\ulF^+_{j,-c_{p j}}]=0 
	  \label{eq:cTETF1}
	\end{align}
	by Lemma~\ref{le:E+F+2}. On the other hand, for all $i\in I\setminus
	\{p\}$
	\begin{align}
	 [\LT _p(E_i),\LT _p(F_i)]
	 =[\ulE ^+_{i,-c_{p i}},\lambda _i(r_p(\chi ))^{-1}\ulF ^+_{i,-c_{p i}}]
	=\LT _p(K_i)-\LT _p(L_i) 
	  \label{eq:cTETF2}
	\end{align}
	by Lemma~\ref{le:E+F+1}.

	Similarly one can show that $\LT ^-_p$ is well-defined. The relations
	$$ [\ulE^-_{i,-c_{p i}},\ulF^-_{j,-c_{p j}}]=
	(-1)^{c_{p i}}\delta _{i,j}\lambda _i(r_p(\chi ^{-1}))\LT ^-_p(K_i-L_i), $$
	where $i,j\in I\setminus \{p\}$,
	follow from Eqs.~\eqref{eq:cTETF1} and \eqref{eq:cTETF2} by applying
	the isomorphism $\phi _2$ and using Lemma~\ref{le:phiE+}.
\end{proof}

Let $p\in I$. Assume that $\chi $ is $p$-finite.
In the next lemma and its proof we use the following abbreviations:
\begin{align}
  \label{eq:cU+'}
  \cU ^+(r_p(\chi ))'=&\,
  \big(\cU ^+(r_p(\chi ))+\cI _p^+(r_p(\chi ))\big)/
  \cI _p^+(r_p(\chi )),\\
  \label{eq:cU+p'}
  \cU ^+_{\epsilon p}(r_p(\chi ))'=&\,
  \big(\cU ^+_{\epsilon p}(r_p(\chi ))+\cI _p^+(r_p(\chi ))\big)/
  \cI _p^+(r_p(\chi )),
\end{align}
where $\epsilon \in \{+,-\}$.
By Cor.~\ref{co:derspecel} the skew-derivations $\derK _p$ and $\derL _p$ of
$\cU ^+(r_p(\chi ))$
induce well-defined skew-derivations on $\cU ^+(r_p(\chi ))'$ which
then will be denoted by the same symbol.

\begin{lemma}
  \label{le:psiadE}
  Let $p\in I$, $\LT _p$ and $\LT _p^-$
	as in Lemma~\ref{le:Lusztig1}.
  Let $\bq _{i j}=r_p(\chi )(\Ndb _i,\Ndb _j)$ for all $i,j\in I$.

    (a) For all $X\in \cU ^+_{-p}(\chi )$ and
    $Y\in \cU ^+_{+p}(\chi )$
    \begin{align*}
      \LT _p(E_pX-(L_p\actl X)E_p)=&\bq _{pp}\derL _p(\LT _p(X)),\\
      \LT _p^-(E_pY-(K_p\actl Y)E_p)=&-\ulK _p^{-1}\actl \derK _p(\LT ^-_p(Y)).
    \end{align*}

    (b) For all $i\in I\setminus \{p\}$ and $t\in \ndN _0$ with $t\le -c_{pi}$
    \begin{align*}
      \LT _p(E^-_{i,t})=&\,
      \bq _{pp}^t\prod _{s=0}^{t-1}\qnum{-c_{p i}-s}{\bq _{p p}}
      \prod _{s=1}^t(1-\bq _{p p}^{-c_{p i}-s}\bq _{p i}\bq _{i p})
      \ulE ^+_{i,-c_{p i}-t},\\
      \LT _p^-(E^+_{i,t})=&\,
      \prod _{s=1}^{-c_{p i}-t}\qnum{s}{\bq _{p p}^{-1}}^{-1}
      \prod _{s=0}^{-c_{p i}-t-1}(\bq _{p p}^{-s}\bq _{p i}^{-1}
      \bq _{i p}^{-1}-1)^{-1}\ulE ^-_{i,-c_{p i}-t}.
    \end{align*}

    (c) For all $i\in I\setminus \{p\}$ and $t\in \ndN _0$ with $t>-c_{pi}$
    \begin{align*}
      \LT _p(E^-_{i,t})=\LT _p^-(E^+_{i,t})=0.
    \end{align*}
    
    (d) The following relations hold.
    \begin{align*}
	    \LT _p(\cU ^+_{-p}(\chi ))\subset \cU ^+_{+p}(r_p(\chi ))',\qquad
	    \LT _p^-(\cU ^+_{+p}(\chi ))\subset \cU ^+_{-p}(r_p(\chi ))'.
    \end{align*}
  \end{lemma}

  \begin{proof}
	  We start with a technical statement.

	  \textit{Step 1. Part (a) holds
	  for all $X,Y\in \cU (\chi )$ with
	  $\LT _p(X)\in \cU ^+_{+p}(r_p(\chi ))'$ and
	  $\LT _p^-(Y)\in \cU ^+_{-p}(r_p(\chi ))'$.}
	  Let $X\in \cU (\chi )$. Assume that $\LT _p(X)\in \cU
    ^+_{+p}(r_p(\chi ))'$. By the remark above the lemma, the
    expression $\derL _p(T_p(X))$ is well-defined.
    By definition of $\LT _p$ one gets
	  \begin{align*}
		  &\LT _p(E_pX-(L_p\actl X)E_p)=\ulF _p\ulL _p^{-1}\LT _p(X)
		  -(\ulL _p^{-1}\actl \LT _p(X))\ulF _p\ulL _p^{-1}\\
		  &\quad =-[\ulL _p^{-1}\actl \LT _p(X),\ulF _p]\ulL _p^{-1}\\
		  &\quad =\big(-\derK _p(\ulL _p^{-1}\actl \LT _p(X))\ulK _p
		  +\ulL _p\derL _p(\ulL _p^{-1}\actl \LT _p(X))\big)\ulL _p^{-1}\\
		  &\quad =\bq _{pp}\derL _p(\LT _p(X))
		  -\bq _{pp}(\ulL _p^{-1}\actl \derK _p(\LT _p(X)))\ulK _p\ulL _p^{-1}\\
		  &\quad =\bq _{pp}\derL _p(\LT _p(X)),
	  \end{align*}
	  where the penultimate equation follows from Lemma~\ref{le:commder} and
	  the last one from the assumption
    $\LT _p(X)\in \cU ^+_{+p}(r_p(\chi ))'$ and Lemma~\ref{le:U+pcoid}(ii).
	  This and a similar calculation for $\LT _p^-(E_pY-(K_p\actl Y)E_p)$
	  imply the statement of step 1.

	  \textit{Step 2. Proof of parts (b) and (c).}
	  We proceed by induction on $t$. For $t=0$, part (b) is valid by
	  the definitions of $\LT _p$ and $\LT ^-_p$. Assume now that the
	  formulas in part (b) are valid for some $t<-c_{p i}$, where
	  $i\in I\setminus \{p\}$. In view of Eqs.~\eqref{eq:defEpim},
	  \eqref{eq:defEmim} one can
	  apply step 1 of the proof to $X=E^-_{i,t}$ and $Y=E^+_{i,t}$. Then one
	  obtains part (b) for $\LT _p(E^-_{i,t+1})$ and
	  $\LT ^-_p(E^+_{i,t+1})$ from the induction hypothesis and
    Cor.~\ref{co:derspecel}. Similarly, if $t=-c_{p i}$, then the analogous
	  induction step shows that $\LT _p(E^-_{i,-c_{p i}+1})$ is a multiple of
	  $\derL _p(E_i)=0$ and hence it is zero. This and a similar argument
	  for $\LT _p^-(E^+_{i,-c_{p i}+1})$ imply part (c).

	  \textit{Step 3. Proof of parts (a) and (d).} Since $\LT _p$ and
	  $\LT ^-_p$ are algebra maps, part (d)
	  follows immediately from the definition of $\cU ^+_{\pm p}(\chi )$ and
	  parts (b) and (c) of the lemma. Finally, part (a) is a
	  direct consequence of step~1 of the proof and part (d) of the lemma.
  \end{proof}

\begin{propo}\label{pr:Lusztig1}
    Let $p$, $\LT _p$ and $\LT _p^-$ as in Lemma~\ref{le:Lusztig1}.

	(i) The maps $\LT _p$, $\LT ^-_p$ induce algebra isomorphisms
    $$ \LT_p,\LT ^-_p:
    \cU (\chi )/\big(\cI _p^+(\chi ),\cI _p^-(\chi )\big)
    \to 
	\cU (r_p(\chi ))/\big(\cI _p^+(r_p(\chi )),\cI _p^-(r_p(\chi ))\big).$$

	(ii) The isomorphisms $\LT _p, \LT ^-_p$ in (i) satisfy the equations
    \begin{gather*}
      \LT _p\LT _p^-=\,\LT _p^-\LT _p=\id,
    \end{gather*}
    \begin{align*}
		  \LT _p\varphi _{\ula }=&\,\varphi _{\ulb }
		  \LT _p,&
		  &\text{where } \ula \in (\fienz )^I,\quad
		  b_i=a_ia_p^{-c_{pi}} \text{ for all $i\in I$,}\\
		  \LT ^-_p\varphi _{\ula }=&\,\varphi _{\ulb }
		  \LT ^-_p,&
		  &\text{where } \ula \in (\fienz )^I,\quad
		  b_i=a_i a_p^{-c_{pi}} \text{ for all $i\in I$,}\\
		  \LT _p\phi _2=&\,\phi _2\LT ^-_p\varphi _{\ula },&
		  &\text{where } a_i=\,(-1)^{\delta _{i,p}} \text{ for all $i\in
		  I$,}\\
		  \LT _p\phi _3=&\,\phi _3\LT _p\varphi _{\ullam },& &\text{where }
		  \lambda _p=q_{p p}^{-1},\\
		  & & &\lambda _i=\lambda _i(r_p(\chi ))^{-1}
		  \text{ for all $i\in I\setminus \{p\}$,}\\
		  \LT ^-_p\phi _3=&\,\phi _3\LT ^-_p\varphi _{\ullam },& &\text{where }
		  \lambda _p=q_{p p}^{-1},\\
		  & & &\lambda _i=(-1)^{c_{p i}}\lambda _i(r_p(\chi ^{-1}))
		  \text{ for all $i\in I\setminus \{p\}$,}\\
		  \LT _p\phi _4=&\,\phi _4\LT ^-_p\varphi _{\ula }&
      &\text{for some $\ula \in (\fienz )^I$.}
    \end{align*}
\end{propo}

Note that part~(ii) makes only sense if one uses appropriate bicharacters. For
example, the equation $\LT _p\LT _p^-=\id $ means that if $\LT _p^-$ is defined
with respect to $\chi $, then $\LT _p$ has to be defined with respect to
$r_p(\chi )$. Similar adaptation has to be performed for the
commutation relations with $\phi _2$ and $\phi _3$.

\begin{proof}
	First check that equation
	\begin{align}
		\LT _p\phi _2(X)=\phi _2\LT ^-_p \varphi _{\ula }(X),\quad
		\text{where } a_i=\,(-1)^{\delta _{i,p}} \text{ for all $i\in I$,}
		\label{eq:TphiX=phiT-X}
	\end{align}
	holds for all generators $X$ of $\cU (\chi )$. Since $\LT _p$, $\LT _p^-$,
  $\phi _2$, and $\varphi _{\ula }$ are algebra maps, this implies that
	\begin{align}
		\LT _p\phi _2=\phi _2\LT ^-_p \varphi _{\ula },\quad
		\text{where } a_i=\,(-1)^{\delta _{i,p}} \text{ for all $i\in I$.}
		\label{eq:Tphi=phiT-}
	\end{align}
	Further, the equations 
	$\LT _p\varphi _{\ula }=\varphi _{\ulb }\LT _p$,
	$\LT ^-_p\varphi _{\ula }=\varphi _{\ulb }\LT ^-_p$
	as algebra maps $\cU (\chi )\to
    \cU (r_p(\chi ))/(\cI _p^+(r_p(\chi )),\cI _p^-(r_p(\chi )))$
	follow immediately from the definitions of $\LT _p$, $\LT ^-_p$ and
	$\varphi _{\ula }$.
	Using Eq.~\eqref{eq:Tphi=phiT-} one can easily see with help of
	Lemmata~\ref{le:Lusztig1}, \ref{le:psiadE} and \ref{le:phiE+} that $\LT _p$
	and $\LT _p^-$ are well-defined on the given
	quotient of $\cU (\chi )$. Again using Lemma~\ref{le:psiadE} one
	gets that $\LT _p\LT _p^-=\LT _p^-\LT _p=\id $ and
	$\LT _p\phi _3=\phi _3\LT _p\varphi _{\ullam }$.
  The equation
	$\LT _p^-\phi _3=\phi _3\LT _p^-\varphi _{\ullam }$
  follows from equations
	$\LT _p\phi _3=\phi _3\LT _p\varphi _{\ullam }$ and
	$\LT _p\phi _2=\phi _2\LT _p^-\varphi _{\ula }$
  by Prop.~\ref{pr:commiso}. 
  Equation 
	$\LT _p\phi _4=\phi _4\LT _p^-\varphi _{\ula }$ can be obtained
  similarly to Eq.~\eqref{eq:Tphi=phiT-}.
\end{proof}

\subsection{Lusztig isomorphisms for $U (\chi )$}
\label{ssec:Uli}

We continue to use
the notation from Sects.~\ref{sec:Nichdiag} and \ref{sec:li} and from
Prop.~\ref{pr:Uchi}.

\begin{lemma} \label{le:S+gen}
  Let $p\in I$. Assume that $\chi $ is $p$-finite.

	(i) One has $\cI ^+_p(\chi )\subset \cS ^+(\chi )$.

  (ii) Let $\epsilon \in \{+,-\}$.
  The ideal $\cS ^+(\chi )$ of $\cU ^+(\chi )$ is generated
	by the subset
  \begin{align*}
	  (\cS ^+(\chi )\cap \fie [E_p]) \cup 
	  (\cS ^+(\chi )\cap \cU ^+_{\epsilon p}(\chi )).
  \end{align*}
\end{lemma}

\begin{proof}
	To part~(i). The generators of $\cI ^+_p(\chi )$ are in $\cS ^+(\chi
  )$ because of Cor.~\ref{co:derspecel} and
  Prop.~\ref{pr:Nicholschar}(2)$\Rightarrow $(1). This implies the claim.

	To (ii). We consider the case $\epsilon =1$, the proof for the other one is
  similar.
  Let $X\in \cS ^+(\chi )$. By Lemma~\ref{le:U+dec}
	there exists $m\in \ndN _0$ and uniquely determined elements $X_0,\ldots
	,X_m\in \cU ^+_{+p}(\chi )$ such that $X=\sum _{i=0}^mX_iE_p^i$.
	By Lemma~\ref{le:S+homog} it suffices to consider the case when $X$ is
	homogeneous with respect to the standard grading.
  Further, since $\qfact{n}{q_{p p}}=0$ implies that
  $E_p^n\in \cS ^+(\chi )$, see Prop.~\ref{pr:Nicholschar},
  one can assume that $\qfact{m}{q_{p p}}\not=0$, and that either
  $X=0$ or $X_m\notin \cS ^+(\chi )$.
	By Lemma~\ref{le:commEFi} and Cor.~\ref{co:derspecel} one gets
	\begin{align*}
		\sum _{i=0}^m(\derK _p)^m(X_iE_p^i)=\qfact{m}{q_{pp}}X_m.
	\end{align*}
	Thus Prop.~\ref{pr:Nicholschar} gives that $X_m\in \cS ^+(\chi )$.
  Hence $X=0$, and (ii) is proven.
\end{proof}

\begin{propo}\label{pr:derTcomm}
  Let $p$, $\LT _p$ and $\LT _p^-$ as in Lemma~\ref{le:Lusztig1}.
	
	(i) For all $i\in I\setminus \{p\}$ there exists
	$\ula \in (\fienz )^I$ such that
  \begin{align}\label{eq:derLTcomm}
		\derL _i \LT _p
		=\LT _p\circ (\derL _p)^{-c_{p i}}\derL _i \varphi _{\ula }
	\end{align}
	as a linear map $\cU ^+_{-p}(\chi )\to \cU ^+_{+p}(r_p(\chi ))'$,
  see Eq.~\eqref{eq:cU+p'}.

    (ii) For all $i\in I\setminus \{p\}$ there exists
	$\ula \in (\fienz )^I$ such that
	\begin{align*}
		\derK _i \LT ^-_p
		=\LT ^-_p\circ (\derK _p)^{-c_{p i}}\derK _i \varphi _{\ula }
	\end{align*}
	as a linear map $\cU ^+_{+p}(\chi )\to \cU ^+_{-p}(r_p(\chi ))'$.
\end{propo}

\begin{proof}
  We prove part~(i) in $3$ steps and leave the similar proof of part~(ii)
  to the reader.

  \textit{Step 1. Eq.~\eqref{eq:derLTcomm} holds on the generators of $\cU
  ^+_{-p}(\chi )$.} Let $j\in I\setminus \{p\}$ and $m\in \ndN _0$.
  If $m>-c_{p j}$, then the evaluations of both sides of
  Eq.~\eqref{eq:derLTcomm} on $E^-_{i,m}$
  give $0$: the left hand side by Lemma~\ref{le:psiadE}(c) and the right hand
  side by Lemmata~\ref{le:phiE+} and \ref{le:Ipgood}
  and by Prop.~\ref{pr:Lusztig1}(i). Assume now that $m\le c_{p j}$.
  If $j\not=i$, then Lemma~\ref{le:psiadE}(b) and Cor.~\ref{co:derspecel} imply
  that both sides of Eq.~\eqref{eq:derLTcomm} are $0$. Suppose that $j=i$.
  Then
  \begin{align*}
    \derL _i\LT _p(E^-_{i,m})\in &\,\fienz \derL _i(E^+_{i,-c_{p i}-m})
    =\fienz \delta _{m,-c_{p i}}, \\
    \LT _p\big((\derL _p)^{-c_{p i}}\derL _i \varphi _{\ula }(E^-_{i,m})\big)
    \in &\,
    \fienz \LT _p\big((\derL _p)^{-c_{p i}}\derL _i (E^-_{i,m})\big)\\
    =&\,\fienz \LT _p\big((\derL _p)^{-c_{p i}}(E_p^m)\big)\\
    =&\,\fienz  \LT _p(\delta _{m,-c_{p i}})
    =\fienz  \delta _{m,-c_{p i}}.
  \end{align*}
  
  \textit{Step 2. The map $\vartheta _i=(\derL _p)^{-c_{p i}}\derL _i
  \varphi _{\ula }\in \End _{\fie }(\cU ^+_{-p}(\chi ))$ satisfies
  \begin{align*}
    \vartheta _i(EE')=\vartheta _i(E)E'+
    (L_p^{c_{p i}}L_i^{-1}\actl E)\vartheta _i(E')\quad \text{
    for all $E,E'\in \cU ^+_{-p}(\chi )$.}
  \end{align*}
  }
  The statement follows immediately from Eqs.~\eqref{eq:derKL2},
  \eqref{eq:derL_KL} and from
  $\derL _p(E)=\derL _p(E')=0$, see Cor.~\ref{co:derspecel}.

  \textit{Step 3. Eq.~\eqref{eq:derLTcomm} holds on $\cU ^+_{-p}(\chi )$.}
  In view of step~1 it suffices to show that if Eq.~\eqref{eq:derLTcomm} holds
  on $E,E'\in \cU ^+_{-p}(\chi )$, then it also holds on $EE'$.
  Since $\LT _p$ is an algebra map, the latter follows from
  Eq.~\eqref{eq:derKL2}, step~2 and equation
  $\LT _p(L_p^{c_{p i}}L_i^{-1})=L_i^{-1}$.
\end{proof}

\begin{theor}\label{th:Liso}
    Let $p$, $\LT _p$ and $\LT _p^-$
    as in Lemma~\ref{le:Lusztig1}.
    The maps $\LT _p$, $\LT ^-_p$ induce algebra isomorphisms
	$$ \LT_p,\LT ^-_p: U (\chi )\to U (r_p(\chi )).$$
	The analogs of the commutation relations in Prop.~\ref{pr:Lusztig1}
	hold.
\end{theor}

\begin{proof}
  Extend the notation in Eqs.~\eqref{eq:cU+'} and \eqref{eq:cU+p'} by defining
  \begin{align*}
    \cS ^+(\chi )'=&\, \cS ^+(\chi )/ \cI _p^+(\chi ), & 
    \cS (\chi )'=&\,
    \cS (\chi )/ (\cI _p^+(\chi ), \cI _p^-(\chi )).
  \end{align*}
	In view of the commutation relations between $\LT _p$ and $\phi _3$
	respectively $\LT ^-_p$ and $\phi _3$ it
	suffices to show that
	$\LT _p(\cS ^+(\chi ))\subset \cS (r_p(\chi ))'$ and
	$\LT ^-_p(\cS ^+(\chi ))\subset \cS (r_p(\chi ))'$.
	We prove the above relation for $\LT _p$. The proof for $\LT ^-_p$ goes
	similarly. Further, by Lemma~\ref{le:S+gen} it suffices to show that
  \begin{align*}
    \LT _p(\cS ^+(\chi )\cap \cU ^+_{-p}(\chi ))\subset
    \cS ^+(r_p(\chi ))',\quad
    \LT _p(\cS ^+(\chi )\cap \fie [E_p])\subset \cS (r_p(\chi ))',
  \end{align*}
	where the latter relation is obviously true since
  \begin{align}\label{eq:S+capEp}
    \cS ^+(\chi )\cap \fie [E_p]{=}
		\begin{cases}
      0 & \text{if $\hght \chi (\Ndb _p)=\infty $,}\\
      \sum _{m=h}^\infty \fie E_p^m &
      \text{if $h=\hght \chi (\Ndb _p)<\infty $.}
		\end{cases}
	\end{align}

	Since $\cS ^+(\chi )\cap \cU ^+_{-p}(\chi )$ is $\ndZ ^I$-graded, it is
	sufficient to show that $\LT _p(X)\in \cS ^+(r_p(\chi ))'$ for any
	homogeneous element $X\in \cS ^+(\chi )\cap \cU ^+_{-p}(\chi )$.
	This can be done by induction on $|\mu |$, where $\LT _p(X)\in \cU
  ^+_{+p}(r_p(\chi ))'_\mu $, see Lemma~\ref{le:psiadE}(d).
  The induction hypothesis is fulfilled
	since $\LT _p(X)\in \cU ^+_{+p}(r_p(\chi ))'_0$ implies $X\in \fie 1$,
	and hence $X\in \cS ^+(\chi )$ if and only if $X=0$.

	Let now $n\in \ndN $ and assume that 
	relations $X\in \cS ^+(\chi )\cap \cU ^+_{-p}(\chi )$ and
	$\LT _p(X)\in \cU ^+_{+p}(r_p(\chi ))'_\mu $ with $|\mu |\le n$
  imply that
	$\LT _p(X)\in \cS ^+(r_p(\chi ))'$.
	Let $Y\in \cS ^+(\chi )\cap \cU ^+_{-p}(\chi )$
	such that $\LT _p(Y)\in \cU ^+_{+p}(r_p(\chi ))'_\mu $,
  where $|\mu |=n+1$.
	We have to show that $\LT _p(Y)\in \cS ^+(r_p(\chi ))'_\mu $.
	By Prop.~\ref{pr:Nicholschar} this is equivalent to the relations
	$\derL _i(\LT _p(Y))\in \cS ^+(r_p(\chi ))'_{\mu -\Ndb _i}$
  for all $i\in I$. If $i=p$, then one gets from Lemma~\ref{le:psiadE} that
	\begin{align*}
		\derL _p(\LT _p(Y))=&\,q_{pp}^{-1}\LT _p(E_pY-(L_p\actl Y)E_p).
	\end{align*}
	Since $E_pY-(L_p\actl Y)E_p\in \cS ^+(\chi )\cap \cU ^+_{-p}(\chi )$,
	induction hypothesis implies that $\derL _p(\LT _p(Y))\in \cS
	^+(r_p(\chi ))'$. On the other hand, if $i\not=p$, then analogously
	Props.~\ref{pr:Nicholschar} and \ref{pr:derTcomm}, see also step~2 of the
  proof of the latter, imply that
	$\derL _i(\LT _p(Y))\in \cS ^+(r_p(\chi ))'$.
  This completes the proof of the theorem.
\end{proof}

\subsection{Coxeter relations between Lusztig isomorphisms}
\label{ssec:Coxli}

The aim of this subsection is to prove Thm.~\ref{th:LTCox}, that is,
Lusztig isomorphisms
satisfy Coxeter type relations.
Note that a case by case proof as in
\cite[Subsect.\,33.2]{b-Lusztig93} is not reasonable
because of the presence of dozens of different
examples of rank 2.

In the following claims we will use the following setting.

\begin{setti}\label{se:ijseq}
	Let $\chi \in \cX $. Assume that $\chi '$ is $p$-finite for all
  $p\in I$, $\chi '\in \cG (\chi )$. Let $i,j\in I$ with $i\not=j$.
  Let $i_{2n+1}=i$ and $i_{2n}=j$ for all $n\in \ndZ $.
  Let $M=|R^\chi _+\cap (\ndN _0\Ndb _i+\ndN _0\Ndb _j)|$.
\end{setti}

\begin{lemma}
	\label{le:roorank2}
  Assume Setting~\ref{se:ijseq}. Then
	\begin{align*}
		M=&\min \{m\in \ndN _0\,|\,
		\s _{i_m}\cdots \s_{i_2}\s _{i_1}^\chi (\Ndb _j)\in
		-\ndN _0^I\}\\
		=&1+\min \{m\in \ndN _0\,|\,
    \s_{i_m}\cdots \s_{i_2}\s_{i_1}^\chi (\Ndb _j)=\Ndb _{i_{m+1}}\}.
	\end{align*}
\end{lemma}

\begin{proof}
	See \cite[Lemma~6]{a-HeckYam08}. The right hand side has to be interpreted
  as $\infty $ if the minimum is taken over the empty set.
\end{proof}

The main result in this subsection is based on the following lemma.

\begin{lemma}
	\label{le:Coxeterkey}
  Assume Setting~\ref{se:ijseq}.
	Let $m,r\in \ndN _0$, and assume that
	\begin{align}
		\label{eq:Coxeterkey-1}
    \s _{i_m}\cdots \s _{i_2}\s _{i_1}^\chi (\Ndb _i+r\Ndb _j)=\Ndb _{i_{m+1}}.
	\end{align}
	Then there exists $t\in \ndN _0$ such that
  $\s _{i_m}\cdots \s _{i_2}\s _{i_1}^\chi (\Ndb _j)
  =\Ndb _{i_m}+t\Ndb _{i_{m+1}}$.
\end{lemma}

\begin{proof}
  By the definition of $\s _k^{\chi '}$, where $k\in I$,
  $\chi '\in \cG (\chi )$,
  one gets $\s _{i_m}\cdots \s _{i_2}\s _{i_1}^\chi (\Ndb _j)=
	t_0\Ndb _{i_m}+t\Ndb _{i_{m+1}}$ for some $t_0,t\in \ndZ $. One has to
	show that $t_0=1$ and $t\in \ndN _0$. By Eq.~\eqref{eq:w*chi},
  $\det \s _k^{\chi '}|_{\ndZ \Ndb _i\oplus \ndZ \Ndb _j}=-1$ for all
  $k\in I$, $\chi '\in \cG (\chi )$. Using this, Eq.~\eqref{eq:Coxeterkey-1}
	implies that
  $t_0=1$. By Prop.~\ref{pr:w*func}(d) and Axioms (R1), (R3)
  one obtains that $t\in \ndN _0$.
\end{proof}

\begin{propo}
	\label{pr:T(Eim)}
  Assume Setting~\ref{se:ijseq}.
  Let $m,r\in \ndN _0$.
  Assume that
	\begin{align}\label{eq:T(Eim)-1}
		m<M,\quad
    \s _{i_m}\cdots \s _{i_2}\s _{i_1}^\chi (\Ndb _i+r\Ndb _j)
		\in \ndN _0^I.
	\end{align}
  Let $w=\s _{i_m}\cdots \s _{i_2}\s _{i_1}^\chi $. Then for
	$E^+_{i,r(j)}, E^-_{i,r(j)}\in U(\chi )$
	\begin{align}\label{eq:T(Eim)-2}
		\LT _{i_m}\cdots \LT _{i_2}\LT _{i_1}(E^+_{i,r(j)}) \in 
		U^+(w^*\chi )_{w(\Ndb _i+r\Ndb _j)},\\
		\label{eq:T(Eim)-3}
		\LT ^-_{i_m}\cdots \LT ^-_{i_2}\LT ^-_{i_1}(E^-_{i,r(j)}) \in 
		U^+(w^*\chi )_{w(\Ndb _i+r\Ndb _j)}.
	\end{align}
	In particular, if $w(\Ndb _i+r\Ndb _j)=\Ndb _{i_{m+1}}$, then
	\begin{align}
		\label{eq:T(Eim)-4}
		\LT _{i_m}\cdots \LT _{i_1}(\fie E^+_{i,r(j)})
		=\fie E_{i_{m+1}},\quad
		\LT ^-_{i_m}\cdots \LT ^-_{i_1}(\fie E^-_{i,r(j)})
		=\fie E_{i_{m+1}}.
	\end{align}
\end{propo}

\begin{proof}
	The last statement of the proposition follows at once from the equation
	$U^+(w^*\chi )_{\Ndb _i}=\fie E_i$ and the fact that the maps $\LT _p$,
	where $p\in \{i,j\}$, are algebra isomorphisms.
	
	The remaining assertions will be proven by induction on $m$.
	If $m=0$, then the claim follows from the definition of $E^\pm _{i,r(j)}$.
	Assume now that $m>0$ and that the lemma holds for all smaller values of
	$m$. First we prove by an indirect proof that $r>0$. Assume that $r=0$.
	Then
  \begin{align*}
    R^{w^*\chi }_+\ni \s _{i_m}\cdots \s _{i_2}\s _{i_1}^\chi (\Ndb _i)
    =\s _{i_m}\cdots \s_{i_3}\s _{i_2}^{r_1(\chi )}(-\Ndb _i)
	  \in -\ndN _0\Ndb _i-\ndN _0\Ndb _j,
  \end{align*}
	where the last relation follows from
  Lemma~\ref{le:roorank2} and the first formula of
	assumption~\eqref{eq:T(Eim)-1}.
	The obtained relation
  \begin{align*}
    R^{w^*\chi }_+\cap -(\ndN _0\Ndb _i+\ndN _0\Ndb _j)\not=\emptyset
  \end{align*}
  is a contradiction to $R^{w^*\chi }_+\subset \ndN _0^I$,
  and hence $r>0$.

	Now we perform the induction step by induction on $r$.
	One gets
	\begin{align*}
		\LT _{i_m}&\cdots \LT _{i_2}\LT _{i_1}(E^+_{i,r(j)})=
		\LT _{i_m}\cdots \LT _{i_2}\LT _{i_1}(E_jE^+_{i,r-1(j)}
		-(K_j\actl E^+_{i,r-1(j)})E_j).
		\intertext{By Thm.~\ref{th:Liso} this is equal to}
		&\LT _{i_m}\cdots \LT _{i_2}\LT _{i_1}(E_j)
		\LT _{i_m}\cdots \LT _{i_2}\LT _{i_1}(E^+_{i,r-1(j)})\\
		&-
		(\LT _{i_m}\cdots \LT _{i_2}\LT _{i_1}(K_j)
    \actl \LT _{i_m}\cdots \LT _{i_2}\LT _{i_1}(E^+_{i,r-1(j)}))
		\LT _{i_m}\cdots \LT _{i_2}\LT _{i_1}(E_j).
	\end{align*}
  If $w(\Ndb _i+(r-1)\Ndb _j)\in R^{w^*\chi }_+$, then after replacing
	$\LT _{i_1}(E_j)=\LT _i(E_j)$ by $E^+_{j,-c_{ij} (i)}$ in the above
	formula one can apply the induction hypotheses for $m-1$ respectively
	$r-1$. Thus in this case relation~\eqref{eq:T(Eim)-2} holds.

  Assume now that $w(\Ndb _i+(r-1)\Ndb _j)\notin R^{w^*\chi }_+$. In this case,
	which covers the case $r=1$, we will not use induction hypothesis on $r$.
	This way we ensure that the basis of the induction will be proved.

	By \cite[Lemma~1]{a-HeckYam08} there exists $n\in \ndN _0$ with $n<m$ such
  that $\s _{i_n}\cdots \s _{i_2}\s _{i_1}^\chi (\Ndb _i+(r-1)\Ndb _j)
  =\Ndb _{i_{n+1}}$. Therefore induction hypothesis (on $m$) gives that
	\begin{align*}
		\LT _{i_m}&\cdots \LT _{i_2}\LT _{i_1}(\fie E^+_{i,r(j)})=
		\fie \Big(\LT _{i_m}\cdots \LT _{i_2}\LT _{i_1}(E_j)
		\LT _{i_m}\cdots \LT _{i_{n+2}}\LT _{i_{n+1}}(E_{i_{n+1}})\\
		&-
		(\LT _{i_m}\cdots \LT _{i_2}\LT _{i_1}(K_j)
    \actl \LT _{i_m}\cdots \LT _{i_{n+2}}\LT _{i_{n+1}}(E_{i_{n+1}}))
		\LT _{i_m}\cdots \LT _{i_2}\LT _{i_1}(E_j)\Big).
	\end{align*}
  By Lemma~\ref{le:Coxeterkey},
  $\s _{i_n}\cdots \s _{i_2}\s _{i_1}^\chi (\Ndb _j)=
  \Ndb _{i_n}+t\Ndb _{i_{n+1}}$ for some $t\in \ndN _0$.
	Since $n<m$, the second formula in Eq.~\eqref{eq:T(Eim)-4} together with
  relations $T_pT_p^-=\id $ for all $p\in I$ imply that
	\begin{align*}
		\LT _{i_m}\cdots \LT _{i_2}\LT _{i_1}(\fie E^+_{i,r(j)})=&\,
		\fie \LT _{i_m}\cdots \LT _{i_{n+1}}\Big(E^-_{i_n,t(i_{n+1})}
		E_{i_{n+1}}\\
		&\,-(K_{i_n}K_{i_{n+1}}^t\actl E_{i_{n+1}})E^-_{i_n,t(i_{n+1})}\Big).
	\end{align*}
	Using Lemma~\ref{le:E+E-} and Lemma~\ref{le:psiadE} one gets
	\begin{align*}
		\LT _{i_m}\cdots \LT _{i_2}\LT _{i_1}(\fie E^+_{i,r(j)})
		=&\,\fie \LT _{i_m}\cdots \LT _{i_{n+1}}(E^-_{i_n,t+1(i_{n+1})})\\
		=&\,\fie \LT _{i_m}\cdots \LT _{i_{n+2}}(E^+_{i_n,t'(i_{n+1})})
	\end{align*}
	for some $t'\in \ndN _0$. Now one has $m-n-1<m$, and hence induction
	hypothesis can be applied to the last formula to obtain the statement of the
	lemma for $E^+_{i,r(j)}$.

	The proof of the induction step for $E^-_{i,r(j)}$ goes analogously.
\end{proof}

\begin{corol}
	\label{co:T(Ei)}
  Assume Setting~\ref{se:ijseq}.
	Let $m\in \ndN _0$.
  Let $w=\s _{i_m}\cdots \s _{i_2}\s _{i_1}^\chi $.

  (i) If $m<M$, then
  for $E_j\in U(\chi )$
	\begin{align}\label{eq:T(Ei)-1}
		\LT _{i_m}\cdots \LT _{i_2}\LT _{i_1}(E_j)\in 
		U^+(w^*\chi )_{w(\Ndb _j)},\\
		\label{eq:T(Ei)-2}
		\LT ^-_{i_m}\cdots \LT ^-_{i_2}\LT ^-_{i_1}(E_j) \in 
		U^+(w^*\chi )_{w(\Ndb _j)}.
	\end{align}

  (ii) If $m=M$ then
	\begin{align}
		\fie \LT _{i_m}\cdots \LT _{i_2}\LT _{i_1}(E_j)
		=\fie \LT _{i_{m-1}}\cdots \LT _{i_1}\LT _{i_0}(E_j)
		=\fie F_{i_m}L_{i_m}^{-1}.
		\label{eq:T(Ei)-3}
	\end{align}
\end{corol}

\begin{proof}
	For (i) use that
  \begin{align*}
    \LT _{i_1}(E_j)=&\,E^+_{j,-c_{ij}(i)},&
    \LT ^-_{i_1}(\fie E_j)=&\,\fie E^-_{j,-c_{ij}(i)}
  \end{align*}
	and apply Prop.~\ref{pr:T(Eim)}.
	For (ii) use (i), equation
  $\s _{i_{m-1}}\cdots \s _{i_2}\s _{i_1}^\chi (\Ndb _j)=\Ndb _{i_m}$,
	Prop.~\ref{pr:Lusztig1}, and the definitions of $\LT _{i_1}$ and $\LT
	_{i_m}$.
\end{proof}

\begin{lemma}
	\label{le:TwEk}
  Assume Setting~\ref{se:ijseq}. Let $k\in I\setminus \{i,j\}$ and
  $m\in \ndN _0$. Let
  $w_m=\s _{i_1}\cdots \s _{i_{m-1}}\s _{i_m}^\chi $.
  If $m\le M$, then
  \begin{align*}
		\LT _{i_1}\cdots \LT _{i_m}(E_k)\in 
    U^+_{+i}(w_m^*\chi ) \cap U^+_{+j}(w_m^*\chi ),
	\end{align*}
  where $E_k\in U(\chi )$. If $m<M$, then
	\begin{align}
		\LT _{i_1}\cdots \LT _{i_m}(E_k)\in 
		U^+_{-j}(w_m^*\chi ).
		\label{eq:TwEk-1}
	\end{align}
\end{lemma}

\begin{proof}
	We proceed by induction on $m$. For $m=0$ the lemma clearly holds. Let now
	$m>0$. Then the relation
	\begin{align*}
		\LT _{i_1}\cdots \LT _{i_m}(E_k)= 
		\LT _{i_1}(\LT _{i_2}\cdots \LT _{i_m}(E_k))\in 
		U^+_{+i}(w_m^*\chi )
	\end{align*}
	follows immediately from Eq.~\eqref{eq:TwEk-1} and
	Lemma~\ref{le:psiadE}(d).
	According to Lemma~\ref{le:commEFi} and Prop.~\ref{pr:U+pchar} it
	remains to show that
	\begin{align}
		\label{eq:TwEk-2}
		[F_j,\LT _{i_1}\cdots \LT _{i_m}(E_k)]=&\, 0 & &\text{if $m<M$},\\
		\label{eq:TwEk-3}
		[F_j,\LT _{i_1}\cdots \LT _{i_m}(E_k)]\in &\,L_jU^+(w_m^*\chi ) &
    &\text{if $m=M<\infty $.}
	\end{align}
  By the first equation in Prop.~\ref{pr:Lusztig1}(ii),
	\begin{align*}
		[F_j,\LT _{i_1}\cdots \LT _{i_m}(E_k)]=&\,
		\LT _{i_1}\cdots \LT _{i_m}[\LT ^-_{i_m}\cdots \LT ^-_{i_1}(F_j),E_k].
	\end{align*}
	If $m<M$, then Cor.~\ref{co:T(Ei)} and
	Prop.~\ref{pr:Lusztig1} imply that the expression
	$\LT ^-_{i_m}\cdots \LT ^-_{i_1}(F_j)$
	lies in the subalgebra of $U^-(\chi )$ generated by $F_i$ and $F_j$.
	Thus the above commutator is zero and hence Eq.~\eqref{eq:TwEk-2}
	holds. On the other hand, if $m=M$, then 
	$$\LT ^-_{i_m}\cdots \LT ^-_{i_1}(\fie F_j)=\LT ^-_{i_m}(\fie F_{i_m}) =
	\fie E_{i_m}L_{i_m}^{-1}$$
	and hence
	\begin{align*}
		\fie [F_j,\LT _{i_1}\cdots \LT _{i_m}(E_k)]=&\,
		\fie \LT _{i_1}\cdots \LT _{i_m}[E_{i_{m}}L_{i_m}^{-1},E_k]\\
		= &\,\fie \LT _{i_1}\cdots \LT _{i_m}(L_{i_m}^{-1})
		\LT _{i_1}\cdots \LT _{i_m}(E^-_{k,1(i_{m})})\\
		= &\,\fie \LT _{i_1}\cdots \LT _{i_{m-1}}(L_{i_m})
		\LT _{i_1}\cdots \LT _{i_{m-1}}(E^+_{k,t(i_{m})})
	\end{align*}
	for some $t\in \ndN _0$.
  Since $\s _{i_1}\cdots \s _{i_{m-2}}\s _{i_{m-1}}^{r_{i_m}(\chi )}(\Ndb _{i_m})
	=\Ndb _j$,
	induction hypothesis and Cor.~\ref{co:T(Ei)} imply that
	Eq.~\eqref{eq:TwEk-3} holds for $m=M$.
\end{proof}

\begin{lemma}
	\label{le:T-TwEkinU++j}
  Assume Setting~\ref{se:ijseq}.
  Let $k\in I\setminus \{i,j\}$ and $m\in \ndN _0$. Assume that
  $m<M<\infty $.
  Then
	\begin{align*}
		\LT ^-_{i_1}\ldots \LT ^-_{i_m}\LT _{i_{m+1}}\cdots
		\LT _{i_{m+M}} (E_k) \in &\,
		U^+_{+i}(r_{i_1}\cdots r_{i_{m+M-1}} r_{i_{m+M}}(\chi ))\\
		&\, \cap 
    U^+_{+j}(r_{i_1}\cdots r_{i_{m+M-1}}r_{i_{m+M}}(\chi )),
	\end{align*}
  where $E_k\in U(\chi )$.
	Further, if $m>0$, then
	\begin{align*}
		\LT ^-_{i_1}\ldots \LT ^-_{i_m}\LT _{i_{m+1}}\cdots
		\LT _{i_{m+M}} (E_k) \in 
		U^+_{-i}(r_{i_1}\cdots r_{i_{m+M-1}} r_{i_{m+M}}(\chi )).
	\end{align*}
\end{lemma}

\begin{proof}
	If $m=0$, then the lemma holds by Lemma~\ref{le:TwEk}.
  Suppose now that $0<m<M$. Then by Prop.~\ref{pr:U+pchar} and
	Lemma~\ref{le:commEFi} it suffices to show
	that the following relations hold.
	\begin{gather}
		\label{eq:T-TwEk-1}
		\LT ^-_{i_1}\cdots \LT ^-_{i_m}\LT _{i_{m+1}}\cdots \LT
		_{i_{m+M}}(E_k)\in U^+(r_{i_1}\cdots r_{i_{m+M-1}}
    r_{i_{m+M}}(\chi )),\\
		\label{eq:T-TwEk-2}
		[F_i,\LT ^-_{i_1}\cdots \LT ^-_{i_m}\LT _{i_{m+1}}\cdots \LT
		_{i_{m+M}}(E_k)]=0,\\
		\label{eq:T-TwEk-3}
		\derK _j(\LT ^-_{i_1}\cdots \LT ^-_{i_m}\LT _{i_{m+1}}\cdots
    \LT _{i_{m+M}}(E_k))= 0.
	\end{gather}
	We proceed by induction on $m$. Induction hypothesis gives that
	\begin{align*}
		\LT ^-_{i_2}\cdots \LT ^-_{i_m}\LT _{i_{m+1}}\cdots \LT
		_{i_{m+M}}(E_k)\in U^+_{+i}(r_{i_2}\cdots r_{i_{m+M-1}}
		r_{i_{m+M}}(\chi )).
	\end{align*}
  Thus Lemma~\ref{le:psiadE}(d) implies
	Eq.~\eqref{eq:T-TwEk-1}. For Eq.~\eqref{eq:T-TwEk-2} one
	calculates
	\begin{align*}
		&\fie [F_i,\LT ^-_{i_1}\cdots \LT ^-_{i_m}\LT _{i_{m+1}}\cdots \LT
		_{i_{m+M}}(E_k)]\\
		&\, = \fie
		\LT ^-_{i_1}\cdots \LT ^-_{i_m}\LT _{i_{m+1}}\cdots \LT _{i_{m+M}}
		([\LT ^-_{i_{m+M}} \cdots \LT ^-_{i_{m+1}} \LT _{i_m} \cdots \LT _{i_1}
		(F_i) ,E_k])\\
		&\, = \fie
		\LT ^-_{i_1}\cdots \LT ^-_{i_m}\LT _{i_{m+1}}\cdots \LT _{i_{m+M}}
		([\LT ^-_{i_{m+M}} \cdots \LT ^-_{i_{m+1}} \LT _{i_m} \cdots \LT _{i_2}
		(K_i^{-1}E_i) ,E_k]).\\
		\intertext{Apply Cor.~\ref{co:T(Ei)}. One obtains
		the equations}
		&\, = \fie
		\LT ^-_{i_1}\cdots \LT ^-_{i_m}\LT _{i_{m+1}}\cdots \LT _{i_{m+M}}
		([\LT ^-_{i_{m+M-1}} \cdots \LT ^-_{i_m} \LT _{i_m} \cdots \LT _{i_2}
    (K_i^{-1}E_i) ,E_k])\\
		&\, = \fie
		\LT ^-_{i_1}\cdots \LT ^-_{i_m}\LT _{i_{m+1}}\cdots \LT _{i_{m+M}}
		([\LT ^-_{i_{m+M-1}} \cdots \LT ^-_{i_m} \LT _{i_m} \cdots \LT _{i_1}
		(F_i) ,E_k])\\
		&\, = \fie
		\LT ^-_{i_1}\cdots \LT ^-_{i_m}\LT _{i_{m+1}}\cdots \LT _{i_{m+M}}
		([\LT ^-_{i_{M-m-1}} \cdots \LT ^-_{i_0} (F_i) ,E_k]).
	\end{align*}
	Since $m\ge 1$, one gets $M-m<M$. Thus Cor.~\ref{co:T(Ei)}
	and Prop.~\ref{pr:Lusztig1} give that 
	$\LT ^-_{i_{M-m-1}} \cdots \LT ^-_{i_0} (F_i)$ is in the
	subalgebra of $U(r_{i_{M-m-1}} \cdots r_{i_1}r_{i_0}(\chi ))$
	generated by $F_i$ and $F_j$. Therefore the above commutator is zero and
	Eq.~\eqref{eq:T-TwEk-2} is proven.
		
	Eq.~\eqref{eq:T-TwEk-3} can be obtained from
	Prop.~\ref{pr:derTcomm}(ii) as follows.
	\begin{align*}
		&\fie \derK _j(\LT ^-_{i_1}\cdots \LT ^-_{i_m}\LT _{i_{m+1}}\cdots \LT
		_{i_{m+M}}(E_k))\\
		&\quad =
		\fie \LT _i^-\big( (\derK _i)^t\derK _j(\LT ^-_{i_2}\cdots \LT ^-_{i_m}
		\LT _{i_{m+1}}\cdots \LT _{i_{m+M}} (E_k))\big)=0,
	\end{align*}
	where $t\in \ndN _0$ is an appropriate integer and the last equation
	follows from Prop.~\ref{pr:U+pchar} and the induction hypothesis.
\end{proof}

\begin{theor}
	\label{th:LTCox}
	Let $\chi \in \cX $. Assume that $\chi '$ is $p$-finite for all
  $p\in I$,
  $\chi '\in \cG (\chi )$. Let $i,j\in I$ with $i\not=j$.
  Let $i_{2n+1}=i$ and $i_{2n}=j$ for all $n\in \ndZ $.
  Assume that
  $M=|R^\chi _+\cap (\ndN _0\Ndb _i+\ndN _0\Ndb _j)|<\infty $.
	Then there exists $\ula \in (\fienz )^I$ such that
	\begin{equation}
		\begin{aligned}
			\LT _{i_M}\cdots \LT _{i_2}\LT _{i_1}=&\,
      \LT _{i_{M-1}}\cdots \LT _{i_1}\LT _{i_0} \varphi _{\ula }
		\label{eq:LTCox}
		\end{aligned}
	\end{equation}
   as algebra isomorphisms $U(\chi )\to U(r_{i_M}\cdots
   r_{i_2}r_{i_1}(\chi ))$.
\end{theor}

\begin{proof}
	By Prop.~\ref{pr:Lusztig1} the statement of the theorem is equivalent
	to
	\begin{align}
		\label{eq:LTCox-1}
		\LT _{i_M}\cdots \LT _{i_2}\LT _{i_1}(\fie E_k)=
		\LT _{i_{M-1}}\cdots \LT _{i_1}\LT _{i_0} (\fie E_k)\quad
		\text{for all $k\in I$}.
	\end{align}
	By Cor.~\ref{co:T(Ei)} the above equation is fulfilled for $k\in
	\{i,j\}$.  Suppose now that $k\notin \{i,j\}$.
	Then Lemma~\ref{le:T-TwEkinU++j} (for $m=M-1$ and $i,j$
  interchanged)
  and Lemma~\ref{le:psiadE}(d) imply that
	$$ \LT ^-_{i_1}\ldots \LT ^-_{i_M}\LT _{i_{M+1}}\cdots \LT _{i_{2M}}(E_k)\in
  U^+(r_{i_1}\cdots r_{i_{2M-1}} r_{i_{2M}}(\chi )).$$
  By Thm.~\ref{th:Coxgr},
  $\s _{i_1}\cdots \s _{i_{2M-1}}\s _{i_{2M}}^\chi =\id $.
	Hence
  \[ \LT ^-_{i_1}\ldots
	\LT ^-_{i_M}\LT _{i_{M+1}}\cdots \LT _{i_{2M}}(E_k)\in U^+(\chi )
  _{\Ndb _k}, \]
	and therefore
	$\LT ^-_{i_1}\cdots \LT ^-_{i_M}\LT _{i_{M+1}}\cdots
  \LT _{i_{2M}}(E_k)\in \fienz E_k$.
  Using equations $\LT _p\LT _p^-=\id $ from
  Prop.~\ref{pr:Lusztig1}(ii), where $p\in \{i,j\}$,
	one gets Eq.~\eqref{eq:LTCox-1} for $k\notin \{i,j\}$.
\end{proof}

\begin{theor}\label{th:wEpinU+}
	Let $\chi \in \cX $. Assume that $\chi '$ is $i$-finite for all
  $i\in I$, $\chi '\in \cG (\chi )$.
  Let $m\in \ndN _0$ and $i_1,\dots,i_m\in I$
  such that $w=\s _{i_m}\cdots \s _{i_2}\s _{i_1}^\chi
  \in \Hom (\Wg (\chi ))$ is a reduced expression.
  Let $p\in I$. Assume that $w(\Ndb _p)\in R^{w^*\chi }_+$.
	Then
  \begin{align}\label{eq:TTE}
		\LT _{i_m}\cdots \LT _{i_1}(E_p)\in U^+(w^*\chi ),
	\end{align}
  where $E_p\in U(\chi )$.
\end{theor}

\begin{proof}
	We proceed by induction on $m$. If $m=0$, then there is nothing to prove.
	If $m=1$, then $i_1\not=p$ by assumption and hence the theorem
	holds by definition of $\LT _{i_1}$.

	Assume now that $m\ge 2$ and that the theorem is true for all smaller
  values of $m$. Then again $p\not=i_1$ by \cite[Cor.~3]{a-HeckYam08}.
	Let $j_{2n}=p$ and $j_{2n+1}=i_1$
	for all $n\in \ndN _0$. Let $r\in \ndN _0$ be maximal with respect
	to the property that $\ell (w_r)=m-r$, where $w_r=w \s _{j_1} \s _{j_2}\cdots
	\s _{j_r}$.
  Then $1\le r\le m$, since
  $w \s _{j_1}=\s _{i_m}\cdots \s _{i_3}\s _{i_2}^{r_{i_1}(\chi )}$.
	Further, $\ell (w_r \s _{j_{r+1}})=m-r+1$ by the
  maximality of $r$ and $\ell (w_r \s _{j_r})=m-r+1$ since $w_r\s _{j_r}=w_{r-1}$.
	Let $k_1,\ldots ,k_{m-r}\in I$ such that
  $$w_r=\s _{k_{m-r}}\cdots \s _{k_2}\s _{k_1}^{(w'_r)^*\chi },\quad
  \text{where }w'_r=\s _{j_r}\cdots \s _{j_2}\s _{j_1}^\chi .$$
	Then $w=w_rw'_r$,
  and hence Thm.~\ref{th:LTCox} and Matsumoto's theorem, see
  \cite[Thm.~5]{a-HeckYam08}, imply that
  \begin{align*}
    \LT _{k_{m-r}}\cdots \LT _{k_1}\LT _{j_r}\cdots \LT _{j_1}=
    \LT _{i_m}\cdots \LT _{i_2}\LT _{i_1}\varphi _{\ula }
  \end{align*}
  for some $\ula \in (\fienz )^I$.
  Further, the assumption $w(\Ndb _p)\in R^{w^*\chi }_+$ implies that
  $\s _{j_r}\cdots \s _{j_2}\s _{j_1}^\chi (\Ndb _p)\in
  R^{(w'_r)^*\chi }_+$, and hence
	$\LT _{j_r}\cdots \LT _{j_1}(E_p)$ lies in the subalgebra of
  $U^+( (w'_r)^*\chi )$
	generated by $E_p$ and $E_{i_1}$. Since $m-r<m$, induction hypothesis
	implies that
  \begin{align*}
	  \LT _{k_{m-r}}\cdots \LT _{k_1}(E_{k_0})\in U^+(w^*\chi ) \quad
	  \text{for $k_0\in \{p,i_1\}$.}
  \end{align*}
	Since $ \LT _{k_{m-r}}\cdots \LT _{k_1}$ is an algebra map,
  Eq.~\eqref{eq:TTE} holds for $m$.
\end{proof}

Recall the algebra map $\varphi _\tau :U(\chi )\to U(\tau ^*\chi )$
defined in Prop.~\ref{pr:algiso}, where
$\tau $ is a permutation of $I$. Let $\hat{\tau }$ be the automorphism of
$\ndZ ^I$ given by $\hat{\tau }(\Ndb _p)=\Ndb _{\tau (p)}$ for all $p\in
I$. Note that for any $\chi \in \cX $ and all $i,j\in I$ one
has $\chi (-\Ndb _i,-\Ndb _j)=\chi (\Ndb _i,\Ndb _j)$, and hence
$(-\id)^*\chi =\chi $. Further, if $\chi '$ is $p$-finite for all
$\chi '\in \cG (\chi )$ and $p\in I$, then for each
$\chi '\in \cG (\chi )$ there is a unique longest element
$w_0\in \Hom (\chi ',\underline{\,\,})\subset \Hom (\Wg (\chi ))$,
see \cite[Cor.~5]{a-HeckYam08}.

\begin{corol}\label{co:Tw0}
    Let $\chi \in \cX $. Assume that $M=|R^\chi _+|$ is finite.
    Let $i_1,\ldots ,i_M\in I$ such that
    $w_0=\s _{i_M}\ldots \s _{i_2}\s _{i_1}^\chi $ is a longest
    element of $\Hom (\Wg (\chi ))$.
    Then there exists $\ullam \in (\fienz )^I$ and a
    permutation $\tau $ of $I$ such that
    $w_0=-\hat{\tau }$ and
    \begin{align*}
        \LT _{i_M}\cdots \LT _{i_2}\LT _{i_1} =
        \phi _1 \circ \varphi _\tau \circ \varphi _{\ullam }
    \end{align*}
    as algebra maps $U(\chi )\to U(w_0^*\chi )$.
\end{corol}

\begin{proof}
  Since $w_0(R^\chi _+)=-R^{w_0^*\chi }_+$, there exists a unique
    permutation $\tau $ of $I$ such that $w_0(\Ndb _i)=-\Ndb _{\tau (i)}$
    for all $i\in I$.

    Let $p\in I$. Since $w_0$ has maximal length, $\ell (\s _p w_0)=M-1$.
    Let $j_1,\ldots ,j_{M-1}\in I$ such that $\s _p\s _{j_{M-1}}\cdots
    \s _{j_2}\s _{j_1}^\chi =w_0$. By Thm.~\ref{th:LTCox}
    \begin{align*}
        \LT _{w_0}:=
        \LT _{i_M}\cdots \LT _{i_1}=\LT _p\LT _{j_{M-1}}\cdots \LT
        _{j_1}\varphi _{\ula }
    \end{align*}
    for some $\ula \in (\fienz )^I$.
    Further, Thm.~\ref{th:wEpinU+} and relation $\s _pw_0(\Ndb _{\tau
    ^{-1}(p)})=\Ndb _p$ imply that there exists $\lambda _{\tau ^{-1}(p)}\in
    \fienz $ such that
    $$\LT _{j_{M-1}}\cdots \LT _{j_1}\varphi _{\ula }(E_{\tau ^{-1}(p)})
    = \lambda _{\tau ^{-1}(p)}E_p.$$
    Thus
    \begin{align*}
        \LT _{w_0}(E_{\tau ^{-1}(p)})=&\,\LT _p\LT _{j_{M-1}}\cdots \LT
        _{j_1}\varphi _{\ula }(E_{\tau ^{-1}(p)})\\
        =&\,\lambda _{\tau ^{-1}(p)}\LT _p(E_p)
        =\lambda _{\tau ^{-1}(p)}F_pL_p^{-1}.
    \end{align*}
    Put $\ullam =(\lambda _i)_{i\in I}$.
    Similarly to the above arguments one can show that
    for each $i\in I$ there exists $\mu _i\in \fienz $ such that
    \begin{gather*}
        \LT _{w_0}(K_i) = \phi _1(\varphi _\tau (\varphi _{\ullam }(K_i))),
        \quad
        \LT _{w_0}(L_i) = \phi _1(\varphi _\tau (\varphi _{\ullam }(L_i))),\\
        \LT _{w_0}(F_i) = \mu _i\phi _1(\varphi _\tau (\varphi _{\ullam
        }(F_i))).
    \end{gather*}
    Since $\LT _{w_0}$ is an algebra map, one
    obtains that $\mu _i=1$ for all $i\in I$. This proves the corollary.
\end{proof}

\section{A characterization of Nichols algebras of diagonal type}
\label{sec:charNich}

The following theorem, which is
an application of the Lusztig isomorphisms constructed in the previous
section, gives a
characterization of Nichols algebras of diagonal type with finite root
system.

\begin{theor}
	\label{th:Nichchar}
	Let $\chi \in \cX $. Assume that $R^\chi _+$ is finite.
	For all $\chi '\in \cG (\chi )$ let $\cJ ^+(\chi ')$ be an ideal of
	$\cU ^+(\chi ')$ such that
	\begin{gather*}
    \coun (\cJ ^+(\chi '))=\{0\},\quad
    X\actl \cJ ^+(\chi ')\subset \cJ ^+(\chi '),\quad
		\cI ^+_p(\chi ')\subset \cJ ^+(\chi '),\\
		\derK _p(\cJ ^+(\chi '))\subset \cJ ^+(\chi '),\quad
		\derL _p(\cJ ^+(\chi '))\subset \cJ ^+(\chi ')
	\end{gather*}
  for all $X\in \cU ^0(\chi ')$, $p\in I$.
	For all $\chi '\in \cG (\chi )$ let
  $\cJ (\chi ')$ and $\tJ (\chi ')$ be the ideals of $\cU (\chi ')$
	generated by $\cJ ^+(\chi ')+\phi _4(\cJ ^+(\chi '))$ and
	$\cJ ^+(\chi ')+\phi _4(\cS ^+(\chi '))$, respectively.
	The following are equivalent.
	\begin{enumerate}
		\item
      $\cU ^+(\chi ')/\cJ ^+(\chi ')=U^+(\chi ')$
      for all $\chi '\in \cG (\chi )$.
		\item $\cJ ^+(\chi ')=\cS ^+(\chi ')$
      for all $\chi '\in \cG (\chi )$.
    \item The algebra maps $\LT _p:\cU (\chi ')\to \cU (r_p(\chi '))
			/\cJ (r_p(\chi '))$ satisfy
			\begin{align}
				\label{eq:Nichchar-1}
				\LT _p(\cJ ^+(\chi '))=\{0\}\qquad
				\text{for all $\chi '\in \cG (\chi )$, $p\in I$.}
			\end{align}
		\item The algebra maps $\LT _p:\cU (\chi ')\to \cU (r_p(\chi '))
      /\tJ (r_p(\chi '))$ satisfy
			\begin{align}
				\label{eq:Nichchar-2}
				\LT _p(\tJ (\chi '))=\{0\}\qquad
				\text{for all $\chi '\in \cG (\chi )$, $p\in I$.}
			\end{align}
		\item The algebra maps $\LT ^-_p:\cU (\chi ')\to \cU (r_p(\chi '))
      /\cJ (r_p(\chi '))$ satisfy
			\begin{align}
				\label{eq:Nichchar-3}
				\LT ^-_p(\cJ ^+(\chi '))=\{0\}\qquad
				\text{for all $\chi '\in \cG (\chi )$, $p\in I$.}
			\end{align}
		\item The algebra maps $\LT ^-_p:\cU (\chi ')\to \cU
			(r_p(\chi '))/\tJ (r_p(\chi '))$ satisfy
			\begin{align}
				\label{eq:Nichchar-4}
				\LT ^-_p(\tJ (\chi '))=\{0\}\qquad
				\text{for all $\chi '\in \cG (\chi )$, $p\in I$.}
			\end{align}
	\end{enumerate}
\end{theor}

If the statements in Thm.~\ref{th:Nichchar} are fulfilled, then
because of Thm.~\ref{th:Liso} the
algebra maps $\LT _p$, $\LT ^-_p$ in statements~(3) and (5)
induce isomorphisms $\cU (\chi ')/\cJ (\chi ')\to \cU (r_p(\chi '))
/\cJ (r_p(\chi '))$ for all $\chi '\in \cG (\chi )$, $p\in I$.

\begin{proof}
	The equivalence of claims~(1) and (2) is the definition of
	$U^+(\chi ')$. The implications (2)$\Rightarrow $(3) and (2)$\Rightarrow
	$(5) have been proven in Thm.~\ref{th:Liso}.

	Next we prove the implication (3)$\Rightarrow $(4).
  Let $\chi '\in \cG (\chi )$ and $p\in I$. Then
	$\cJ ^+(\chi ')\subset \cS ^+(\chi ')$ by
  Prop.~\ref{pr:Nicholschar2}. Thus one has to show that the maps
  $\LT _p:\cU (\chi ')\to \cU (r_p(\chi '))/\cJ (r_p(\chi '))$
	satisfy
	\begin{align}
        \label{eq:Nichchar-5}
		\LT _p(\phi _4(\cS ^+(\chi ')))\subset
		(\phi _4(\cS ^+(r_p(\chi ')))+\cJ (r_p(\chi ')))/\cJ (r_p(\chi ')).
	\end{align}
    Equivalently, since $\phi _4(\cJ (r_p(\chi ')))=\cJ (r_p(\chi '))$,
    the last equation in Prop.~\ref{pr:Lusztig1}(ii) and
    Lemma~\ref{le:S+homog} imply that relation~\eqref{eq:Nichchar-5}
    is equivalent to
	\begin{align*}
		\LT ^-_p(\cS ^+(\chi '))\subset
		(\cS ^+(r_p(\chi '))+\cJ (r_p(\chi ')))/\cJ (r_p(\chi ')).
	\end{align*}
    Further, by Lemma~\ref{le:S+gen}(ii) it suffices to check the
    following inclusions.
    \begin{align}
        \label{eq:Nichchar-6}
        \LT ^-_p(\cS ^+(\chi ')\cap \cU ^+_{+p}(\chi '))\subset &\,
		(\cS ^+(r_p(\chi '))+\cJ (r_p(\chi ')))/\cJ (r_p(\chi ')),\\
        \LT ^-_p(\cS ^+(\chi ')\cap \fie [E_p])=&\,\{0\}.
        \label{eq:Nichchar-7}
    \end{align}
    Now relation~\eqref{eq:Nichchar-6} follows from Lemma~\ref{le:psiadE}(d)
    and Thm.~\ref{th:Liso}. Finally, Eq.~\eqref{eq:Nichchar-7} is a consequence
    of Eq.~\eqref{eq:S+capEp} and the assumption
    $\cJ ^+_p(\chi ')\subset \cJ ^+(\chi ')$. Thus the implication
    (3)$\Rightarrow $(4) is proven.

    We finish the proof of the theorem with showing the implication
    (4)$\Rightarrow $(2). The remaining open implication (6)$\Rightarrow $(2)
    can be proven in a similar way.

    Let $\chi '\in \cG (\chi )$.
    Since $R^{\chi '}_+$ is finite, there exists a longest element
    $w_0\in \Hom (\chi ',\underline{\,\,})\subset \Hom (\Wg (\chi ))$.
    Let $M=|R^{\chi '}_+|$ and $i_1,\dots ,i_M\in I$ such that
    $\s _{i_M}\cdots \s _{i_2}\s_{i_1}^{\chi '}$ is a reduced expression of $w_0$.
    By the assumption of statement~(4) the map
    $\LT _{w_0}:=\LT _{i_M}\cdots \LT _{i_1}:
    \cU (\chi ')\to \cU (w_0^*\chi ')/\tJ (w_0^*\chi ')$ is well-defined
    and satisfies
    $$ \LT _{w_0}(\tJ (\chi '))=\{0\}.$$
    In particular, $w_0(R^{\chi '}_+)=-R^{w_0^*\chi '}_+$ implies that
    \begin{align}\label{eq:Tw0(S-)}
        \LT _{w_0}(\phi _4(\cS ^+(\chi ')))=\{0\}.
    \end{align}
    Because of the relations
    $\cI ^+_p(\chi ')+\phi _4(\cI ^+_p(\chi '))\subset \cJ (\chi ')$
    the result of Cor.~\ref{co:Tw0} holds also for $\LT _{w_0}$, namely
    \begin{align*}
        \LT _{w_0}=\phi _1\circ \varphi _\tau \circ \varphi _{\ullam }
    \end{align*}
    for some $\ullam \in (\fienz )^I$ and a permutation $\tau $ of $I$.
    Thus Eq.~\eqref{eq:Tw0(S-)} gives that
    \begin{align*}
        \phi _1(\varphi _\tau (\varphi _{\ullam }(\phi _4(\cS ^+(\chi
        ')))))
        \subset \cJ ^+(w_0^*\chi ')\cU ^0(w_0^*\chi '),
    \end{align*}
    and hence $\cS ^+(w_0^*\chi ')\subset \cJ ^+(w_0^*\chi ')$ by
    Prop.~\ref{pr:commiso} and Lemma~\ref{le:phiS+}.
    This proves the implication (4)$\Rightarrow $(2).
\end{proof}

We are going to give an application of Thm.~\ref{th:Nichchar}, see
Ex.~\ref{ex:finCartan}.
Owing to the fact that the representation theory is not yet developed,
for the proof a couple of technical formulas are used, which can be
obtained by standard techniques.

\begin{lemma}
   \label{le:adEpXY}
  Let $\chi \in \cX $, $\mu \in \ndZ ^I$, and
  $p\in I$.
  Then for all $m\in \ndN _0$ and all
  $X\in \cU (\chi )_\mu $ and $Y\in \cU (\chi )$ one has
  \begin{align*}
    (\ad E_p)^m(XY)=\sum _{n=0}^m\chi (n\Ndb _p,\mu )\qchoose{m}{n}{q_{p p}}
    (\ad E_p)^{m-n}X\cdot (\ad E_p)^nY.
  \end{align*}
\end{lemma}

\begin{proof}
  The algebra $\cU (\chi )$ is a module algebra with respect
  to the adjoint action $\ad $ of $\cU (\chi )$, and hence
  \begin{align*}
    (\ad Z)(XY)=(\ad Z_{(1)})X\cdot (\ad Z_{(2)})Y \quad
    \text{for all $Z\in \cU (\chi )$}.
  \end{align*}
  Then Rem.~\ref{re:brcopr}, Lemma~\ref{le:brcoprE}(i), and
  Eqs.~\eqref{eq:KErel} and \eqref{eq:KFrel}
  imply the claim.
\end{proof}

\begin{corol}
  \label{co:adEp2}
  Let $\chi \in \cX $ and $p,i\in I$ such that
  $p\not=i$ and $q_{p p}^{-c_{p i}}q_{p i}q_{i p}=1$. Then for any $\ndZ
  ^I$-homogeneous element
  $Y\in (\cU ^+_{+p}(\chi )+\cI ^+_p(\chi ))/\cI ^+_p(\chi )$
  with $(\ad E_p)^{r+1}Y=0$ for some $r\in \ndN _0$ one has
  \begin{align*}
    (\ad E_p)^{-c_{p i}+r}(E_iY-(K_i\actl &\,Y)E_i)\\
    = \qchoose{-c_{p i}+r}{r}{q_{p p}}q_{p i}^r
    &\,\big( E^+_{i,-c_{p i}}\cdot (\ad E_p)^rY\\
    &- (K_iK_p^{-c_{p i}}\actl (\ad E_p)^rY)E^+_{i,-c_{p i}}\big).
  \end{align*}
\end{corol}

\begin{proof}
  The left adjoint action of $\cU (\chi )$ induces an action on the algebra
  $(\cU ^+_{+p}(\chi )+\cI ^+_p(\chi ))/\cI ^+_p(\chi )$.
  Thus Lemma~\ref{le:adEpXY},
  Eq.~\eqref{eq:KErel}, and relations $(\ad E_p)^{1-c_{p i}}E_i=(\ad
  E_p)^{r+1}Y=0$ give that
  \begin{align*}
    (\ad E_p)^{-c_{p i}+r}(E_iY) = &\qchoose{-c_{p i}+r}{-c_{p i}}{q_{p p}}
    q_{p i}^r E^+_{i,-c_{p i}}\cdot (\ad E_p)^rY,\\
    (\ad E_p)^{-c_{p i}+r}( (K_i\actl Y)E_i) = &
    \qchoose{-c_{p i}+r}{r}{q_{p p}}
    ((\ad E_p)^r(K_p^{-c_{p i}}K_i\actl Y)) E^+_{i,-c_{p i}}.
  \end{align*}
  The condition on $\chi $ in the corollary gives the equation
  $q_{p i}^rK_iK_p^{-c_{p i}}E_p^r= E_p^rK_iK_p^{-c_{p i}}$ which implies the
  claim.
\end{proof}

\begin{examp} \label{ex:finCartan}
  It was proven already by Lusztig \cite[Thm.\,33.1.3]{b-Lusztig93}
  that for quantized symmetrizable Kac-Moody algebras $U_q(\lag )$,
  defined over the field $\ndQ (q)$, Serre-relations (the generators of the
  ideals $\cI ^+_p(\chi )$) are sufficient to define the ideal $\cS ^+(\chi
  )$. A careful choice of related results on Kac-Moody algebras leads to the
  proof of this statement even if $q$ is not a root of $1$, see
  \cite{p-HeckKolb06a}. Using twisting of Nichols algebras,
  see \cite[Prop.\,3.9, Rem.\,3.10]{inp-AndrSchn02} one can show that
  the analogous statement holds for
  multiparameter quantizations of Kac-Moody algebras over fields of
  characteristic zero. In this example an easy application of
  Thm.~\ref{th:Nichchar} is demonstrated on multiparameter quantizations of
  semisimple Lie algebras. As an improvement compared to
  \cite{b-Lusztig93} it is allowed that $\fie $ is an arbitrary field.

Let $\chi \in \cX $. Assume that $R^\chi _+$ is finite, and that
$\qnum{m}{q_{i i}}\not=0$ for all $m\in \ndN $, $i\in I$.
Thus $\chi $ is of (finite) Cartan type, that is, there is
a symmetrizable Cartan matrix
$C=(c_{i j})_{i,j\in I}$ of finite type
such that
\begin{align}
  q_{i i}^{-c_{i j}}q_{i j}q_{j i}=1
  \label{eq:Cartantype}
\end{align}
for all $i,j\in I$.
In this case $C^{\chi '}=C$ for all $\chi '\in \cG (\chi )$.

Thm.~\ref{th:Nichchar} characterizes $U^+(\chi )$ which is
the upper triangular part of the multiparameter version of a
Drinfel'd--Jimbo algebra. In the present setting it can be easily proven
that the ideal $\cS ^+(\chi )$ is generated by the Serre relations, that is
\begin{align}\label{eq:Serre}
  \cS ^+(\chi )=\sum _{p\in I}\cI ^+_p(\chi ).
\end{align}
Indeed, by Def.~\ref{de:Ip+} and Thm.~\ref{th:Nichchar}(3)$\Rightarrow $(2)
one has to check that
\begin{align} \label{eq:TiE+}
  \LT _p( (\ad E_i)^{1-c_{i j}}E_j)=0 \quad \text{for all $i,j,p\in I$ with
  $i\not=j$.}
\end{align}
If $p=i$, then Eq.~\eqref{eq:TiE+} follows from Lemmata~\ref{le:E+=E-} and
\ref{le:psiadE}(c).
If $p\not=i$ and $p\not=j$, then one gets
\begin{align*}
  \LT _p( (\ad E_i)^{1-c_{i j}}E_j)=
  \big(\widetilde{\ad }\LT _p(E_i)\big)^{1-c_{i j}}\LT _p(E_j)
  = \big(\widetilde{\ad }E^+_{i,-c_{p i}}\big)^{1-c_{i j}}E^+_{j,-c_{p j}},
\end{align*}
where
\begin{align*}
  (\widetilde{\ad }\LT _p(E_i))X=\LT _p(E_i)X-(K_iK_p^{-c_{p i}}\actl X)
  \LT _p(E_i).
\end{align*}
Thus equations $E^+_{i,1-c_{p i}}=E^+_{j,1-c_{p j}}=0$ and
Cor.~\ref{co:adEp2}, which has to be applied $1-c_{i j}$ times,
imply that
\begin{align*}
  \LT _p( (\ad E_i)^{1-c_{i j}}E_j)\in &\,
  \fienz (\ad E_p)^{-c_{p i}(1-c_{i j})-c_{p j}}
  \big((\ad E_i)^{1-c_{i j}}E_j\big)
  =\{0\}.
\end{align*}

It remains to consider the case $j=p\not=i$.
If $c_{i j}=0$, then in all algebras $\cU (\chi ')$ with $\chi '\in
\cG (\chi )$ we have
\begin{align*}
  E_iE_p-(K_i\actl E_p)E_i=
  E_iE_p-(L_i\actl E_p)E_i\in \fienz (E_pE_i-(K_p\actl E_i)E_p).
\end{align*}
This case was considered below Eq.~\eqref{eq:TiE+}.
Thus, since $R^\chi _+$ is finite, it remains to consider the case
\begin{align*}
  \min \{c_{p i},c_{i p}\}\in \{-1,-2,-3\},\qquad \max \{c_{p i}, c_{i p}\}=-1.
\end{align*}
We are going to show that
\begin{align}\label{eq:TpSerre}
  \fie \LT _p( (\ad E_i)^{1-c_{i p}}E_p)=
  \fie (\ad E_p)^{-c_{p i}(1-c_{i p})-2}(\ad 'E_i)^{1-c_{i p}}E_p,
\end{align}
where $(\ad 'E_i)X=E_iX-(L_i\actl X)E_i$. In fact, $\ad '$ can be considered
as the left adjoint action of $\cU (\chi )$ on itself via a second Hopf
algebra structure of $\cU (\chi )$, but we will not use this structure.
Further, Lemma~\ref{le:E+=E-} gives that the above equality finishes the proof
of Eq.~\eqref{eq:Serre}. 

\textit{WARNING!!!} Since $\chi $ is not symmetric, the structure constants of
$\chi $ and $r_p(\chi )$ do not coincide. Without loss of generality
we may assume that both sides of Eq.~\eqref{eq:TpSerre} are in 
$(\cU ^+_{+p}(\chi )+\cI ^+_p(\chi ))/\cI ^+_p(\chi )$, and hence in both
expressions we may use the structure constants of $\chi $.

On the one hand we have
\begin{align*}
  \fie \LT _p( (\ad E_i)^{1-c_{i p}}E_p)=&\,
  \fie \LT _p( (\ad E_i)^{-c_{i p}}E^-_{i,1})\\
  =&\,\fie (\widetilde{\ad } E^+_{i,-c_{p i}})^{-c_{i p}}E^+_{i,-c_{p i}-1}.
\end{align*}
For this we can give an explicit formula by performing in Eq.~\eqref{eq:Eim+}
the following replacements:
\begin{align*}
  E_p\mapsto & E^+_{i,-c_{p i}},&
  E_i\mapsto & E^+_{i,-c_{p i}-1},&
  K_p\mapsto & K_iK_p^{-c_{p i}},\\
  q_{p p}\mapsto & q_{i i}, &
  q_{p i}\mapsto & q_{i i}q_{p i}, &
  m\mapsto & -c_{i p}.
\end{align*}
One obtains that
\begin{equation}
  \begin{aligned}
    &\fie \LT _p( (\ad E_i)^{1-c_{i p}}E_p)\\
    &=\fie 
    \sum _{s=0}^{-c_{i p}}(-q_{p i})^s q_{i i}^{s(s+1)/2}
    \qchoose{-c_{i p}}{s}{q_{i i}}(E^+_{i,-c_{p i}})^{-c_{i p}-s}
    E^+_{i,-c_{p i}-1}(E^+_{i,-c_{p i}})^s.
  \end{aligned}
  \label{eq:TpSerre2}
\end{equation}
Let first $c_{i p}=-1$. Then $q_{i i}q_{i p}q_{p i}=1$, and hence
Lemma~\ref{le:adEpXY} yields that
{\allowdisplaybreaks
\begin{align*}
  &\fie (\ad E_p)^{-2c_{p i}-2}(\ad 'E_i)^2E_p
  \displaybreak[0]\\
  &\quad =
  \fie (\ad E_p)^{-2c_{p i}-2}(\ad 'E_i)E^+_{i,1})
  \displaybreak[2]\\
  &\quad =
  \fie (\ad E_p)^{-2c_{p i}-2}(E_i E^+_{i,1}-q_{i p} E^+_{i,1} E_i)
  \displaybreak[2]\\
  &\quad =
  \fie \bigg( \qchoose{-2c_{p i}-2}{-c_{p i}-1}{q_{p p}}
  (q_{p i}^{-c_{p i}-1}E^+_{i,-c_{p i}-1}E^+_{i,-c_{p i}}
  \displaybreak[0]\\
  &\qquad \qquad \qquad 
  -q_{i p}q_{p i}^{-c_{p i}-1}q_{p p}^{-c_{p i}-1}E^+_{i,-c_{p i}}
  E^+_{i,-c_{p i}-1})
  \displaybreak[0]\\
  &\qquad
  + \qchoose{-2c_{p i}-2}{-c_{p i}}{q_{p p}}(
  q_{p i}^{-c_{p i}-2} E^+_{i,-c_{p i}}E^+_{i,-c_{p i}-1}
  \displaybreak[0]\\
  &\qquad \qquad \qquad 
  -q_{i p}q_{p i}^{-c_{p i}}q_{p p}^{-c_{p i}}
  E^+_{i,-c_{p i}-1} E^+_{i,-c_{p i}})
  \bigg).\\
  \intertext{Using Lemma~\ref{le:qchooserel} this gives}
  &\quad =
  \fie \qchoose{-2c_{p i}-2}{-c_{p i}-1}{q_{p p}}
  \frac{1}{\qnum{-c_{p i}}{q_{p p}}}\big(
  q_{p i}^{-c_{p i}-1}
  \qnum{-c_{p i}}{q_{p p}}
  E^+_{i,-c_{p i}-1}E^+_{i,-c_{p i}}\\
  &\qquad \qquad \qquad 
  -q_{i p}q_{p i}^{-c_{p i}-1}q_{p p}^{-c_{p i}-1}
  \qnum{-c_{p i}}{q_{p p}}
  E^+_{i,-c_{p i}} E^+_{i,-c_{p i}-1}\\
  &\qquad \qquad \qquad 
  +q_{p i}^{-c_{p i}-2}
  \qnum{-c_{p i}-1}{q_{p p}}
  E^+_{i,-c_{p i}}E^+_{i,-c_{p i}-1}\\
  &\qquad \qquad \qquad 
  -q_{i p}q_{p i}^{-c_{p i}}q_{p p}^{-c_{p i}}
  \qnum{-c_{p i}-1}{q_{p p}}
  E^+_{i,-c_{p i}-1} E^+_{i,-c_{p i}}
  \big).\\
  \intertext{By Eq.~\eqref{eq:Cartantype}
  one has $q_{p p}^{-c_{p i}}q_{i p}q_{p i}=1$,
  and hence we conclude that}
  &\quad = \fie ( -q_{p i}^{-c_{p i}-2}q_{p p}^{-1}
  E^+_{i,-c_{p i}}E^+_{i,-c_{p i}-1}
  +q_{p i}^{-c_{p i}-1}q_{p p}^{-c_{p i}-1}
  E^+_{i,-c_{p i}-1}E^+_{i,-c_{p i}}
  ).
\end{align*}
}
The latter formula coincides with the one in Eq.~\eqref{eq:TpSerre2}
if $c_{i p}=-1$.

Let now $c_{p i}=-1$ and $c_{i p}=-2$. Then
\begin{align*}
  &(\ad E_p)(\ad 'E_i)^2 E^+_{i,1}\\
  &\quad =(\ad E_p)(\ad 'E_i)(E_i E^+_{i,1}
  -q_{i i}^{-1}q_{p i}^{-1} E^+_{i,1} E_i)\\
  &\quad =(\ad E_p)(E_i^2 E^+_{i,1}
  -(q_{i i}^{-1}+q_{i i}^{-2})q_{p i}^{-1}E_i E^+_{i,1} E_i
  +q_{i i}^{-3}q_{p i}^{-2} E^+_{i,1} E_i^2)
  \\
  &\quad =E^+_{i,1}E_i E^+_{i,1}+q_{p i} E_i (E^+_{i,1})^2\\
  &\quad \quad
  -(q_{i i}^{-1}+q_{i i}^{-2})q_{p i}^{-1}(E^+_{i,1})^2 E_i 
  -(q_{i i}^{-1}+q_{i i}^{-2})q_{p i}q_{p p} E_i (E^+_{i,1})^2\\
  &\quad \quad
  +q_{i i}^{-3}q_{p i}^{-1}q_{p p} (E^+_{i,1})^2 E_i
  +q_{i i}^{-3}q_{p p} E^+_{i,1} E_i E^+_{i,1}\\
  &\quad =
  -q_{i i}^{-2}q_{p i}^{-1} (E^+_{i,1})^2 E_i
  +(1+q_{i i}^{-1}) E^+_{i,1} E_i E^+_{i,1}
  -q_{i i} q_{p i} E_i (E^+_{i,1})^2.
\end{align*}
Similarly, if
$c_{p i}=-1$ and $c_{i p}=-3$, then
{\allowdisplaybreaks
\begin{align*}
  &(\ad E_p)^2(\ad 'E_i)^3 E^+_{i,1}\\
  &\,=(\ad E_p)^2 \big( E_i^3 E^+_{i,1}
  -(q_{i i}^{-1} +q_{i i}^{-2} + q_{i i}^{-3}) q_{p i}^{-1} E_i^2 E^+_{i,1} E_i
  \\
  &\,\quad \quad
  +(q_{i i}^{-3} +q_{i i}^{-4} + q_{i i}^{-5}) q_{p i}^{-2} E_i E^+_{i,1}
  E_i^2
  -q_{i i}^{-6} q_{p i}^{-3} E^+_{i,1}E_i^3\big)
  \displaybreak[2]\\
  &\,=(\ad E_p)\big(
  E^+_{i,1} E_i^2 E^+_{i,1}
  + q_{p i} E_i E^+_{i,1} E_i E^+_{i,1}
  + q_{p i}^2 E_i^2 (E^+_{i,1})^2\\
  &\,\quad
  -\qnum{3}{q_{i i}} q_{i i}^{-3} q_{p i}^{-1}
  (E^+_{i,1}E_i E^+_{i,1} E_i
  +q_{p i} E_i (E^+_{i,1})^2 E_i
  +q_{p i}^3q_{p p} E_i^2 (E^+_{i,1})^2)
  \\
  &\,\quad
  +\qnum{3}{q_{i i}} q_{i i}^{-5} q_{p i}^{-2}
  ( (E^+_{i,1})^2 E_i^2
  +q_{p i}^2 q_{p p} E_i (E^+_{i,1})^2 E_i
  +q_{p i}^3 q_{p p} E_i E^+_{i,1} E_i E^+_{i,1} )
  \\
  &\,\quad
  -q_{i i}^{-6} q_{p i}^{-3}(
  q_{p i} q_{p p} (E^+_{i,1})^2 E_i^2
  + q_{p i}^2 q_{p p} E^+_{i,1} E_i E^+_{i,1} E_i
  +q_{p i}^3 q_{p p} E^+_{i,1} E_i^2 E^+_{i,1}
  \big)\\
  &\,=(\ad E_p)\big(
  \qnum{2}{q_{i i}} q_{i i}^{-5} q_{p i}^{-2} (E^+_{i,1})^2 E_i^2
  -(1+\qnum{3}{q_{i i}}) q_{i i}^{-3} q_{p i}^{-1}
  E^+_{i,1}E_i E^+_{i,1} E_i
  \\
  &\,\quad +(1-q_{i i}^{-3})E^+_{i,1} E_i^2 E^+_{i,1}
  +(1-q_{i i}^{-3}) E_i (E^+_{i,1})^2 E_i\\
  &\,\quad +
  (1+\qnum{3}{q_{i i}}q_{i i}^{-2}) q_{p i} E_i E^+_{i,1} E_i E^+_{i,1}
  - \qnum{2}{q_{i i}} q_{i i}q_{p i}^2 E_i^2 (E^+_{i,1})^2
  \big)\\
  &\,=
  (q_{i i}+q_{i i}^2-2-q_{i i}-q_{i i}^2+1-q_{i i}^{-3}) (E^+_{i,1})^3 E_i\\
  &\,\quad
  +q_{p i}(q_{i i}+q_{i i}^2 +q_{i i}^3-1+2+q_{i i}^{-1}+q_{i i}^{-2})
  (E^+_{i,1})^2 E_i E^+_{i,1}\\
  &\,\quad
  +q_{p i}^2 (-2 q_{i i}^3-q_{i i}^4-q_{i i}^5 +q_{i i}^3-1-q_{i i}-q_{i i}^2)
  E^+_{i,1} E_i (E^+_{i,1})^2\\
  &\,\quad
  +q_{p i}^3 (q_{i i}^6-q_{i i}^3 +q_{i i} +q_{i i}^2 +2q_{i i}^3 -q_{i i}
  -q_{i i}^2)
  E_i (E^+_{i,1})^3\\
  &\,=-(1+q_{i i}^{-3})  \big(
  (E^+_{i,1})^3 E_i
  -q_{p i} (q_{i i}+q_{i i}^2+q_{i i}^3) (E^+_{i,1})^2 E_i E^+_{i,1}\\
  &\,\quad
  +q_{p i}^2 (q_{i i}^3+q_{i i}^4+q_{i i}^5) E^+_{i,1} E_i (E^+_{i,1})^2
  -q_{p i}^3 q_{i i}^6 E_i (E^+_{i,1})^3
  \big).
\end{align*}
}
Again, the last expression coincides with the one in Eq.~\eqref{eq:TpSerre2}.
This finishes the proof of Eq.~\eqref{eq:TpSerre} and, with it, the proof of
Eq.~\eqref{eq:Serre}.
\end{examp}

\vspace{\baselineskip}

{\small \textbf{Acknowledgement.}
  Many thanks are due to M.~Gra{\~n}a for his interest in the
  subject, for numerous helpful discussions, and for providing a short and
  elegant proof of the crucial Lemma~\ref{le:Coxeterkey}. Further, I'm
  grateful for H.-J.~Schneider for his explanations of the papers
  \cite{p-RadSchn06/67} and \cite{p-RadSchnRep06}. Finally, I want to thank
  H.~Yamane for the discussions on the philosophy of Lusztig isomorphisms
  between Lie superalgebras and for additional valuable comments on
  the manuscript.
}
  

\providecommand{\bysame}{\leavevmode\hbox to3em{\hrulefill}\thinspace}
\providecommand{\MR}{\relax\ifhmode\unskip\space\fi MR }
\providecommand{\MRhref}[2]{%
  \href{http://www.ams.org/mathscinet-getitem?mr=#1}{#2}
}
\providecommand{\href}[2]{#2}

\end{document}